\documentclass[11pt,letterpaper]{article}
\usepackage[T1]{fontenc}
\usepackage{lmodern}
\usepackage[margin=23mm]{geometry}
\usepackage{amsmath,amssymb,amsthm}
\usepackage{microtype}
\usepackage[titles]{tocloft}
\usepackage{parskip}
\usepackage{graphicx}
\usepackage{booktabs}
\usepackage[numbers,sort&compress]{natbib}
\usepackage[hidelinks,hypertexnames=false]{hyperref}
\usepackage{xurl}
\newcommand{\doi}[1]{doi:~\href{https://doi.org/#1}{\nolinkurl{#1}}}
\theoremstyle{definition}
\newtheorem{theorem}{Theorem}
\newtheorem{corollary}[theorem]{Corollary}
\newtheorem{proposition}[theorem]{Proposition}

\pdftrailerid{}
\hypersetup{
  pdftitle={Global minimax risk and acquisition laws for heterogeneous information fusion},
  pdfauthor={Armon Rasooli; Mohammad Sadegh Narimani},
  pdfcreator={},
  pdfproducer={}}

\title{Global minimax risk and acquisition laws for heterogeneous information fusion}
\author{Armon Rasooli\thanks{Corresponding author.
Email: \href{mailto:armonrasooli@gmail.com}{armonrasooli@gmail.com}.
ORCID: \href{https://orcid.org/0009-0002-6583-7090}{0009-0002-6583-7090}.}
\and Mohammad Sadegh Narimani\thanks{ORCID:
\href{https://orcid.org/0009-0004-8648-2238}{0009-0004-8648-2238}.}}
\date{}

\begin{document}

\maketitle

\begin{abstract}

We prove an all-allocation global target-risk theorem for independent
Gaussian sources that share a scalar nuisance and an unknown contact
coordinate. For a fixed immersed nuisance curve with finitely many
multiple fibres and pairwise nonparallel branch tangents, squared target
risk is comparable to a primary estimation floor plus a
target-gap-weighted Gaussian discrimination profile. Independent
localization,
finite branch selection and target-class refitting give an estimator
attaining this comparison across every nonnegative integer allocation. A
full-box specialization has positive-definite primary Fisher information
everywhere and globally identifies its target, yet models with identical
derivatives of every order along a critical hypersurface have different
polynomial risk exponents. We construct a finite certified estimator and
quantify the numerical accuracy needed to preserve rare branch decisions
and acquisition windows. A polynomial tangency family both demonstrates
the geometric boundary and quantifies its repair: known auxiliary gain,
source contact order and target vanishing order determine a risk law
uniform through zero gain. For a shared spherical direction, a
separately proved composite-testing result transfers through an unknown
contact coordinate using only counted observations. These theorems
distinguish the resources needed for local estimation, discrimination
between parameter regions and nuisance alignment, and yield acquisition
thresholds under explicit scalar-observation costs.

\end{abstract}


\clearpage
\tableofcontents
\clearpage

\section{Introduction}
\label{sec:main-1}

A fusion system often receives repeated observations from several known
response maps. One source may provide accurate local measurements while
another distinguishes parameter configurations that produce nearly
identical primary responses. The central design question is then
quantitative:
how many observations of each source are needed to estimate the desired
target, and which features of the response maps determine that
requirement? Answering this question requires preserving the common
parameter across the sources. Optimizing each source over an independent
nuisance parameter generally describes a different experiment.

Two familiar diagnostics are insufficient by themselves. Local
differential rank determines whether small displacements within a
parameter neighborhood change the noiseless response to first order.
Exact fibres, defined as sets of parameters with identical responses,
determine noiseless identification. Neither diagnostic measures the
relative rates at which source differences and target differences vanish
between distinct local neighborhoods. Those rates matter statistically
even when every exact fibre is target-constant and every local Fisher
information matrix is positive definite.

This paper develops a global target-risk analysis for that situation.
Its principal positive theorem concerns a fixed immersed scalar nuisance
curve with finitely many multiple fibres and pairwise nonparallel
tangent directions at distinct preimages. Additional sources measure a
common scalar amplitude through known functions of an independently
localized coordinate. The theorem derives a uniform risk comparison for
every allocation of nonnegative integer observation counts. The
comparison retains the exact Gaussian-tail scale and the target gap
associated with each actual competing pair. Its upper bound requires a
global estimator; a collection of binary lower bounds alone would not
establish the result.

A specialization places the mechanism on a full three-dimensional box
with an ordinary coordinate target. Its core map is everywhere
immersive, its only nuisance aliases have equal targets, and its source
maps can agree to every differential order along the entire critical
hypersurface.
Nevertheless,
the core-only squared-risk exponent ranges through the interval
\((0,1)\). This separation persists within each fixed nonanalytic Gevrey
class. The example isolates a global statistical obstruction that is
invisible to local rank and the complete critical jet, despite exact
target identification.

The construction also supports finite certified inference and explicit
acquisition results. The nuisance fit reduces to a quintic stationary
equation on a compact interval. Certified root isolation,
absolute-accuracy function evaluation, and controlled score arithmetic
retain the statistical comparison even near an exponentially sensitive
acquisition threshold. Under a declared total cost, the minimum resolver
allocation needed to attain a risk order differs from the optimizer of a
specified risk proxy. These are separate optimization questions.

Finally, finite branch multiplicity is a substantive restriction. With a
continuous spherical direction nuisance, minimum cross-target distance
does not determine the Gaussian prefactor. We therefore examine a
separately specified calibration experiment. An independently observed
noisy direction reference changes composite discrimination while leaving
the least cross-label distance unchanged. A counted attenuation
construction transfers its fixed-radius minimax law through an unknown
contact coordinate, without assuming that an estimated-radius
substitution is statistically exact.

The contribution has four parts: an all-allocation risk law derived from
finite branch geometry; constructive inference with sufficient numerical
precision; flat and tangential boundaries, including a gain-uniform
auxiliary repair law; and a shared-direction extension that isolates
composite nuisance alignment. The tangency repair is a quantitative
continuation of the boundary example: its explicit shape-and-sign
estimator works uniformly even when the known gain vanishes. Together,
these results turn branch geometry into explicit acquisition laws and
identify the additional information needed when tangency or a continuous
nuisance changes the discrimination problem.

\section{Related work and contribution}
\label{sec:main-2}

Testing lower bounds and deterministic moduli provide a foundational
connection between inverse geometry and statistical risk.
Optimal-recovery theory characterizes estimation through geometric
continuity properties \cite{Donoho1994}, and dual formulations of Le Cam's method
relate testing bounds to matching estimators \cite{PolyanskiyWu2026}. Gaussian testing
and recovery over convex sets, including unions of convex components and
independent product
observations,
give powerful tools for heterogeneous experiments \cite{GoldenshlugerJuditskyNemirovski2015,Guigues2020,Juditsky2020}. The
present work addresses a complementary setting: fixed nonlinear mean
images with retained nuisance ambiguities, target coalescence and
unrestricted relative source counts.

The main advance is a matching global minimax theorem derived directly
from finite-fibre geometry. It identifies the complete target-weighted
Gaussian discrimination profile, proves its attainability for every
nonnegative integer allocation and separates it from an independent
primary-source estimation floor. This conclusion requires control of all
approximately indistinguishable parameter pairs, including parameters
near different branches. The target-value-class refit is essential when
distinct nuisance branches have the same limiting target. These features
go beyond a collection of binary lower bounds or a local information
calculation.

Nonconvexity matters in this comparison. Polyanskiy and Wu give an
example in which a small-radius modulus fails to characterize quadratic
risk \cite[Example~4]{PolyanskiyWu2026}. Our mean images are generally nonconvex;
convexifying them can remove the branch separation that controls their
risk. The theorem instead obtains a quantitative all-pairs alternative
from the fixed curve and uses it to preserve the Gaussian-tail scale in
a global upper bound. Its count quantifier includes allocations outside
every fixed polynomial relationship between source budgets.

Geometric fusion has a substantial foundation. Robinson and Ghrist use
multijet transversality to study ambiguity dimensions and heterogeneous
localization,
with embedding results that also yield uniform inverse-neighborhood
continuity \cite{RobinsonGhrist2012}. Here the multiple fibres are retained, and the goal
is a quantitative risk and acquisition law for a fixed map.
Target-displacement-weighted ambiguity terms also occur in approximate
maximum-likelihood error analyses \cite{Mallat2014}. Our result supplies a
matching minimax
characterization,
including a complete estimator
construction,
under explicit geometric hypotheses.

Nonlinear inverse problems already exhibit response-dependent rates
\cite{RaySchmidtHieber2016}, and smooth flat functions have a developed ultradifferentiable
theory \cite{Sanz2014}. The full-box examples sharpen their statistical
implications:
source maps with the same critical derivatives of every order, exact
fibres, target and local rank can have different global minimax
exponents. This separation persists within each fixed nonanalytic Gevrey
class. The tangency theorem then gives a quantitative repair law
governed jointly by source contact order, target vanishing order and
known auxiliary gain, uniformly down to zero gain.

The constructive analysis builds on polynomial root isolation \cite{Sagraloff2016},
certified polynomial optimization \cite{KlepPovhVolcic2018} and classical derivative
control for nonnegative smooth functions \cite{Glaeser1963}. Its contribution is
the statistical accuracy theorem: finite nuisance fitting and function
evaluation preserve rare branch decisions and the associated acquisition
window. The proof connects numerical precision to the risk scale at
which it is consequential.

Classical phase-reference detection and passive-sensing models provide
antecedents for inference with a noisy reference \cite{Lindsey1966,Cui2015}. The
calibration results here concern a fixed-norm shared direction, an
invariant minimax equalizer and a uniform transfer through an unknown
contact coordinate. All localization and attenuation observations are
counted. This gives a precise source-budget law for nuisance alignment
and explains why the minimum cross-target distance, even together with
orbit dimension, can miss an unbounded discrimination penalty.

\section{Experiment and a global finite-fibre theorem}
\label{sec:main-3}

All observations in this paper are real Gaussian blocks with independent
unit-variance coordinates. The response coordinates are dimensionless
under the declared normalization. Counts refer to independent blocks; a
source with several coordinates contributes several scalar measurements
per block. For source means \(G_j(\theta)\) and counts \(n_j\), the
Kullback--Leibler divergence between two common parameters is 
\begin{equation}
D_{\mathrm{KL}}(\theta\Vert\theta')
=\frac12\sum_j n_j\|G_j(\theta)-G_j(\theta')\|^2. \tag{1}
\end{equation}
 The norm is Euclidean. The notation \(f\asymp g\) means that their
ratio lies between positive constants over the stated range; \(f\sim g\)
means that the ratio tends to one. Constants below may depend on the
fixed response maps and domains, but not on observation counts.

Let \(I,J\subset\mathbb R\) be compact intervals with nonempty
interiors. Let \(C:I\to\mathbb R^D\) be a fixed \(C^2\) curve, defined
on a neighborhood of \(I\), satisfying three conditions. First,
\(C'(b)\ne0\) everywhere. Second, there are only finitely many image
points with multiple preimages; write their complete fibres as 
\[
B_h=\{\beta_{h,1},\ldots,\beta_{h,M_h}\},\qquad
1\le h\le H_0,\quad M_h\ge2.
\]
 Every remaining fibre is a singleton. Third, for distinct preimages
within each fibre, the vectors \(C'(\beta_{h,i})\) and
\(C'(\beta_{h,l})\) are linearly independent. We call this the pairwise
nonparallel branch-tangent condition: the matrix with columns
\(C'(\beta_{h,i})\) and \(-C'(\beta_{h,l})\) has rank two. This
condition is used in every ambient dimension \(D\); it does not require
the two tangent lines to span the ambient space. Multiple fibres with
more than two points and fibres at interval endpoints are allowed.

Let \(q:I\to\mathbb R\) and \(\chi:J\to\mathbb R\) be fixed \(C^1\)
functions. The parameter and target are 
\[
\theta=(r,b,t)\in J\times I\times[-T_*,T_*],\qquad T(\theta)=t,
\]
 where \(T_*>0\) satisfies 
\begin{equation}
\tfrac12\operatorname{diam}(q(I))\|\chi\|_\infty\le T_*.
\tag{2}
\label{eq:2}
\end{equation}
 The primary source, observed \(n=n_0\) times, has mean 
\begin{equation}
F(r,b,t)=\bigl(r,t-q(b)\chi(r),C(b),b a_0(r)\bigr). \tag{3}
\end{equation}
 Each further source \(j=1,\ldots,s\) has either mean \((r,b a_j(r))\)
or scalar mean \(b a_j(r)\), observed \(n_j\) times. Let \(\mathcal L\)
be the indices of the sources containing the radius coordinate \(r\),
including index zero. For \(j\in\mathcal L\), assume that \(a_j\) has a
fixed nonnegative \(C^2\) extension to a neighborhood of \(J\). For
\(j\notin\mathcal L\), assume a \(C^2\) extension bounded below by a
positive constant. The latter sources need not provide additional radius
measurements. A known zero primary amplitude can be omitted together
with its ancillary noise coordinate.

Define the available radius count and detector information by 
\[
N_L=\sum_{j\in\mathcal L}n_j,\qquad
H_{\boldsymbol n}(r)=\sum_{j=0}^{s}n_j a_j(r)^2.
\]
 For each pair in a multiple fibre, set 
\[
d_{h,il}=\tfrac12|\beta_{h,i}-\beta_{h,l}|,\qquad
w_{h,il}=\tfrac14|q(\beta_{h,i})-q(\beta_{h,l})|^2.
\]
 The resulting target-weighted discrimination profile is 
\begin{equation}
\mathcal S(\boldsymbol n)=
\max_{h,i<l}\sup_{r\in J}
 w_{h,il}\chi(r)^2
 \Phi\{-d_{h,il}\sqrt{H_{\boldsymbol n}(r)}\}. \tag{4}
\label{eq:4}
\end{equation}
 An empty maximum is zero. Every term uses the source and target gaps
of one actual common-parameter pair. The minimax risk is 
\[
\mathcal R(\boldsymbol n)=\inf_{\widehat t}\sup_{(r,b,t)\in J\times I\times[-T_*,T_*]}
\mathbb E_{r,b,t}(\widehat t-t)^2,
\]
 where the infimum is over all measurable estimators based on the
observed blocks.

\begin{theorem}[finite-fibre global target risk]
\label{thm:global-risk}
Under these hypotheses, the minimax squared target
risk satisfies 
\begin{equation}
\mathcal R(\boldsymbol n)\asymp(1+n)^{-1}+\mathcal S(\boldsymbol n) \tag{5}
\label{eq:5}
\end{equation}
 for every nonnegative integer allocation. The comparison need not be
uniform as tangent angles vanish or distinct endpoint target values
merge.
\end{theorem}

The hypotheses are properties of the response maps. They do not assume
the existence of an estimator with the desired risk. The primary map is
automatically everywhere immersive: its radius and shifted-target
coordinates isolate two derivative directions, while \(C'\ne0\) supplies
the third.
Nevertheless,
its global risk can be controlled by the second term of \eqref{eq:5}.

The theorem makes the purpose of each source explicit. The core supplies
the shifted target coordinate and the curve fit, so unlimited auxiliary
replication cannot remove its independent variance floor. The amplitude
coordinates decide between distinct curve branches. Their statistical
role depends on the corresponding target difference: if all endpoint
values of \(q\) agree within each multiple fibre, then \(\mathcal S=0\)
and the regular order holds despite nuisance nonidentification. If some
endpoint target values differ, the same nuisance ambiguity can remain
consequential until sufficient amplitude information is acquired. The
radius-dependent factor \(\chi\) determines how much loss a branch error
incurs.

More precisely, under the stated box condition, the primary source
identifies the target exactly when 
\[
\chi(r)\{q(\beta_{h,i})-q(\beta_{h,l})\}=0
\quad\text{for every listed pair whenever }a_0(r)=0.
\]
 This follows by equating its observed coordinates; the lower-bound
alternatives below establish necessity. Exact identification is
therefore a zero-set condition, whereas \eqref{eq:4} retains quantitative
behavior around that zero set. The distinction is useful when choosing a
supplementary source: a noiseless identifying coordinate may still
receive too few repetitions to suppress the target-weighted ambiguity at
the desired risk scale.

\subsection{Proof mechanism}

The complete proof and measurable conventions are in Appendix \ref{app:s1}. We
give the main geometric and statistical steps because each is necessary
for the global statement.

The curve assumptions imply a quantitative all-pairs alternative. There
are fixed \(L,\varepsilon_0>0\) such that
\(\|C(b)-C(c)\|\le\varepsilon\le\varepsilon_0\) implies either 
\begin{equation}
|b-c|\le L\varepsilon, \tag{6}
\label{eq:6}
\end{equation}
 or proximity, at the same scale, to two preimages of one listed
fibre: 
\begin{equation}
|b-\beta_{h,i}|+|c-\beta_{h,l}|\le L\varepsilon. \tag{7}
\label{eq:7}
\end{equation}
 Near the diagonal, \eqref{eq:6} follows from the nonzero derivative and
uniform Taylor control. At a distinct equality pair, the derivative of
\((b,c)\mapsto C(b)-C(c)\) has rank two. Its least singular value and a
quadratic Taylor remainder prove \eqref{eq:7}. Outside these finitely many
neighborhoods and the diagonal, compactness gives a positive image
separation. Thus local differential rank is supplemented by an explicit
restriction on interactions between distinct branches.

For the ambiguity lower bound, fix \(r,h,i,l\) and write
\(\bar q=[q(\beta_{h,i})+q(\beta_{h,l})]/2\). The alternatives 
\[
(r,\beta_{h,i},[q(\beta_{h,i})-\bar q]\chi(r)),\quad
(r,\beta_{h,l},[q(\beta_{h,l})-\bar q]\chi(r))
\]
 belong to the full parameter box by \eqref{eq:2}. Their shifted target and
curve coordinates agree, while their squared complete Gaussian mean
distance is \(4d_{h,il}^2H_{\boldsymbol n}(r)\). Their exact equal-prior
testing error is the normal tail in \eqref{eq:4}. Turning an estimator into a
nearest-target test gives the corresponding lower bound. Separately,
varying only \(t\) by order \((1+n)^{-1/2}\) leaves every auxiliary
source invariant and gives the independent core floor.

For the upper bound, pool the available radius coordinates and clip
their mean to obtain \(\hat r\), with mean squared error at most
\(N_L^{-1}\). Fit the curve to the primary curve-coordinate mean. If the
fit is away from every multiple preimage, \eqref{eq:6} bounds its parameter error
by the actual Gaussian residual. Otherwise \eqref{eq:7} localizes the true
nuisance near one finite fibre. Conditional on the independent radius
observation, collect the whitened amplitude means into 
\[
Y\sim N(bv,I),\qquad v=(\sqrt{n_j}a_j(r))_{n_j>0},
\]
 and let \(\hat v\) be obtained by evaluating the amplitudes at
\(\hat r\). Select among the finite candidate means
\(\beta_{h,l}\hat v\) by Euclidean distance.

A selected nuisance branch need not have a different endpoint target
value. Therefore the final curve refit is performed over the union of
all nearby branches sharing the selected value of \(q(\beta_{h,l})\). If
the target class is correct, \eqref{eq:6}--\eqref{eq:7} bound its target error by the
actual curve noise. This avoids assigning the larger screening error
\(O(\sqrt{\log n/n})\) to the final estimate. If the class is incorrect,
its loss is bounded by the corresponding fixed endpoint target gap,
multiplied by \(\chi(r)^2\).

The remaining issue is precision of the noisy score direction. On a
radius event of probability \(1-O(n^{-8})\), nonnegative-extension
derivative control gives 
\begin{equation}
E:=\|\hat v-v\|^2
\le C\left\{
\frac{\log n}{\sqrt{N_L}}\sqrt H+
\frac{\log^2 n}{N_L}+\frac{H\log n}{N_L}\right\}, \tag{8}
\label{eq:8}
\end{equation}
 where \(H=\|v\|^2\). For a competitor pair with half-separation
\(d>0\), its Gaussian bisector has standardized signed distance at least

\begin{equation}
\mu\ge d\sqrt H-C\sqrt{\log n/n}\sqrt H-C\sqrt E. \tag{9}
\label{eq:9}
\end{equation}
 The last term also accounts for a nonzero midpoint of an asymmetric
fibre pair. The projection inequality 
\[
\frac{(\hat v\cdot v)^2}{\|\hat v\|^2}\ge H-E
\]
 is responsible for preserving the discrimination scale.

When \(H\le K\log n\), \eqref{eq:8}--\eqref{eq:9} change the squared tail argument by at
most \(o(1)\), uniformly over all other counts. The Mills comparison 
\begin{equation}
\Phi(-\sqrt x)\asymp(1+x)^{-1/2}e^{-x/2},\qquad x\ge0,
\tag{10}
\end{equation}
 then preserves its bounded-factor Gaussian tail. For \(H>K\log n\), a
sufficiently large fixed \(K\) makes errors negligible compared with
\(n^{-1}\). A finite union over competing pairs completes the upper
bound. Zero-information scores, bounded primary counts, and endpoint
charts are treated explicitly in Appendix \ref{app:s1}.

\section{Identical local data with different global rates}
\label{sec:main-4}

Consider the full box 
\[
K=[-a,a]\times[-2,2]\times[-2,2],\qquad 0<a<\tfrac12,
\]
 and the means 
\begin{equation}
\begin{split}
F(r,b,t)&=(r,t-b\chi(r),b^2,b^3-b,b\psi(r)),\\
G(r,b,t)&=(r,b(1-r)).
\end{split} \tag{11}
\label{eq:11}
\end{equation}
 The fixed functions are even and \(C^2\) on a neighborhood of the
radius interval. Assume \(\psi\ge0\) there, \(\psi(0)=\psi'(0)=0\),
\(\psi(r)>0\) for \(r\ne0\), \(0\le\chi\le1\), and \(\chi(0)=0\). There
are \(n\) primary and \(m\) resolver blocks, giving \(5n+2m\) scalar
measurements.

\begin{corollary}[full-box target-risk law]
\label{thm:full-box-risk}
The model \eqref{eq:11}
satisfies 
\begin{equation}
\mathcal R_{\psi,\chi}(n,m)\asymp
(1+n)^{-1}+\sup_{0\le r\le a}\chi(r)^2
\Phi\{-\sqrt{n\psi(r)^2+m(1-r)^2}\}. \tag{12}
\label{eq:12}
\end{equation}
 The core is everywhere immersive and globally identifies \(t\). Its
only nontrivial exact fibres are 
\begin{equation}
\{(0,1,t),(0,-1,t)\},\qquad -2\le t\le2. \tag{13}
\end{equation}

\end{corollary}

Indeed, \(C(b)=(b^2,b^3-b)\) has the sole multiple fibre \(\{-1,1\}\),
whose tangent vectors \((-2,2)\) and \((2,2)\) are independent. Theorem~\ref{thm:global-risk} applies with half-separation and squared target half-gap both equal to
one. Evenness and \((1-r)^2\le(1+r)^2\) for \(r\ge0\) give the exact
reduction of the supremum to nonnegative radii. The full derivative rank
follows because \((2b,3b^2-1)\) never vanishes. Equality of core
observations first fixes \(r\), then forces either identical nuisances
or the pair \(b=\pm1,r=0\); the targets agree because \(\chi(0)=0\). For
these unit-covariance real Gaussian blocks, the full one-block Fisher
matrix is \(I_F(\theta)=DF(\theta)^TDF(\theta)\). Full derivative rank
and compactness make its least eigenvalue strictly positive for each
fixed member of the model family. The eigenvalue bound is for fixed
response maps, without a uniform claim over varying function families.
This statement concerns all local parameter coordinates; it does not
assert a global Lipschitz inverse for the target.

\subsection{Smooth and Gevrey separation}

For \(r\ne0\), define 
\begin{equation}
\psi_*(r)=\exp\{-\tfrac12\exp(1/r^2)\},\qquad
\chi_\alpha(r)=\psi_*(r)^\alpha,\quad 0<\alpha<1,
\tag{14}
\label{eq:14}
\end{equation}
 and give both functions value zero at the origin.

\begin{theorem}[critical-jet insufficiency and a coupled window]
\label{thm:critical-jets}
Throughout the family \eqref{eq:14}, the
target, box, complete labelled exact fibre relations, and
derivative-rank profile agree. All derivatives of every order of both
source maps agree at every point of \(\{r=0\}\). The one-block core
Fisher matrix is positive definite everywhere, with a positive minimum
eigenvalue for each fixed model; uniformity over \(\alpha\) is not
asserted.
Nevertheless, 
\begin{equation}
\mathcal R_\alpha(n,0)\asymp n^{-\alpha}. \tag{15}
\label{eq:15}
\end{equation}
 Let \(\ell=\log n\) and \(r_\ell=(\log\ell)^{-1/2}\), once
\(r_\ell<a\). Uniformly over \(0\le m\le C_0\ell\), for every fixed
\(C_0\), 
\begin{equation}
\mathcal R_\alpha(n,m)\asymp n^{-1}
+n^{-\alpha}(1+m)^{-1/2}
\exp\{-\tfrac12m(1-r_\ell)^2\}. \tag{16}
\label{eq:16}
\end{equation}
 Consequently the boundary for regular \(n^{-1}\) risk is 
\begin{equation}
m=\frac{2(1-\alpha)\ell-\log\ell}{(1-r_\ell)^2}+O(1). \tag{17}
\label{eq:17}
\end{equation}
 The same exponent separation \eqref{eq:15} is possible inside every prescribed
Gevrey class of order strictly greater than one.
\end{theorem}

A formal jet at a point is the collection of all derivatives there.
Flatness of \eqref{eq:14} follows because every differentiated term is a
polynomial in \(r^{-1}\) and \(e^{1/r^2}\), multiplied by the dominating
double exponential. Changing \(\alpha\) changes none of these limiting
derivatives. This agreement holds along the critical
hypersurface,
not at every point of the domain.

For the rate calculation, put \(z=n\psi_*(r)^2\) in \eqref{eq:12}. At \(m=0\),
its ambiguity term becomes 
\[
n^{-\alpha}\sup_{0\le z\le n\psi_*(a)^2}z^\alpha\Phi(-\sqrt z),
\]
 whose supremum tends to a finite positive number. The regular floor
is smaller because \(\alpha<1\). With \(m=O(\ell)\), the consequential
values of \(z\) correspond to \(r=(\log(\ell-\log z))^{-1/2}\).
Uniformly on a sufficiently broad central range, replacing this radius
by \(r_\ell\) changes the resolver information by \(o(1)\); the
remaining tails are negligible. This proves \eqref{eq:16}. More precisely, when
also \(m\to\infty\), the optimized ambiguity profile divided by
\(n^{-\alpha}\Phi[-\sqrt m(1-r_\ell)]\) tends to \((2\alpha/e)^\alpha\).
This is an exact profile constant, not an exact composite minimax
constant. Appendix \ref{app:s2} gives the complete uniform argument.

A Gevrey class of order \(s>1\) consists locally of smooth functions
whose derivatives obey bounds \(CA^k(k!)^s\). Choose \(p\ge1/(s-1)\) and
use \(\psi(r)=e^{-|r|^{-p}}\), \(\chi(r)=e^{-\alpha|r|^{-p}}\), with
zero values at the origin. A Cauchy estimate on a complex disk of radius
proportional to \(|r|\) proves 
\[
\sup_r|\partial_r^k e^{-c|r|^{-p}}|
\le A_c^k(k!)^{1+1/p}.
\]
 The same change of variable proves \eqref{eq:15}. In a quasianalytic class,
equal Taylor germs coincide by definition, excluding this particular
construction. No converse classification of global risk in all
quasianalytic models follows.

\subsection{Why target and source corrections interact}

The factor \((1-r_\ell)^{-2}\) in \eqref{eq:17} converts required discrimination
into resolver count after target coalescence has reduced the requirement
to \((1-\alpha)\ell\). Applying it to the fixed-target requirement and
then subtracting \(2\alpha\ell\) raw observations oversamples by 
\begin{equation}
2\alpha\ell\{(1-r_\ell)^{-2}-1\}\sim4\alpha\ell r_\ell\to\infty. \tag{18}
\end{equation}
 The discrepancy exceeds a bounded acquisition window. A slower target
gap makes the same point below the leading power scale: with
\(\chi(r)=e^{-1/r^2}\), the core risk is \(\asymp(\log n)^{-2}\), and
the window numerator is \(2\ell-5\log\ell\). Subtracting the target
saving as raw counts then misses a term asymptotic to
\(8\sqrt{\log\ell}\). These statements concern fixed functions at
asymptotically large counts; the radius-in-domain condition can require
extremely large budgets.

\section{A polynomial boundary of finite-fibre transfer}
\label{sec:main-5}

The nonparallel-pair assumption in Theorem~\ref{thm:global-risk} cannot be replaced by
nonzero curve derivative and finite fibres alone.

\begin{theorem}[tangential ambiguity with a target-constant fibre]
\label{thm:tangency-boundary}
Fix an integer \(p\ge2\), \((b,t)\in[-2,2]^2\), and observe independent
unit-covariance Gaussian blocks with mean 
\[
F_p(b,t)=\bigl(t-q_p(b),C_p(b)\bigr),
\]
 
\begin{equation}
C_p(b)=(b^2-1,b(b^2-1)^p),\qquad q_p(b)=b(b^2-1).
\tag{19}
\end{equation}
 The map is everywhere immersive and globally identifies \(t\). Its
only distinct nuisance aliases are \(b=\pm1\), where \(q_p\) agrees.
Nevertheless, 
\begin{equation}
\mathcal R_p(n)\asymp(1+n)^{-1/p}. \tag{20}
\label{eq:20}
\end{equation}

\end{theorem}

The curve derivative has first coordinate \(2b\); at \(b=0\), its second
coordinate is \((-1)^p\ne0\). Its tangents at \(\pm1\) are collinear.
For small \(z>0\), the actual alternatives 
\[
b_\pm=\pm\sqrt{1+z},\qquad t_\pm=\pm z\sqrt{1+z}
\]
 have identical first two mean coordinates and last-coordinate
difference \(2z^p\sqrt{1+z}\). Taking \(z\asymp n^{-1/(2p)}\) gives
bounded testing difficulty and squared target separation of order
\(n^{-1/p}\).

For the upper bound, the target function on the compact curve image
obeys 
\begin{equation}
|q_p(b)-q_p(c)|\le C\|C_p(b)-C_p(c)\|^{1/p}. \tag{21}
\end{equation}
 Only neighborhoods of the distinct fibre require verification. There
write \(x=b^2-1,y=c^2-1\), with \(b\) positive and \(c\) negative. If
\(xy<0\), the first-coordinate difference \(|x-y|\) controls the target
difference. If \(xy\ge0\), the absolute second-coordinate difference is
\(|b||x|^p+|c||y|^p\), for either parity of \(p\), and controls the
target difference to power \(p\). Local immersion and compactness cover
the remaining pairs. Minimum-distance fitting and Gaussian moments now
prove \eqref{eq:20}; details are in Appendix \ref{app:s1}.

\begin{figure}
\centering
\includegraphics[width=\linewidth]{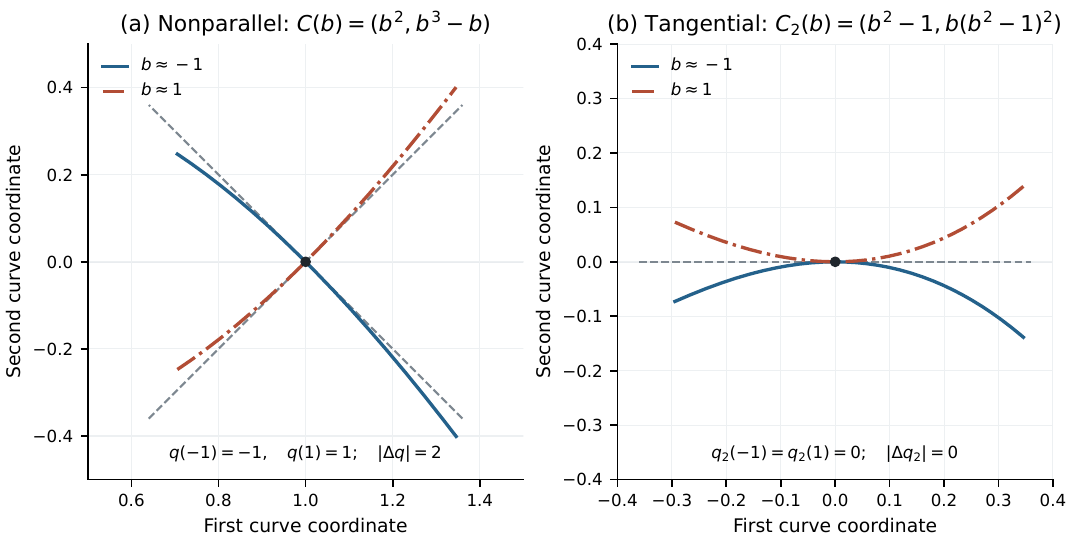}
\caption{Nonparallel and tangential finite-fibre geometry. The curves
are deterministic evaluations of (a) \(C(b)=(b^2,b^3-b)\) and (b)
\(C_2(b)=(b^2-1,b(b^2-1)^2)\) near \(b=\pm1\). Dashed lines show their
tangent directions; endpoint target-factor gaps are two and zero,
respectively. The second geometry illustrates why a zero exact-fibre
target gap does not imply regular target risk without nonparallel branch
separation. No observations or fitted curves are simulated.}
\end{figure}

Thus zero target gaps at the exact fibre do not remove the statistical
contribution of nearby tangential configurations. This polynomial
example is consistent with contact-order analyses: it identifies
precisely the geometric information lost by retaining only exact
endpoint target values.

\subsection{Known-gain repair of a tangential fibre}

The fixed-map constants in Theorem~\ref{thm:global-risk} need not survive collapsing branch
angles. One explicit polynomial family admits a direct uniform analysis
through that limit. Write 
\[
z(b)=b^2-1,\qquad f_j(b)=b z(b)^j.
\]
 Fix integers \(p\ge2\), \(1\le q<p\), \(k_{\mathrm{tan}}\ge1\),
constants \(T,E_*>0\), and a known gain \(\varepsilon\in[0,E_*]\). On
the full box \((b,t)\in[-2,2]\times[-T,T]\), observe independent blocks
and scalars 
\begin{equation}
\begin{aligned}
X_i&\sim N_3((t-f_q(b),z(b),f_p(b)),I_3),&&1\le i\le n,\\
Y_l&\sim N(\varepsilon f_{k_{\mathrm{tan}}}(b),1),&&1\le l\le m.
\end{aligned} \tag{22}
\label{eq:22}
\end{equation}
 The target is \(t\), the nuisance \(b\) is common to both sources,
and the scalar-observation cost is \(3n+m\). The contact order
\(k_{\mathrm{tan}}\) is distinct from the calibration information \(k\)
in Section 8.

\begin{theorem}[gain-uniform polynomial tangency repair]
\label{thm:tangency-repair}
For squared target loss in \eqref{eq:22}, 
\begin{equation}
\mathcal R_\varepsilon(n,m)\asymp
\frac1{1+n}+\min\{(1+n)^{-q/p},
(1+m\varepsilon^2)^{-q/k_{\mathrm{tan}}}\}. \tag{23}
\label{eq:23}
\end{equation}
 The constants depend only on the fixed degrees and box bounds, and
are uniform over every \(n,m\in\mathbb N_0\) and every known
\(\varepsilon\in[0,E_*]\). At \(q=1\), \(T=2\) and \(m=0\), this
recovers Theorem 4.
\end{theorem}

For the lower bound, take \(b_\pm=\pm\sqrt{1+z}\),
\(t_\pm=\pm\sqrt{1+z}z^q\) with small positive \(z\). Their first two
primary means coincide, their squared target separation is
\(4(1+z)z^{2q}\), and their exact divergence is 
\begin{equation}
D_{\mathrm{KL}}=2(1+z)\{nz^{2p}+m\varepsilon^2z^{2k_{\mathrm{tan}}}\}. \tag{24}
\end{equation}
 Choosing the smaller of the two bounded-divergence contact scales
gives the second term in \eqref{eq:23}. A separate pair with \(b=1\) and varying
\(t\) supplies the primary floor independently of the auxiliary
allocation.

The upper estimator first clips the observed shape mean to \([-1,3]\).
Away from zero shape, division of the last primary mean by a
bounded-away-from-zero power of that estimate reconstructs the target
correction with mean squared error \(O(n^{-1})\). Near zero shape, a
single Gaussian sign decision supplies the branch sign; the target
correction has magnitude proportional to \(|z|^q\). The known counts and
gain select the better primary or auxiliary sign detector, and
shape-sign mistakes are absorbed by the vanishing target magnitude.
Final clipping preserves bounded loss. Appendix \ref{app:s7} specifies every
branch and tie and proves the bound on the full box without nonlinear
optimization.

Along \(n\to\infty\), the matching bounds give two different
requirements: 
\begin{equation}
\begin{aligned}
\mathcal R_\varepsilon(n,m)=o(n^{-q/p})
&\quad\Longleftrightarrow\quad m\varepsilon^2/n^{k_{\mathrm{tan}}/p}\to\infty,\\
\mathcal R_\varepsilon(n,m)=O(n^{-1})
&\quad\Longleftrightarrow\quad m\varepsilon^2=\Omega(n^{k_{\mathrm{tan}}/q}).
\end{aligned} \tag{25}
\end{equation}
 For fixed positive gain and total cost \(B=3n+m\), the optimal risk
order is \(B^{-\min\{1,\max(q/p,q/k_{\mathrm{tan}})\}}\). When
\(k_{\mathrm{tan}}<q\), a primary count proportional to \(B\) and an
auxiliary count of order \(B^{k_{\mathrm{tan}}/q}\) attain regular
order. No optimal allocation fraction or exact minimax constant is
asserted.

For \(k_{\mathrm{tan}}=1\), the augmented tangent pair is independent at
every positive gain, with squared cross-product norm
\(64\varepsilon^2\); \eqref{eq:23} is uniform through its collapse. For larger
contact order, the augmented branches can remain tangent. Thus this is a
particular polynomial repair law, not an arbitrary-contact extension of
Theorem 1. Unknown gain is a different inference problem. The sufficient
numerical tolerances in \ref{app:s7} preserve the polynomial risk order; they do
not preserve every rare Gaussian tail in relative error, which is the
distinct precision issue considered next.

\section{Finite certified inference}
\label{sec:main-6}

The measurable estimator of Theorem~\ref{thm:global-risk} allows general known smooth maps.
For \eqref{eq:11}, its nuisance geometry supports a finite certified construction
with explicit tolerances.

\begin{theorem}[certified computation preserves the risk law]
\label{thm:certified-inference}
Suppose the sufficient means are supplied through rational enclosures at
the precisions below, the endpoint \(a\) is effectively represented, and
the fixed functions admit deterministic measurable absolute-accuracy
evaluation. The represented radius must use only the first-coordinate
observations. Its clipping, endpoint
approximation,
and weight oracle must then use only that represented radius, known
counts, and fixed functions, without inspecting the amplitude-score
coordinates. Under this independence-preserving input contract, a finite
estimator attains \eqref{eq:12}. For rational endpoints and represented rational
parameters, the explicit exponential families in Section 4 admit such
evaluation by finite rational enclosures.
\end{theorem}

For observed curve means \(U,V\), minimize 
\[
P_{U,V}(b)=(b^2-U)^2+(b^3-b-V)^2,
\]
 on \([-2,2]\). Its stationary equation is 
\begin{equation}
\tfrac12P'_{U,V}(b)
=3b^5-2b^3-3Vb^2+(1-2U)b+V=0. \tag{26}
\end{equation}
 The leading coefficient is always three. Every global minimum is an
endpoint or one of at most five real stationary roots. Repeated roots
must be retained even when the derivative does not change sign.
Square-free reduction and Sturm root counts isolate every distinct root;
evaluating endpoint and root-interval candidates certifies a global
objective tolerance. A local optimizer alone does not provide this
guarantee.

Let \(N=n+m\), and let \(X,U,V\) denote the shifted-target and curve
means, \(R\) the pooled radius mean, and \(Y,B\) the amplitude means.
The following absolute tolerances suffice for \(n\ge2\):

\begin{center}
\small
\begin{tabular}{@{}lr@{}}
\toprule
Computed quantity & Allowed error \\
\midrule
Represented \(X,U,V\) & \((64n)^{-1}\) each \\
Represented \(R,Y,B\) & \((64N)^{-1}\) each \\
Nuisance objective excess & \((64n^2)^{-1}\) \\
Evaluation of \(\psi\) & \((64N)^{-1}\) \\
Evaluation of \(\chi\) and magnitude square root & \((64n)^{-1}\)
each \\
Additional normalized score arithmetic & \(N^{-1/2}\) \\
\bottomrule
\end{tabular}
\end{center}

The rational-arithmetic construction evaluates its score exactly, so the
last allowance is optional. It clips function outputs to their
prescribed ranges and uses a deterministic convention when the computed
weight norm is too small to require discrimination. All root intervals,
objective comparisons, clipping and ties are specified in Appendix \ref{app:s3}.

Two bounds explain the statistical transfer. A certified objective
excess of order \(n^{-2}\) gives a curve residual bounded by twice the
actual Gaussian curve noise plus \(n^{-1}\). Thus local-fit and
magnitude errors retain squared order \(n^{-1}\). For the weights,
nonnegative-extension control gives 
\begin{equation}
|\psi(s)-\psi(r)|\le C\sqrt{\psi(r)}|s-r|+C|s-r|^2. \tag{27}
\end{equation}
 The displayed evaluation tolerances then retain the squared weight
bound \eqref{eq:8}, including at arbitrarily flat points. Rounding the amplitude
data changes the normalized score by at most \(C/\sqrt N\). In the only
consequential region, \(H=O(\log n)\), these changes alter the squared
discrimination level by \(o(1)\). Above that range the error is absorbed
by the core floor.

Absolute precision is essential to this formulation. The estimator need
not evaluate an astronomically small flat mean with small relative
error. For \(e^{-x}\), rational Taylor bounds for \(e^x\) can be
inverted; if \(x\ge P\), the interval \([0,2^{-P}]\) already suffices
because \(e>2\). For the nested exponential, this small-value test can
be applied before forming the inner exponential. The function
\(\chi_\alpha\) is evaluated directly, rather than by raising an
underflowed approximation of \(\psi\) to a fractional power.

Conversely, an unrestricted small-score-error assertion would be false.
An adverse threshold displacement \(\zeta\) changes a simple Gaussian
error by the ratio 
\begin{equation}
\frac{\Phi(-\sqrt H+\zeta)}{\Phi(-\sqrt H)}
\sim\frac{\sqrt H}{\sqrt H-\zeta}
 e^{\zeta\sqrt H-\zeta^2/2}, \tag{28}
\end{equation}
 when \(H\to\infty\) and \(\zeta=o(\sqrt H)\). Hence \(\zeta\to0\)
alone is insufficient. A bounded product \(\zeta\sqrt{\log n}\) suffices
for bounded-factor tail precision in the consequential region; a
vanishing product preserves its conditional ratio-one tail. This is not
a necessary precision theorem for every possible arithmetic procedure.

At fixed polynomial degree, the post-observation root computation has
bit cost polynomial in the represented input length and \(\log n\). This
does not eliminate the cost of acquiring observations or evaluating
arbitrary functions. Smoothness alone does not imply
computability,
and no runtime uniform in unbounded Gaussian data magnitudes is
asserted.

\section{Acquisition under explicit source costs}
\label{sec:main-7}

Fix positive block costs \(c_F,c_G\), independent of the total budget
\(B\), and restrict integer allocations by 
\begin{equation}
c_Fn+c_Gm\le B. \tag{29}
\end{equation}
 Taking \((c_F,c_G)=(5,2)\) counts scalar measurements in \eqref{eq:11}; unit
costs count blocks. Other costs represent an explicitly supplied
economic objective.

\begin{theorem}[order-optimal allocation and its boundary]
\label{thm:cost-allocation}
For every fixed admissible model \eqref{eq:11},
the minimum feasible risk has order \((1+B)^{-1}\). A feasible sequence
attains \(O(B^{-1})\) if and only if 
\begin{equation}
n_B=\Omega(B),\qquad S(n_B,m_B)=O(B^{-1}), \tag{30}
\label{eq:30}
\end{equation}
 where \(S\) is the supremum in \eqref{eq:12}. For the family \eqref{eq:14}, put
\(\ell_B=\log(B/c_F)\). Then \eqref{eq:30} is equivalent to 
\begin{equation}
n_B=\Omega(B),\qquad
m_B\ge
\frac{2(1-\alpha)\ell_B-\log\ell_B}
 {[1-(\log\ell_B)^{-1/2}]^2}-O(1). \tag{31}
\label{eq:31}
\end{equation}

\end{theorem}

The core-invariant lower-bound pair proves that \(n_B=\Omega(B)\) is
necessary. For sufficiency in the general model,
\(m_B=\lceil\gamma\log(1+B)\rceil\), with \(\gamma>2/(1-a)^2\), and
allocation of the remaining cost to the core yield \(O(B^{-1})\). For
\eqref{eq:14}, the fixed-core window \eqref{eq:17}, together with
\(\log n_B=\ell_B+O(1)\), gives \eqref{eq:31}. The derivative of its window
function remains bounded. Taking the ceiling of the displayed resolver
boundary and then the floor of the affordable core count changes the
boundary by only a bounded amount; feedback from the resolver cost is
\(O(\log B/B)\). Appendix \ref{app:s6} supplies the complete all-count argument.

This theorem identifies the minimum resolver use needed for a rate
order. It does not identify the exact total-risk optimizer. To see the
difference, define the explicit continuous proxy 
\[
\Pi_B(m)=\frac{c_F}{B-c_Gm}
+A(B/c_F)^{-\alpha}(1+m)^{-1/2}e^{-a_Bm},
\]
 where \(A>0\) and \(a_B=[1-(\log\ell_B)^{-1/2}]^2/2\). Strict
convexity gives a unique minimizer, and its derivative equation yields

\begin{equation}
m_B^{\mathrm{proxy}}=
\frac{2(2-\alpha)\ell_B-\log\ell_B}
 {[1-(\log\ell_B)^{-1/2}]^2}+O(1). \tag{32}
\end{equation}
 The integer proxy minimizer is one of the adjacent integers. It uses
a larger leading logarithmic count because it balances marginal core
cost against an ambiguity term of order \(B^{-2}\), whereas attaining
the regular order requires only an ambiguity term of order \(B^{-1}\).
Neither the proxy constant nor its optimizer is identified with the
unknown exact minimax risk. Logarithmic resolver necessity also does not
hold for every admissible function: \(\chi\equiv0\) requires none.

\section{A shared-direction calibration extension}
\label{sec:main-8}

The finite count-independent constellation in Theorem~\ref{thm:global-risk} preserves a
simple Gaussian prefactor. A continuous nuisance orbit requires separate
analysis. Fix \(d\ge1\), let \(u\in S^{d-1}\) be an unknown unit
direction, and consider 
\begin{equation}
A\sim N_d(\sqrt k\,u,I_d),\qquad
B\sim N_d(\zeta R u,I_d),\qquad \zeta\in\{0,1\}. \tag{33}
\label{eq:33}
\end{equation}
 Here \(k\) is calibration information and \(R^2\) is detector
information. Define \(e_d(R,k)\) as the minimax binary classification
error over the common direction and both labels, allowing randomized
tests. The calibration and detector directions are not optimized
independently.

Let \(\mu_d\) be uniform probability measure on the sphere and define 
\[
Z_d(t)=\int_{S^{d-1}}e^{tu_1}\,d\mu_d(u).
\]
 The ratio of the two direction-mixture densities is 
\begin{equation}
L_{R,k}(a,b)=e^{-R^2/2}
\frac{Z_d(|\sqrt k\,a+Rb|)}{Z_d(\sqrt k\,|a|)}. \tag{34}
\end{equation}
 Orthogonal-group averaging reduces minimax testing to these two
mixtures. For \(R>0\), a likelihood-ratio threshold that equalizes their
errors is minimax and has the same error for every direction in each
class. At \(R=0\), the classes coincide and a fair decision has error
\(1/2\).

\begin{proposition}[fixed-radius calibrated discrimination]
\label{thm:calibrated-discrimination}
Write 
\[
M_d(k)=\int\sqrt{dP_k\,d\overline P_k},\quad
P_k=N_d(\sqrt k e_1,I_d),\quad
\overline P_k=\int N_d(\sqrt k u,I_d)\,d\mu_d(u),
\]
 where the affinity is not squared, and let
\(K_d=2^{2-d}/[\pi^{1/4}\sqrt{\Gamma(d/2)}]\). For fixed \(d\), as
\(R\to\infty\), 
\begin{equation}
e_d(R,k)=(1+o(1))K_dM_d(k)R^{(d-3)/2}e^{-R^2/8}
(1+2k/R^2)^{(d-1)/4}, \tag{35}
\label{eq:35}
\end{equation}
 uniformly over all \(k\ge0\). In the joint regime \(R,k\to\infty\),

\begin{equation}
e_d(R,k)\sim\Phi(-R/2)(1+R^2/(2k))^{(d-1)/4}. \tag{36}
\end{equation}
 For all \(R,k\ge0\), 
\begin{equation}
e_d(R,k)\asymp
\frac{e^{-R^2/8}}{1+R}
\left(1+\frac{R^2}{1+k}\right)^{(d-1)/4}. \tag{37}
\end{equation}

\end{proposition}

Appendix \ref{app:s4} gives the complete invariant reduction, affinity formula,
and uniform asymptotic proof. The bounded-calibration affinity in \eqref{eq:35}
cannot be replaced by its large-\(k\) approximation when asserting an
exact constant. For \(d\ge2\), bounded-factor known-direction
performance requires \(k=\Omega(R^2)\), while a ratio tending to one
requires \(k/R^2\to\infty\). At \(d=1\), the direction set has only two
elements and creates no unbounded power penalty.

\subsection{Transfer through an unknown contact coordinate}

Now let the common parameter be
\((x,\zeta,u)\in[0,a]\times\{0,1\}\times S^{d-1}\), with \(0<a<1/2\).
Fix \(\psi\in C^2([0,a])\), positive for \(x>0\), with
\(\psi(0)=\psi'(0)=0\). Observe 
\begin{equation}
\begin{aligned}
(X_i,V_i)&\sim N_{d+1}((x,\zeta\psi(x)u),I),&&i=1,\ldots,n,\\
U_j&\sim N_d(\zeta(1-x)u,I),&&j=1,\ldots,m,\\
C_l&\sim N_d(u,I),&&l=1,\ldots,k.
\end{aligned} \tag{38}
\label{eq:38}
\end{equation}
 The target is \((x,\zeta)\), with loss
\((\widehat x-x)^2+(\widehat\zeta-\zeta)^2\), and \(u\) is nuisance.
There are \((d+1)n+d(m+k)\) scalar observations. Put 
\begin{equation}
H_\psi(n,m)=\min_{0\le x\le a}\{n\psi(x)^2+m(1-x)^2\}. \tag{39}
\label{eq:39}
\end{equation}
 Let \(\mathcal E_d\) denote the minimax label error, allowing
randomized tests, and \(\mathcal R_d\) the minimax squared target risk
in \eqref{eq:38}. The large-core transfer construction below is a deterministic
function of counted observations.

\begin{theorem}[unknown-coordinate calibration transfer]
\label{thm:calibration-transfer}
Uniformly over all nonnegative integer triples,

\begin{equation}
\mathcal R_d(n,m,k)\asymp(1+n)^{-1}
+e_d(\sqrt{H_\psi(n,m)},k). \tag{40}
\label{eq:40}
\end{equation}
 For each fixed \(C_0<\infty\), 
\begin{equation}
\sup_{0\le m\le C_0\log n,\ k\ge0}
\left|\frac{\mathcal E_d(n,m,k)}
{e_d(\sqrt{H_\psi(n,m)},k)}-1\right|\longrightarrow0. \tag{41}
\end{equation}

\end{theorem}

The exact lower bound fixes a minimizing \(x_*\) in \eqref{eq:39}, retaining both
labels and the complete sphere. Projection onto the now-known
source-weight vector gives \eqref{eq:33} plus independent ancillary observations.
Separately, at \(\zeta=0\), varying \(x\) leaves every detector and
calibration mean fixed; this proves the core floor in \eqref{eq:40}.

The upper bound requires more than substituting an estimated \(x\) into
\eqref{eq:35}. Estimate \(x\) from its independent scalar mean and define 
\[
v(x)=(\sqrt n\psi(x),\sqrt m(1-x)),\quad
\hat v=v(\hat x),\quad R_0=\sqrt{H_\psi(n,m)}.
\]
 For \(m\ge1\), the projected detector has conditional law 
\[
\hat B\sim N_d\left(\zeta\frac{\langle\hat v,v(x)\rangle}{\|\hat v\|}u,I_d\right).
\]
 Because \(\hat x\in[0,a]\), \(\|\hat v\|\ge R_0\) on every
realization. From \(n\ge d+1\) already counted scalar
observations,
form the independent standard Gaussian vector of Helmert contrasts 
\[
W_r=\frac{\sum_{i=1}^rX_i-rX_{r+1}}{\sqrt{r(r+1)}},\qquad r=1,\ldots,d.
\]
 It is independent of the scalar sample mean, detector means, and
calibration. With \(s=R_0/\|\hat v\|\), the attenuated vector 
\begin{equation}
\widetilde B=s\hat B+\sqrt{1-s^2}W \tag{42}
\end{equation}
 has conditional identity covariance and mean 
\[
\zeta R_0\frac{\langle\hat v,v(x)\rangle}{\|\hat v\|^2}u.
\]
 Thus attenuation reduces even a large true detector radius to a
controlled perturbation of the fixed least-favourable radius. It uses
neither extra observations nor an unknown-direction recentering.

Nonnegative-extension control gives relative radius error \(O(\eta_n)\),
uniformly in \(x,m\ge1\), where
\(\eta_n=n^{-1/4}\sqrt{\log n}+n^{-1/2}\log n\). For Gaussian measures
differing by a mean shift \(h\), every event satisfies 
\begin{equation}
Q(E)\le\Phi\{\Phi^{-1}(P(E))+\|h\|\}. \tag{43}
\end{equation}
 Applying this extremal-event inequality to the nominal equalizer
bounds its relative error inflation by \(\exp\{C\eta_nR_0(1+R_0)\}\),
independently of \(k\). Since \(H_\psi(n,m)\le m\), this factor tends to
one in the stated logarithmic range. A sufficiently small Gaussian
localization-tail probability remains negligible even after division by
the rare nominal error. The null detector law is unchanged exactly. For
larger \(m\), a resolver-only norm test has error \(O(n^{-1})\), proving
the remaining all-count upper bound. Zero resolver counts and bounded
primary counts are handled separately in Appendix \ref{app:s5}.

\subsection{Information windows and the insufficiency of scalar distance}

\begin{corollary}[calibrated regular-risk boundary]
\label{thm:calibrated-boundary}
Let
\(\ell=\log n\) and \(H=H_\psi(n,m)\). Regular target risk is
characterized by 
\begin{equation}
\frac H8+\frac12\log(1+H)
-\frac{d-1}{4}\log\left(1+\frac H{1+k}\right)
\ge\ell-O(1). \tag{44}
\end{equation}
 Uniformly over all calibration counts, its information boundary has
representative 
\begin{equation}
H=8\ell-4\log\ell+
2(d-1)\log\left(1+\frac{\ell}{1+k}\right)+O(1). \tag{45}
\label{eq:45}
\end{equation}

\end{corollary}

\begin{figure}
\centering
\includegraphics[width=\linewidth]{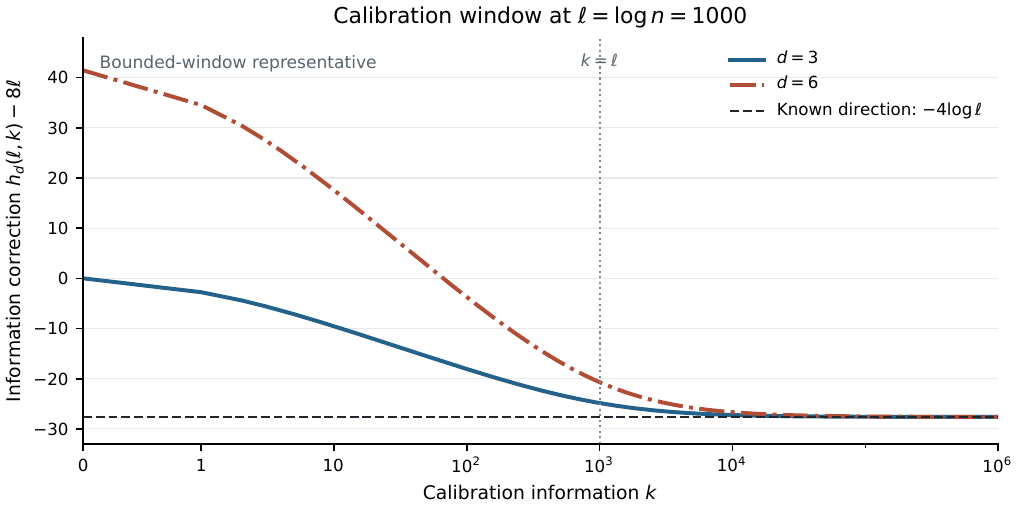}
\caption{Calibration changes the logarithmic information correction. The
representative in \eqref{eq:45}, minus \(8\ell\), is evaluated at
\(\ell=\log n=1000\) for \(d=3,6\) and integer calibration counts. The
dashed horizontal line is the known-direction correction \(-4\log\ell\);
the dotted vertical line marks \(k=\ell\). The axis is linear from zero
to one and logarithmic thereafter. The information boundary is specified
only up to an additive bounded term, so this is a formula-derived
representative rather than an exact finite-sample risk curve.}
\end{figure}

For \(k\asymp\ell^\rho\), \(0\le\rho\le1\), the numerator becomes
\(8\ell+\{2(d-1)(1-\rho)-4\}\log\ell\). To convert this information into
resolver count, one must invert the shared contact geometry: 
\begin{equation}
M_h(n)=\sup_{0\le x\le a}
\frac{h-n\psi(x)^2}{(1-x)^2}. \tag{46}
\end{equation}
 This is the exact continuous count needed for \(H_\psi\ge h\);
rounding changes it by at most one. A bounded information error changes
the count by only a bounded amount: as a function of \(m\),
\(H_\psi(n,m)\) has secant slopes in \([(1-a)^2,1]\), while its inverse
\(M_h(n)\), as a function of \(h\), has secant slopes in
\([1,(1-a)^{-2}]\). For the flat function \eqref{eq:14}, the count boundary is
\eqref{eq:45} divided by \((1-r_\ell)^2\), up to \(O(1)\). Calibration and
contact corrections therefore enter a common inversion.

For every \(k\), the minimum cross-label squared mean distance in \eqref{eq:33}
is \(R^2\). For positive calibration, both the null and alternative mean
orbits have dimension \(d-1\). Comparing \(k=1\) with \(k=R^2\)
therefore preserves both orbit dimensions, yet for fixed \(d\ge2\), 
\begin{equation}
\frac{e_d(R,1)}{e_d(R,R^2)}\asymp R^{(d-1)/2}\to\infty. \tag{47}
\end{equation}
 Minimum distance together with orbit dimension therefore does not
determine the risk comparison. This does not equate complete labelled
distance arrays, which do change with calibration. It explains why the
finite-constellation theorem and the continuous-orbit theorem require
different hypotheses.

\section{Scope and implications for fusion design}
\label{sec:main-9}

The results are mathematical statements about the declared Gaussian
experiments. Their validity follows from the proofs in the article and
the appendices. The figures illustrate the stated response
geometry and proved formulas; they do not supply empirical premises for
the theorems.

The risk decomposition gives a direct design rule. First identify which
source supplies an estimation direction that the other sources cannot
recover. Its observation count imposes an independent variance floor.
Next identify the competing common-parameter configurations and their
target differences. Additional observations are valuable when they
suppress the target-weighted discrimination term at the intended loss
scale. Finally, account for any continuous shared nuisance: alignment
information can affect discrimination beyond what a scalar separation
records.

The constructive results retain this distinction at finite precision.
Theorem~\ref{thm:certified-inference} gives sufficient accuracy for nuisance fitting and source
evaluation at the rare-error scale. Theorem~\ref{thm:tangency-repair} uses an explicit
shape-and-sign estimator and has a separate order-preservation
requirement. Proposition~\ref{thm:calibrated-discrimination} and Theorem~\ref{thm:calibration-transfer} establish the exact invariant
testing object and its statistical transfer; they do not require a
general-purpose numerical equalizer as a premise. Thus each statistical
assertion has a complete mathematical justification at its stated
precision.

Known covariance, fixed source functions, fixed dimension and
deterministic counts are substantive assumptions. Unknown covariance,
growing dimension, drifting source directions, correlated observations
and misspecified response maps require additional analysis. The
constants in the finite-fibre theorem are not uniform under collapsing
branch angles. The tangency repair is uniform in the stated known gain,
including zero, but does not cover an unknown gain. These boundaries
identify the conditions under which the acquisition laws can be used.

The combined results distinguish three sources of difficulty: local
estimation within a branch, discrimination between branches with
different targets, and alignment of a continuous shared nuisance. Exact
fibres determine
identification,
while the rates of approximate agreement determine the observation
resources required for accurate estimation. This distinction remains
consequential even when every local Fisher information matrix is
positive definite and the target is globally identified.

\subsection*{Research tools}

OpenAI ChatGPT/Codex was used for mathematical reconstruction, candidate
derivations, counterexample search, literature retrieval, proof development
and checking, auxiliary numerical verification, preparation of the
mathematical figures, and manuscript organization, drafting, editing and
reference checking. The authors are responsible for all content. The figures
evaluate the stated maps and formulas and contain no generated experimental
observations.

\section{Conclusion}
\label{sec:main-10}

Global fusion risk depends on how the source and target differences of
common-parameter alternatives vanish across distinct local regions. For
finite nuisance fibres with pairwise nonparallel branch tangents, this
dependence admits an all-count target-weighted Gaussian-tail
characterization with a matching measurable estimator. The full-box
specialization shows that local Fisher regularity, target-identifying
fibres, and complete critical jets can still leave the global rate
undetermined. Certified finite inference and source-cost analysis make
that characterization operational at a declared computational and
asymptotic scope. At a polynomial tangency boundary, the auxiliary
precision and target vanishing order give distinct improvement and
regular-rate thresholds, uniformly through zero known gain. The
shared-direction extension identifies a complementary phenomenon:
calibration changes composite discrimination through nuisance alignment
even when scalar minimum distance remains fixed. Together, these results
support rigorous acquisition decisions while exposing the exact
geometric conditions needed for their validity.

\subsection*{Funding}

This research received no funding.

\subsection*{Declaration of competing interests}

The authors declare no competing interests.

\subsection*{Data availability}

This is a theoretical study whose results are established by the proofs
in the article and the appendices. Additional supporting
materials are available from the corresponding author upon reasonable
request.

\clearpage
\appendix
\renewcommand{\thesection}{S\arabic{section}}
\renewcommand{\theHsection}{appendix.\arabic{section}}
\numberwithin{theorem}{section}
\renewcommand{\theHtheorem}{\theHsection.\arabic{theorem}}
\section*{Appendices: complete proofs}
\addcontentsline{toc}{section}{Appendices: complete proofs}

These appendices give the complete experiments and proofs supporting the
main text. Gaussian coordinates are real, independent and unit variance
unless explicitly stated otherwise. Counts refer to repeated source blocks.
Each comparison or asymptotic retains the precision and uniformity specified
in its statement. References are shared with the main text and collected at
the end of this document.

\begin{center}
\small
\begin{tabular}{@{}ll@{}}
\toprule
Main result & Complete proof \\
\midrule
Theorem~\ref{thm:global-risk}: global finite-fibre risk & \ref{app:s1-1} \\
Corollary~\ref{thm:full-box-risk}: full-box specialization & \ref{app:s1-2} and \ref{app:s2-1} \\
Theorem~\ref{thm:critical-jets}: critical jets and acquisition window & \ref{app:s2-2}--\ref{app:s2-3} and \ref{app:s2-5} \\
Theorem~\ref{thm:tangency-boundary}: polynomial tangency boundary & \ref{app:s1-3} \\
Theorem~\ref{thm:tangency-repair}: gain-uniform tangency repair & \ref{app:s7-1}--\ref{app:s7-6} \\
Theorem~\ref{thm:certified-inference}: certified estimator & \ref{app:s3-1}--\ref{app:s3-7} \\
Theorem~\ref{thm:cost-allocation}: cost-constrained acquisition & \ref{app:s6-1}--\ref{app:s6-4} \\
Proposition~\ref{thm:calibrated-discrimination}: calibrated fixed-radius testing & \ref{app:s4-1}--\ref{app:s4-4} \\
Theorem~\ref{thm:calibration-transfer}: unknown-coordinate transfer & \ref{app:s5-1}--\ref{app:s5-7} \\
Corollary~\ref{thm:calibrated-boundary}: calibrated acquisition boundary & \ref{app:s5-8} \\
\bottomrule
\end{tabular}
\end{center}

\clearpage
\section{Finite fibres with nonparallel branch tangents and target risk}
\label{app:s1}

\subsection{Primitive model and exact all-count risk law}
\label{app:s1-1}

Let \(J\subset\mathbb R\) and \(I\subset\mathbb R\) be compact intervals
with nonempty interiors. Let \(C:I\to\mathbb R^D\) be the restriction of
a fixed \(C^2\) map defined on a neighborhood of \(I\). Assume:

\begin{enumerate}
\setlength{\itemsep}{0pt}
\item
  \(C'(b)\ne0\) for every \(b\in I\).
\item
  There are only finitely many image points \(z_h\), \(1\le h\le H_0\),
  with multiple preimages. Write their complete fibres as 
\[
B_h=C^{-1}(z_h)=\{\beta_{h,1},\ldots,\beta_{h,M_h}\},\qquad M_h\ge2.
\]
 Every other fibre is a singleton.
\item
  For each \(h\) and each \(i\ne l\), the two vectors
  \(C'(\beta_{h,i})\) and \(C'(\beta_{h,l})\) are linearly independent.
\end{enumerate}

These are
deterministic,
count-independent conditions on a curve. They neither assert an
estimator bound nor encode such a bound. They allow asymmetric double
points and multiple points with more than two preimages. In dimension
two, a multiple point may have any fixed finite number of pairwise
distinct tangent lines. Endpoints of \(I\) are permitted among the
preimages.

Let \(q:I\to\mathbb R\) and \(\chi:J\to\mathbb R\) be fixed \(C^1\)
functions. Fix \(T_*>0\) such that 
\begin{equation}
\frac12\operatorname{diam}(q(I))\,\|\chi\|_\infty\le T_*.
\tag{S1.1.1}
\label{eq:s1-1-1}
\end{equation}
 The common parameter is \((r,b,t)\in J\times I\times[-T_*,T_*]\), the
target is the ordinary coordinate \(t\), and the loss is squared target
error. The primary source, with count \(n=n_0\), has mean 
\begin{equation}
F(r,b,t)=\bigl(r,\ t-q(b)\chi(r),\ C(b),\ b a_0(r)\bigr).
\tag{S1.1.2}
\label{eq:s1-1-2}
\end{equation}
 There are finitely many further sources, indexed by \(j=1,\ldots,s\),
with counts \(n_j\). Source \(j\) has either mean 
\begin{equation}
G_j(r,b,t)=(r,b a_j(r))\quad\text{or}\quad G_j(r,b,t)=b a_j(r).
\tag{S1.1.3}
\label{eq:s1-1-3}
\end{equation}
 Let \(\mathcal L\subset\{0,\ldots,s\}\) denote the sources that
include the radius coordinate \(r\); thus \(0\in\mathcal L\). Every
\(a_j\) is fixed and \(C^2\) on a neighborhood of \(J\). For
\(j\in\mathcal L\), assume that \(a_j\) is nonnegative on that
neighborhood. For \(j\notin\mathcal L\), assume \(a_j(r)\ge c_j>0\)
there. A known zero primary amplitude is allowed; its pure-noise
coordinate can be omitted without changing the experiment.

All repetitions and sources share the same \((r,b,t)\). Their blocks and
coordinates are independent. Define 
\[
N_L=\sum_{j\in\mathcal L}n_j,\qquad
H_{\boldsymbol n}(r)=\sum_{j=0}^s n_j a_j(r)^2,
\]
 
\[
d_{h,il}=\frac{|\beta_{h,i}-\beta_{h,l}|}{2},\qquad
w_{h,il}=\frac{|q(\beta_{h,i})-q(\beta_{h,l})|^2}{4},
\]
 
\begin{equation}
\mathcal S(\boldsymbol n)
=\max_{h,i<l}\ \sup_{r\in J}
 w_{h,il}\chi(r)^2
 \Phi\bigl(-d_{h,il}\sqrt{H_{\boldsymbol n}(r)}\bigr),
\tag{S1.1.4}
\label{eq:s1-1-4}
\end{equation}
 with maximum over an empty set equal to zero. Here \(\Phi\) is the
standard normal distribution function. The profile retains the actual
source and target gaps of the same parameter pair; there is no
sourcewise nuisance minimization.

\begin{theorem}
Under these hypotheses, the minimax target risk
satisfies 
\begin{equation}
\quad
\mathcal R(\boldsymbol n)\asymp \frac1{1+n}+\mathcal S(\boldsymbol n)
\quad
\tag{S1.1.5}
\label{eq:s1-1-5}
\end{equation}
 uniformly over all nonnegative integer counts. Constants may depend
on the fixed maps, intervals, source number, and curve geometry, but not
on the counts. In particular, constants are not claimed uniform as two
distinct endpoint target values merge or as branch angles vanish. This
is a comparison theorem, not an exact multiplicative minimax constant.
\end{theorem}

\subsubsection*{A. Quantitative all-pairs localization derived from the primitive hypotheses}
\addcontentsline{toc}{subsubsection}{A. Quantitative all-pairs localization derived from the primitive hypotheses}

There exist \(\varepsilon_0,L_0>0\) such that, whenever
\(\|C(b)-C(c)\|\le\varepsilon\le\varepsilon_0\), either 
\begin{equation}
|b-c|\le L_0\varepsilon,
\tag{S1.1.6}
\label{eq:s1-1-6}
\end{equation}
 or, for a common fibre \(B_h\) and distinct indices \(i,l\), 
\begin{equation}
|b-\beta_{h,i}|+|c-\beta_{h,l}|\le L_0\varepsilon.
\tag{S1.1.7}
\label{eq:s1-1-7}
\end{equation}

To prove this, first use \(\inf_I\|C'\|>0\), uniform continuity of
\(C'\), and Taylor's formula to obtain a uniform local lower bound
\(\|C(b)-C(c)\|\ge c_0|b-c|\) for \(|b-c|\) sufficiently small. Near a
distinct equality pair \((\beta_{h,i},\beta_{h,l})\), the derivative of
\((x,y)\mapsto C(x)-C(y)\) has two independent columns. If its least
singular value is \(\sigma>0\), Taylor's formula gives 
\[
\|C(\beta_{h,i}+u)-C(\beta_{h,l}+v)\|
\ge \sigma\sqrt{u^2+v^2}-M(u^2+v^2)
\ge (\sigma/2)\sqrt{u^2+v^2}
\]
 in a sufficiently small fixed rectangle. There are only finitely many
such rectangles. Outside their union and a fixed neighborhood of the
diagonal, compactness gives a strictly positive minimum of
\(\|C(b)-C(c)\|\), because all its zeros have already been listed.
Shrink \(\varepsilon_0\), and take the maximum of finitely many inverse
constants. This proves
\eqref{eq:s1-1-6}--\eqref{eq:s1-1-7},
also at interval endpoints by one-sided restriction.

Choose fixed, pairwise disjoint closed interval charts
\(I_{h,i}\subset I\) around all \(\beta_{h,i}\), so small that \(C\) is
uniformly bi-Lipschitz on each chart and a point sufficiently close to
\(\beta_{h,i}\) belongs to its chart. The chart is one-sided when
necessary.

\subsubsection*{B. Actual lower-bound configurations}
\addcontentsline{toc}{subsubsection}{B. Actual lower-bound configurations}

The conversion from testing error to squared estimation risk is the standard
two-point reduction; see \cite[Chapter~2]{Tsybakov2009}. The alternatives below
retain the common parameter across every source.

Fix \(h,i,l,r\), put \(\bar q=(q(\beta_{h,i})+q(\beta_{h,l}))/2\), and
compare 
\[
\theta_i=(r,\beta_{h,i},[q(\beta_{h,i})-\bar q]\chi(r)),
\quad
\theta_l=(r,\beta_{h,l},[q(\beta_{h,l})-\bar q]\chi(r)).
\]
 Condition \eqref{eq:s1-1-1} places both points in the declared full box. Their
shifted target coordinates both equal \(-\bar q\chi(r)\), their curve
coordinates agree, and every observed radius agrees. Their complete
squared Gaussian mean distance is 
\[
4d_{h,il}^2 H_{\boldsymbol n}(r).
\]
 Thus their exact equal-prior binary testing error is
\(\Phi(-d_{h,il}\sqrt H)\). Nearest-target classification turns every
target estimator into a test and gives risk at least
\(w_{h,il}\chi(r)^2\Phi(-d_{h,il}\sqrt H)\). Taking the supremum yields
\(\mathcal R\ge\mathcal S\).

For the separate floor, fix any \((r,b)\) and compare
\(t=\pm c(1+n)^{-1/2}\), with fixed \(0<c<T_*\). Only the shifted target
coordinate of the primary source changes. The Gaussian Kullback--Leibler
divergence is \(2nc^2/(1+n)\), bounded independently of every other
count, while target separation squared is \(4c^2/(1+n)\). Hence
\(\mathcal R\ge c'/(1+n)\). The maximum of these two bounds is at least
half their sum.

\subsubsection*{C. Fully specified measurable estimator}
\addcontentsline{toc}{subsubsection}{C. Fully specified measurable estimator}

For bounded \(n\), including zero, return zero; bounded loss and the
floor handle all auxiliary allocations. For sufficiently large \(n\),
pool exactly the observed radius coordinates, clip to \(J\), and obtain

\[
\hat r=\operatorname{clip}_{J}\bar R,\qquad
\bar R\sim N(r,N_L^{-1}).
\]
 Denote the shifted target and curve means by \(\bar X\) and
\(\bar C\). The radius noise is independent of these means and of every
amplitude mean. Fix dense sequences in \(I\) and in every finite union
of the fixed charts. Every minimization below selects the first sequence
element whose Euclidean residual is within \(1/n\) of the infimum.
Compactness and continuity identify this infimum with the compact
minimum, and a countable first-index rule is measurable.

First fit \(\widetilde b\) to \(\bar C\) over all of \(I\). Then 
\begin{equation}
\|C(\widetilde b)-C(b)\|
\le 2\|\bar C-C(b)\|+1/n.
\tag{S1.1.8}
\label{eq:s1-1-8}
\end{equation}
 Set \(\delta_n=A\sqrt{\log n/n}\), where \(A\) is a sufficiently
large fixed constant. If \(\widetilde b\) is farther than
\(A_1\delta_n\) from every multiple-fibre preimage, trust this fit and
set \(\widehat b=\widetilde b\). Choose \(A_1>3L_0\).

Otherwise, for all sufficiently large \(n\), the nearby preimage belongs
to a uniquely determined fibre \(B_h\); deterministic smallest-index
conventions resolve the irrelevant exceptional ties. Form 
\[
v=(\sqrt{n_j}a_j(r))_{n_j>0},\qquad
\widehat v=(\sqrt{n_j}a_j(\hat r))_{n_j>0},
\]
 and let \(Y\) collect the corresponding whitened amplitude means.
Conditional on \(\hat r\), 
\[
Y\sim N(bv,I).
\]
 Select an index \(\ell\) minimizing the finite set 
\begin{equation}
\|Y-\beta_{h,\ell}\widehat v\|^2,
\tag{S1.1.9}
\label{eq:s1-1-9}
\end{equation}
 with smallest-index ties. If \(\widehat v=0\), use the same
deterministic convention. Next, let 
\begin{equation}
\mathcal I_{h,\ell}
=\bigcup_{k:\ q(\beta_{h,k})=q(\beta_{h,\ell})} I_{h,k}.
\tag{S1.1.10}
\label{eq:s1-1-10}
\end{equation}
 Refit \(C\) to \(\bar C\) over this finite union and call the result
\(\widehat b\). Finally return 
\begin{equation}
\widehat t=\operatorname{clip}_{[-T_*,T_*]}
 \{\bar X+q(\widehat b)\chi(\hat r)\}.
\tag{S1.1.11}
\end{equation}
 The target-value grouping \eqref{eq:s1-1-10} is essential: an incorrect
nuisance-branch label with the same endpoint target value must not
introduce a screening-radius loss.

\subsubsection*{D. Actual-error refitting and target loss}
\addcontentsline{toc}{subsubsection}{D. Actual-error refitting and target loss}

Let \(\mathcal E_C=\{\|\bar C-C(b)\|\le\delta_n\}\). By a coordinatewise
Gaussian tail bound and a sufficiently large \(A\),
\(\Pr(\mathcal E_C^c)\le 2D n^{-8}\). On \(\mathcal E_C\), a trusted fit
must satisfy \eqref{eq:s1-1-6}, because \eqref{eq:s1-1-7} would place it within
\(3L_0\delta_n\) of a listed preimage. Its squared parameter error is
therefore bounded by a constant times \((2\|\bar C-C(b)\|+1/n)^2\),
whose expectation is \(O(n^{-1})\).

If a fibre \(B_h\) is selected, the all-pairs lemma shows that the true
\(b\) is within \(C\delta_n\) of one uniquely determined
\(\beta_{h,i}\), and hence lies in its fixed chart. Suppose the selected
target class is correct, meaning \(q(\beta_{h,\ell})=q(\beta_{h,i})\).
The refitting domain then contains the true \(b\), so its curve
discrepancy again obeys \eqref{eq:s1-1-8}. The all-pairs lemma implies either an
ordinary \(O(\|\bar C-C(b)\|+1/n)\) parameter error, or closeness at
that same actual-noise scale to two endpoints with equal \(q\) values.
Since \(q\) is Lipschitz, in either case 
\begin{equation}
|q(\widehat b)-q(b)|
\le C\{\|\bar C-C(b)\|+1/n\}.
\tag{S1.1.12}
\end{equation}
 Thus no \(\log n/n\) penalty is introduced even when several branches
share their endpoint target value.

On an incorrect target class, the selected local fit lies within
\(C\delta_n\) of a preimage in that class. Indeed the corresponding
endpoint is feasible and has curve value \(z_h\), while the true curve
value is within \(C\delta_n\) of \(z_h\). The local inverse bound gives
this assertion. Consequently 
\[
|q(\widehat b)-q(b)|
\le |q(\beta_{h,\ell})-q(\beta_{h,i})|+C\delta_n.
\]
 Among the fixed, finite, nonzero endpoint target gaps there is a
positive minimum. For sufficiently large \(n\), the last expression is
bounded by a fixed multiple of the displayed endpoint gap. If there are
no nonzero gaps, this incorrect-class case is empty.

Using boundedness and Lipschitz continuity of \(\chi\) and \(q\),
clipping, the variance of \(\bar X\), and
\(\mathbb E|\hat r-r|^2\le1/N_L\le1/n\), the target risk is at most 
\begin{equation}
C/n+C\chi(r)^2\sum_{h,i,\ell:\ q(\beta_{h,i})\ne q(\beta_{h,\ell})}
 w_{h,i\ell}\Pr\{\text{select \(\ell\) from true branch \(i\), good events}\}.
\tag{S1.1.13}
\label{eq:s1-1-13}
\end{equation}
 The finite sum and bounded exceptional-event loss change only
constants.

\subsubsection*{E. Independent localization preserves the sharp Gaussian tail}
\addcontentsline{toc}{subsubsection}{E. Independent localization preserves the sharp Gaussian tail}

A fixed nonnegative \(C^2\) extension gives the Glaeser estimate \cite{Glaeser1963}
\(|a_j'|^2\le C a_j\) on \(J\) for \(j\in\mathcal L\). Taylor's theorem
therefore gives 
\[
|a_j(\hat r)-a_j(r)|
\le C\sqrt{a_j(r)}|\hat r-r|+C|\hat r-r|^2.
\]
 For \(j\notin\mathcal L\), the positive lower bound and bounded
derivative instead give
\(|a_j(\hat r)-a_j(r)|^2\le C a_j(r)^2|\hat r-r|^2\). On the independent
radius event 
\[
\mathcal E_R=\{|\hat r-r|\le A\sqrt{\log n/N_L}\},
\]
 whose complement has probability \(O(n^{-8})\) after increasing
\(A\), put \(E=\|\widehat v-v\|^2\). Cauchy--Schwarz over the localizing
sources yields 
\begin{equation}
E\le C\left\{
\frac{\log n}{\sqrt{N_L}}\sqrt H
+\frac{\log^2 n}{N_L}
+\frac{H\log n}{N_L}
\right\}.
\tag{S1.1.14}
\label{eq:s1-1-14}
\end{equation}
 Here \(H=H_{\boldsymbol n}(r)\). This estimate is uniform in every
auxiliary count, including counts much larger than the primary count.

For a true branch \(i\) and selected competitor \(\ell\), write
\(d=|\beta_{h,i}-\beta_{h,\ell}|/2>0\) and
\(c=(\beta_{h,i}+\beta_{h,\ell})/2\). Selection of \(\ell\) in \eqref{eq:s1-1-9}
implies crossing the pairwise Gaussian bisector between these two
candidate means. The midpoint \(c\) need not vanish. On the good curve
event the true amplitude is within \(C\delta_n\) of \(\beta_{h,i}\). If
\(\widehat v\ne0\), the standardized signed distance to the wrong
bisector is at least 
\begin{equation}
\mu\ge(d-C\delta_n)
\frac{\widehat v\cdot v}{\|\widehat v\|}
-|c|\sqrt E.
\tag{S1.1.15}
\label{eq:s1-1-15}
\end{equation}
 Every amplitude component is nonnegative, so the projection is
nonnegative. Orthogonal projection gives 
\begin{equation}
\frac{(\widehat v\cdot v)^2}{\|\widehat v\|^2}\ge H-E.
\tag{S1.1.16}
\label{eq:s1-1-16}
\end{equation}
 Together,
\eqref{eq:s1-1-15}--\eqref{eq:s1-1-16}
imply 
\begin{equation}
\mu\ge d\sqrt H-C\delta_n\sqrt H-C\sqrt E.
\tag{S1.1.17}
\label{eq:s1-1-17}
\end{equation}
 A harmless enlargement of \(C\) makes this valid also when \(E>H\),
because the right side can then be made nonpositive while \eqref{eq:s1-1-15}
remains at least \(-C\sqrt E\).

For every fixed \(K<\infty\), \eqref{eq:s1-1-14} implies, uniformly where
\(H\le K\log n\), 
\[
\delta_n H+\sqrt{HE}+E=o(1).
\]
 Hence the squared positive discrimination level differs from \(d^2H\)
by at most \(o(1)\); if \(H\) is itself tiny, the desired tail is
bounded below by a positive constant and no sign restriction is needed.
The Mills comparison 
\[
\Phi(-\sqrt x)\asymp (1+x)^{-1/2}e^{-x/2},\qquad x\ge0,
\]
 shows that a bounded additive squared-argument perturbation changes
the tail by at most a fixed factor. Conditional on the independent
radius and curve data, the wrong-bisector probability is therefore at
most 
\begin{equation}
C\Phi(-d\sqrt H),\qquad H\le K\log n.
\tag{S1.1.18}
\label{eq:s1-1-18}
\end{equation}
 If \(\widehat v=0\), then \(H=E\); in this low-information range
\eqref{eq:s1-1-14} forces \(H=o(1)\), so the same bound follows from probability
at most one.

For \(H>K\log n\), \eqref{eq:s1-1-14} gives \(E/H=o(1)\) uniformly because
\(N_L\ge n\), and \eqref{eq:s1-1-17} gives \(\mu\ge d\sqrt H/2\) for sufficiently
large \(n\). Choose \(K\) using the smallest positive, fixed \(d\) among
the finitely many pairs. Then the wrong-bisector probability is at most
\(\exp(-d^2H/8)\le n^{-4}\). Thus this whole region is absorbed by the
regular floor. Combining \eqref{eq:s1-1-13}, \eqref{eq:s1-1-18}, the finite pair count,
and the lower bounds proves \eqref{eq:s1-1-5}.

\subsection{Exact scope, specializations, and limits}
\label{app:s1-2}

\subsubsection*{The flat-coordinate specialization}
\addcontentsline{toc}{subsubsection}{The flat-coordinate specialization}

Take \(I=[-2,2]\), \(C(b)=(b^2,b^3-b)\), \(q(b)=b\), \(J=[-a,a]\),
\(a_0=\psi\), and one localizing resolver with \(a_1(r)=1-r\). The curve
has exactly one multiple fibre \(\{-1,1\}\); its tangent vectors
\((-2,2)\) and \((2,2)\) are independent. Thus \(d=1\), \(w=1\), and
\eqref{eq:s1-1-5} is precisely \ref{app:s2-1}, with its original five- and two-dimensional
blocks. Evenness reduces the full-radius supremum to \([0,a]\) exactly
as in the flat-coordinate model. The present proof also derives the
needed all-pairs property from primitive nonparallel branch geometry,
while the flat-coordinate model's explicit polynomial identity remains a
stronger direct certificate for that specialization.

\subsubsection*{Positive-amplitude scalar sources}
\addcontentsline{toc}{subsubsection}{Positive-amplitude scalar sources}

Take \(I=[-1,1]\), the same curve, \(q(b)=b\), \(\chi=1\), \(J=[0,1]\),
and \(a_0=0\). Omit the known-zero amplitude coordinate. The primary is
then exactly \((r,t-b,b^2,b^3-b)\). Two scalar nonlocalizing sources
with \(a_1(r)=P(r)>0\), \(a_2(r)=1+r\) define a positive-amplitude
experiment when \(P\) is \(C^2\). Its positive lower bounds satisfy the
nonlocalizing condition, so no radius observations have been silently
added. Since \(w=d=1\), the profile is 
\[
\Phi\left(-\sqrt{\min_r\{n_1P(r)^2+n_2(1+r)^2\}}\right),
\]
 where the primary count \(n_3\) is the count \(n\) in \ref{app:s1-1}. The
application uses \(C^2\) amplitudes; an extension to \(C^1\) amplitudes
is outside the hypotheses of \ref{app:s1-1}.

\subsubsection*{Limits of the scalar-curve parameterization}
\addcontentsline{toc}{subsubsection}{Limits of the scalar-curve parameterization}

A two-pole harmonic experiment with free complex amplitudes generally
has a multidimensional parameter manifold modulo labels and complex
resolver constellations. Such a model is outside the scalar-curve form
\eqref{eq:s1-1-2}--\eqref{eq:s1-1-3}.
Gaussian binary bounds for fixed pairs and finite-constellation
comparisons remain available, but a global unknown-parameter theorem
requires its own control of approximate aliases and local charts.

\subsubsection*{Identification and local regularity}
\addcontentsline{toc}{subsubsection}{Identification and local regularity}

The primary map \eqref{eq:s1-1-2} is everywhere immersive: its radius coordinate
isolates the \(r\) derivative, its shifted target coordinate isolates
the \(t\) derivative, and \(C'(b)\ne0\) isolates the \(b\) derivative.
Subject to \eqref{eq:s1-1-1}, it globally identifies the target exactly when 
\begin{equation}
\chi(r)\{q(\beta_{h,i})-q(\beta_{h,l})\}=0
\quad\text{whenever }a_0(r)=0,
\tag{S1.2.1}
\label{eq:s1-2-1}
\end{equation}
 for all listed fibre pairs. Equality of primary observations first
fixes \(r\), then either fixes \(b\) or selects one listed fibre; a
nonzero \(a_0(r)\) resolves the scalar amplitude. Equation \eqref{eq:s1-2-1} then
proves sufficiency, and the actual pairs in the lower bound prove
necessity. Thus even within this reusable class, local immersion and
exact target identification do not determine the profile in \eqref{eq:s1-1-4}.

\subsection{A polynomial obstruction when nonparallel pair localization is removed}
\label{app:s1-3}

Fix an integer \(p\ge2\), \(b\in[-2,2]\), and \(t\in[-2,2]\). Set 
\[
C_p(b)=(b^2-1,\ b(b^2-1)^p),\qquad q_p(b)=b(b^2-1),
\]
 and observe \(n\) independent unit-covariance Gaussian blocks with
mean 
\begin{equation}
F_p(b,t)=(t-q_p(b),\ C_p(b)).
\tag{S1.3.1}
\end{equation}
 An independent directly observed radius may be appended without
affecting any conclusion.

\setcounter{theorem}{2}
\begin{proposition}
This fixed polynomial model is everywhere
immersive, its only nontrivial exact nuisance fibre is \(\{-1,1\}\), its
exact fibres globally identify \(t\), and 
\begin{equation}
\mathcal R_p(n)\asymp(1+n)^{-1/p}.
\tag{S1.3.2}
\label{eq:s1-3-2}
\end{equation}
 Consequently the exact-fibre target-weighted profile in \eqref{eq:s1-1-4} is
identically zero for this model, whereas its risk is slower than
\(n^{-1}\). Removing the nonparallel-pair condition from \ref{app:s1-1} would make
the theorem false even for polynomial maps on a full box.
\end{proposition}

\textbf{Proof of geometry.} 
\[
C_p'(b)=\left(2b,(b^2-1)^{p-1}\{(2p+1)b^2-1\}\right).
\]
 For \(b\ne0\) its first component is nonzero, while at zero its
second component is \((-1)^p\ne0\). Hence it is an immersion. Equality
of first coordinates forces \(c=b\) or \(c=-b\); the latter equality in
the second coordinate forces \(b(b^2-1)^p=0\), leaving only the distinct
pair \(\pm1\). At both endpoints of that fibre \(q_p=0\), so the shifted
target coordinate determines the same target. But the two tangent
vectors at \(\pm1\) are \((\pm2,0)\), which are collinear.

\textbf{Actual lower bound.} For small \(s>0\), take 
\[
b_\pm=\pm\sqrt{1+s},\qquad t_\pm=\pm s\sqrt{1+s}.
\]
 Their first mean coordinate is zero, their second is \(s\), and their
last coordinates are \(\pm s^p\sqrt{1+s}\). Their target squared
half-gap is \((1+s)s^2\), and their exact binary error is
\(\Phi(-\sqrt{n(1+s)s^{2p}})\). Taking \(s=c n^{-1/(2p)}\) gives a fixed
positive error and risk at least \(c'n^{-1/p}\). Bounded counts follow
by fixed separated alternatives.

\textbf{Matching upper bound.} On the compact curve image there is a
global target inverse modulus 
\begin{equation}
|q_p(b)-q_p(c)|\le C\|C_p(b)-C_p(c)\|^{1/p}.
\tag{S1.3.3}
\label{eq:s1-3-3}
\end{equation}
 Near the diagonal, local immersion proves a Lipschitz bound, which
implies this weaker H\"older bound on a compact set. Away from the
diagonal and the two ordered double-point
neighborhoods,
compactness gives a positive image separation. It remains to check \(b\)
near \(1\), \(c\) near \(-1\). Put \(x=b^2-1\), \(y=c^2-1\); there
\(b>0\), \(c<0\), and \(|b|,|c|\) are bounded above and below. If
\(xy<0\), 
\[
|q_p(b)-q_p(c)|\le C(|x|+|y|)=C|x-y|,
\]
 which is controlled by the first image-coordinate difference. If
\(xy\ge0\), then for either parity of \(p\), 
\[
|b x^p-c y^p|=|b|\,|x|^p+|c|\,|y|^p
\ge c_0(|x|+|y|)^p,
\]
 where for odd \(p\) a common sign may be factored out. The left side
is the absolute second image-coordinate difference and bounds
\(|q_p(b)-q_p(c)|^p\) up to a constant. This proves \eqref{eq:s1-3-3}.

Fit \(C_p\) by a measurable minimum-distance rule to its Gaussian sample
mean and estimate \(t\) by the shifted-target sample mean plus \(q_p\)
of that fit, followed by clipping. The curve residual is at most twice
the actual Gaussian noise plus \(1/n\). Equation \eqref{eq:s1-3-3} and finite
Gaussian moments give squared target risk
\(O(n^{-1}+n^{-1/p})=O(n^{-1/p})\). This proves \eqref{eq:s1-3-2}.

The example identifies the additional information needed at a tangency:
its contact order \(p\) determines the slower risk law. Finite exact
fibres alone cannot replace quantitative control between branches.

\subsection{Composite-nuisance obstruction}
\label{app:s1-4}

For the point-versus-sphere problem treated in \ref{app:s4} without calibration 
\[
Z\sim N_d(bRu,I_d),\qquad b\in\{0,1\},\quad u\in S^{d-1},
\]
 the minimum cross-label squared mean distance is \(R^2\) for every
fixed dimension \(d\). Every simple pair has error \(\Phi(-R/2)\), but
\ref{app:s4} gives the composite minimax error 
\[
e_d(R)\sim K_dR^{(d-3)/2}e^{-R^2/8}.
\]
 As \(R\to\infty\), its ratio to the simple-pair tail is of order
\(R^{(d-1)/2}\). In particular for \(d\ge2\), the scalar minimum
distance cannot determine even a bounded-factor comparison to the binary
tail. This comparison concerns the scalar minimum distance. The complete
labelled pairwise-distance arrays contain additional information and are
not being equated.

The obstruction can be stated within one fixed dimension and one fixed
two-source architecture. In the orientation experiment, the complete
least cross-label distance of 
\[
A\sim N_d(\sqrt{k}u,I_d),\qquad B\sim N_d(bRu,I_d)
\]
 is exactly \(R^2\) for every \(k\), because
\(\inf_{u,v}\{k\|u-v\|^2+R^2\}=R^2\). For every positive calibration
count, both the null and alternative mean orbits have dimension \(d-1\).
Compare \(k=1\) with \(k=R^2\), preserving both dimensions.
Nevertheless,
for fixed \(d\ge2\), the orientation-calibration theorem gives 
\[
\frac{e_d(R,1)}{e_d(R,R^2)}\asymp R^{(d-1)/2}\longrightarrow\infty.
\]
 Thus the descriptor consisting only of the minimum cross-target
distance and the orbit dimension is insufficient even within fixed
dimension; calibration alignment of the shared orbit matters. The
fixed-radius calibration law in \ref{app:s4-3}-\ref{app:s4-4} quantifies this dependence.

The continuous orientation orbit changes the hypotheses needed for a
sharp risk theorem. Conditional on its independently localized scalar
radius and its finite curve chart, \ref{app:s1-1} has only finitely many known
scalar candidate means, with a count-independent finite pair count. In
the sphere problem, the difficult alternatives form a continuous orbit
whose effective angular complexity increases with \(R\). The finite
union bound that retains the binary prefactor in \ref{app:s1-1} has no uniform
finite replacement there. The orientation experiment quantifies how an
additional shared-direction calibration source changes that orbit
penalty.

\newpage

\section{Flat target gaps and acquisition windows}
\label{app:s2}

\subsection{Exact all-count comparison on a fixed full box}
\label{app:s2-1}

Fix \(0<a<1/2\) and let

\[
K=[-a,a]_r\times[-2,2]_b\times[-2,2]_t,\qquad T(r,b,t)=t.
\]

Let \(\psi,\chi\) be fixed even \(C^2\) functions on a neighborhood of
\([-a,a]\), with

\[
\psi(0)=\psi'(0)=0,\quad \psi(r)>0\ (r\ne0),\quad 0\le\chi(r)\le1,\quad \chi(0)=0.
\]

Assume \(\psi\ge0\) on that neighborhood. In the smooth examples below
both functions are \(C^\infty\) and flat at zero. Define the labelled
means

\[
F(r,b,t)=\bigl(r,\ t-b\chi(r),\ b^2,\ b^3-b,\ b\psi(r)\bigr)\in\mathbb R^5,
\]

\begin{equation}
G(r,b,t)=\bigl(r,\ b(1-r)\bigr)\in\mathbb R^2. \tag{S2.1}
\end{equation}

Observe \(n,m\ge0\) independent blocks from \(N_5(F,I_5)\) and
\(N_2(G,I_2)\), respectively. All repetitions and sources share the same
parameter. The scalar count is \(5n+2m\). Set

\[
H_{n,m}(r)=n\psi(r)^2+m(1-r)^2,
\]

\begin{equation}
S_{\psi,\chi}(n,m)=\sup_{0\le r\le a}\chi(r)^2\Phi\bigl(-\sqrt{H_{n,m}(r)}\bigr). \tag{S2.2}
\label{eq:s2-2}
\end{equation}

The reduction of the supremum to nonnegative \(r\) is exact:
\(\psi,\chi\) are even and \((1-r)^2\le(1+r)^2\) for \(r\ge0\).

\begin{theorem}
For this fixed experiment, uniformly over all
nonnegative integer counts,

\begin{equation}
\quad R_{\psi,\chi}(n,m)\asymp\frac1{1+n}+S_{\psi,\chi}(n,m).\quad \tag{S2.3}
\label{eq:s2-3}
\end{equation}

The comparison constants may depend on the fixed functions and \(a\),
but not on counts. This is an all-count comparison theorem, not an exact
multiplicative minimax constant. Its upper bound is obtained by a fully
specified measurable estimator below.
\end{theorem}

\subsubsection*{Geometry: full rank and target identification}
\addcontentsline{toc}{subsubsection}{Geometry: full rank and target identification}

The derivatives of \(F\) in \(r,t,b\) are independent everywhere: the
first coordinate isolates the \(r\) derivative; the second isolates
\(t\); and the \(b\) derivative in the third and fourth coordinates is
\((2b,3b^2-1)\), which never vanishes. Thus \(DF\) has rank three on a
neighborhood of every point of \(K\).

To classify its fibres, equality of \(b^2\) gives \(b'=b\) or \(b'=-b\).
For opposite unequal values, equality of \(b^3-b\) forces \(b=\pm1\).
Equality of the last coordinate then forces \(\psi(r)=0\), hence
\(r=0\). Because \(\chi(0)=0\), equality of the second coordinate forces
\(t=t'\). Consequently the only nontrivial core fibres are

\begin{equation}
\{(0,1,t),(0,-1,t)\},\qquad -2\le t\le2. \tag{S2.4}
\label{eq:s2-4}
\end{equation}

Every fibre is
target-constant,
and each has at most two points. For \(|t|<2\) the two points of \eqref{eq:s2-4}
lie in the interior of the full box. Thus \(F\) globally identifies the
target even though it does not globally identify the nuisance. The
complete labelled exact fibre relations of both sources are the same for
every admissible choice of \(\chi\): the core relation was just computed
and \(G\) is unchanged. The lower-bound pairs below approach these
fibres while their target gap shrinks. There is no target-separated
exact core alias.

\subsubsection*{Lower bounds using actual common-parameter alternatives}
\addcontentsline{toc}{subsubsection}{Lower bounds using actual common-parameter alternatives}

For each \(r\in[0,a]\), take

\[
\theta_+=(r,1,\chi(r)),\qquad\theta_-=(r,-1,-\chi(r)).
\]

These belong to \(K\). Their first four core coordinates are identical,
while their last core difference is \(2\psi(r)\) and their second
resolver difference is \(2(1-r)\). The squared whitened mean distance is
\(4H_{n,m}(r)\), the real Gaussian KL is \(2H_{n,m}(r)\), and the exact
equal-prior optimal binary error is \(\Phi(-\sqrt{H_{n,m}(r)})\). A
wrong nearest-target decision entails squared error at least
\(\chi(r)^2\). Therefore

\begin{equation}
R_{\psi,\chi}(n,m)\ge S_{\psi,\chi}(n,m). \tag{S2.5}
\end{equation}

Separately fix \(r=b=0\) and compare \(t=\pm c(1+n)^{-1/2}\), with fixed
small \(c>0\). Only the second core coordinate changes. The resolver is
exactly invariant and the KL is bounded independently of \(m\). Hence

\begin{equation}
R_{\psi,\chi}(n,m)\ge c'/(1+n). \tag{S2.6}
\end{equation}

Taking the maximum proves the lower half of \eqref{eq:s2-3}. Unlimited resolver
measurements cannot remove this independent core-coordinate floor.

\subsubsection*{An explicit polynomial all-pairs dichotomy}
\addcontentsline{toc}{subsubsection}{An explicit polynomial all-pairs dichotomy}

Write \(C(b)=(b^2,b^3-b)\) and let \(d=b-b'\), \(s=b+b'\). Direct
factorization gives

\begin{equation}
\|C(b)-C(b')\|^2=d^2\left[s^2+\frac{(3s^2+d^2-4)^2}{16}\right]. \tag{S2.7}
\label{eq:s2-7}
\end{equation}

If \(|d|\le1\), the bracket is at least \(1/16\): either \(s^2\ge1/16\),
or \(3s^2+d^2-4\le-45/16\). Thus

\begin{equation}
|b-b'|\le4\|C(b)-C(b')\|\quad\text{when }|b-b'|\le1. \tag{S2.8}
\label{eq:s2-8}
\end{equation}

If \(|d|>1\) and \(\|C(b)-C(b')\|\le\varepsilon\le1\), \eqref{eq:s2-7} gives
\(|s|\le\varepsilon\) and
\(|d^2-4|\le4\varepsilon+3\varepsilon^2\le7\varepsilon\). Since
\(|d|+2>3\),

\[
\bigl||d|-2\bigr|\le7\varepsilon/3.
\]

It follows that, with \(\epsilon=\operatorname{sign}(d)\),

\begin{equation}
|b-\epsilon|\le2\varepsilon,\qquad |b'+\epsilon|\le2\varepsilon. \tag{S2.9}
\label{eq:s2-9}
\end{equation}

Thus two good fits are either close in the ordinary parameter or both
close to the opposite points \(\pm1\). This proves the exact all-pairs
property used by the estimator, rather than inferring it from derivative
rank.

\subsubsection*{Estimator construction}
\addcontentsline{toc}{subsubsection}{Estimator construction}

Bounded \(n\), including zero, are handled by returning zero: loss is at
most four and \((1+n)^{-1}\) absorbs this with a fixed constant. Assume
henceforth \(n\) exceeds a fixed sufficiently large threshold and put
\(N=n+m\).

Pool the first coordinates of both sources to obtain
\(\bar R\sim N(r,N^{-1})\) and set
\(\hat r=\operatorname{clip}_{[-a,a]}\bar R\). Denote the remaining core
means by

\[
\bar X=t-b\chi(r)+n^{-1/2}Z_X,\quad(\bar U,\bar V)=C(b)+n^{-1/2}(Z_U,Z_V),
\]

\[
\bar Y=b\psi(r)+n^{-1/2}Z_Y.
\]

For \(m>0\), let \(\bar B=b(1-r)+m^{-1/2}Z_B\) be the second resolver
mean. These noises are mutually independent, and independent of
\(\hat r\).

Fix a countable dense sequence in \([-2,2]\). Define \(\tilde b\) to be
the first element with residual to \((\bar U,\bar V)\) within \(1/n\) of
the countable infimum. Continuity and compactness make that infimum the
compact minimum. A first-index selection from countably many measurable
comparisons is measurable, and

\begin{equation}
\|C(\tilde b)-C(b)\|\le2\|(\bar U,\bar V)-C(b)\|+1/n. \tag{S2.10}
\label{eq:s2-10}
\end{equation}

Set \(\delta_n=\sqrt{16\log n/n}\). If

\[
\operatorname{dist}(\tilde b,\{-1,1\})>8\delta_n,
\]

put \(\hat b=\tilde b\). Otherwise put

\[
\hat q=\sqrt{\operatorname{clip}_{[1/4,4]}\bar U},
\]

\begin{equation}
A=n\psi(\hat r)\bar Y+m(1-\hat r)\bar B,\qquad\hat b=\hat q\operatorname{sign}(A), \tag{S2.11}
\label{eq:s2-11}
\end{equation}

where a missing resolver contribution is zero and
\(\operatorname{sign}(0)=1\). Finally return

\begin{equation}
\hat t=\operatorname{clip}_{[-2,2]}\{\bar X+\hat b\chi(\hat r)\}. \tag{S2.12}
\label{eq:s2-12}
\end{equation}

No true nuisance or noiseless coordinate is used. The countable
minimization is an existence-based measurable rule; no computational
complexity guarantee is claimed.

\subsubsection*{Regular branches and magnitude error}
\addcontentsline{toc}{subsubsection}{Regular branches and magnitude error}

On \(\|(\bar U,\bar V)-C(b)\|\le\delta_n\), the discrepancy in \eqref{eq:s2-10}
is at most \(3\delta_n\). If the estimator trusts \(\tilde b\), \eqref{eq:s2-9}
rules out \(|\tilde b-b|>1\), since that would put \(\tilde b\) within
\(6\delta_n\) of \(\pm1\). Equation \eqref{eq:s2-8} then bounds the error by four
times the actual discrepancy in \eqref{eq:s2-10}. Its squared expectation is
\(O(n^{-1})\), including the trusted-branch indicator.

If the estimator instead enters \eqref{eq:s2-11}, \(\tilde b\) is within
\(8\delta_n\) of \(\pm1\). Equation \eqref{eq:s2-10}, or just its first
coordinate, then shows \(\bigl||b|-1\bigr|\le C\delta_n\). For all
sufficiently large \(n\), \(b^2\in[1/4,4]\) and the clipped square root
obeys

\begin{equation}
|\hat q-|b||\le C|\bar U-b^2|. \tag{S2.13}
\label{eq:s2-13}
\end{equation}

The excluded Gaussian event has probability at most
\(e^{-8\log n}=n^{-8}\), since the squared norm of two standard normal
coordinates is \(\chi^2_2\). Bounded loss absorbs it. The final estimate
uses the actual errors \eqref{eq:s2-10}, \eqref{eq:s2-13}, rather than the larger
screening radius; thus no \(\log n/n\) loss appears.

Since \(\chi\) is bounded and Lipschitz on the fixed interval, clipping,
\eqref{eq:s2-12}, and independence of the radius coordinate give

\begin{equation}
\mathbb E(\hat t-t)^2\le C/n+C\chi(r)^2\Pr\{\text{wrong sign in (S2.11), good event}\}. \tag{S2.14}
\label{eq:s2-14}
\end{equation}

The \(C/n\) term includes the variance of \(\bar X\), the actual
magnitude/local-fit
error, the radius mean squared error \(N^{-1}\), and the
exceptional-event probability.

\subsubsection*{Noisy weighting preserves the complete Gaussian tail}
\addcontentsline{toc}{subsubsection}{Noisy weighting preserves the complete Gaussian tail}

A fixed nonnegative \(C^2\) extension of \(\psi\) implies the Glaeser bound \cite{Glaeser1963} \(|\psi'(r)|^2\le C\psi(r)\) on the compact interval. One
elementary proof uses Taylor's inequality at a displacement in the
negative derivative direction when that displacement stays in a fixed
extension neighborhood; the remaining derivative range is handled by
compactness. Taylor's formula consequently gives

\begin{equation}
|\psi(\hat r)-\psi(r)|\le C\sqrt{\psi(r)}|\hat r-r|+C|\hat r-r|^2. \tag{S2.15}
\label{eq:s2-15}
\end{equation}

Let

\[
v=(\sqrt n\psi(r),\sqrt m(1-r)),\qquad \hat v=(\sqrt n\psi(\hat r),\sqrt m(1-\hat r)),
\]

with zero-count components omitted, and \(H=\|v\|^2=H_{n,m}(r)\). Every
component is nonnegative. Conditional on the radius coordinate, the sign
score has signed mean \(|b|\hat v\cdot v\) and standard deviation
\(\|\hat v\|\). When \(\hat v\ne0\), its squared positive
signal-to-noise ratio is

\begin{equation}
\rho^2=b^2\frac{(\hat v\cdot v)^2}{\|\hat v\|^2}\ge b^2\{H-\|\hat v-v\|^2\}. \tag{S2.16}
\label{eq:s2-16}
\end{equation}

This follows from orthogonal projection: the distance of \(v\) from the
line through \(\hat v\) is at most \(\|v-\hat v\|\). If \(\hat v=0\),
then \(m=0\) and \(\hat r=0\); this has probability zero for a
continuous radius observation when \(n>0\), because zero is an interior
clipping point. An arbitrary convention on that null event suffices.

On the further Gaussian event

\[
|\hat r-r|\le\sqrt{16\log n/N},
\]

whose complement has probability at most \(2n^{-8}\), \eqref{eq:s2-15} yields

\begin{equation}
E:=\|\hat v-v\|^2\le C\left\{\frac{\log n}{\sqrt N}\sqrt H+\frac{H\log n}{N}+\frac{\log^2 n}{N}\right\}. \tag{S2.17}
\label{eq:s2-17}
\end{equation}

Indeed \(n\psi(r)\le\sqrt{nH}\), \(m\le H/(1-a)^2\), and \(n\le N\). On
the ambiguous branch \(b^2\ge1-C\delta_n\). Thus, for every fixed
\(K_0\), uniformly over all \(m\) and parameters with
\(H\le K_0\log n\),

\begin{equation}
\rho^2\ge H-o(1), \tag{S2.18}
\end{equation}

because \(N\ge n\) and \(\delta_n\log n\to0\). A bounded additive change
in a squared Gaussian-tail argument changes its tail by at most a fixed
factor, uniformly for nonnegative arguments. For
completeness,
this follows from the Mills comparison
\(\Phi(-\sqrt x)\asymp(1+x)^{-1/2}e^{-x/2}\); the ratio of the
polynomial factors for \(x\) and \((x-1)_+\) is bounded. Hence the
conditional wrong-sign probability in this region is at most

\begin{equation}
C\Phi(-\sqrt H). \tag{S2.19}
\label{eq:s2-19}
\end{equation}

In the region \(H>K_0\log n\), \eqref{eq:s2-17} gives \(E/H=o(1)\) uniformly, and
\eqref{eq:s2-16} gives \(\rho^2\ge H/2\) for sufficiently large \(n\). Choosing,
for example, \(K_0=32\) bounds the wrong-sign probability by
\(e^{-H/4}\le n^{-8}\). Combining this with \eqref{eq:s2-14}, \eqref{eq:s2-19}, and the
even-function reduction in \eqref{eq:s2-2} proves the upper half of \eqref{eq:s2-3} for
every allocation. The argument does not assume that \(m\) is
logarithmic; only potentially consequential errors automatically fall in
that range.

\subsection{Identical formal jets, arbitrarily different polynomial minimax exponents}
\label{app:s2-2}

For \(r\ne0\) set

\[
\psi_*(r)=\exp\left\{-\tfrac12\exp(1/r^2)\right\},\qquad\psi_*(0)=0,
\]

\begin{equation}
\chi_\alpha(r)=\psi_*(r)^\alpha,\qquad0<\alpha<1. \tag{S2.20}
\label{eq:s2-20}
\end{equation}

Both functions are even, \(C^\infty\), and flat at zero. Differentiating
any finite number of times produces a finite sum of factors polynomial
in \(r^{-1}\) and \(e^{1/r^2}\) multiplied by the original double
exponential; every such derivative tends to zero at the origin. The same
argument applies to every fixed positive \(\alpha\).

The box, target, source dimensions, covariance, complete labelled exact
fibre relations, and derivative-rank profile in \ref{app:s2-1} are identical
throughout this family. Moreover \textbf{all derivatives of every order
of both source maps at every point of the critical hypersurface
\(\{r=0\}\) agree for all \(\alpha\)}. In particular their complete
formal germs at both points of each nontrivial fibre \eqref{eq:s2-4} agree. The
target is the same linear coordinate map \(t\) throughout. The full
one-block core Fisher matrix is positive definite everywhere and, for
each fixed \(\alpha\), has a positive minimum eigenvalue on \(K\). This
is a pointwise and uniform local regularity statement; it does not
control near pairs coming from the two distinct charts in \eqref{eq:s2-4}.

Nevertheless,
with no resolver
observations,
putting \(q=n\psi_*(r)^2\) in \eqref{eq:s2-2} gives

\[
S_{\psi_*,\chi_\alpha}(n,0)=n^{-\alpha}\sup_{0\le q\le n\psi_*(a)^2}q^\alpha\Phi(-\sqrt q).
\]

The supremum converges to a positive finite constant. The upper
finiteness follows from the Gaussian tail; positivity follows by taking
any fixed \(q>0\), available for sufficiently large \(n\). Since
\(0<\alpha<1\), \ref{app:s2-1} yields

\begin{equation}
R_\alpha(n,0)\asymp n^{-\alpha}. \tag{S2.21}
\end{equation}

Thus the complete formal jets at the exact core-fibre set, together with
the complete exact fibre relation and the everywhere-full derivative
rank, do not determine even the \textbf{polynomial exponent} of
core-only global minimax risk in the \(C^\infty\) category. The
assertion is about jets at the critical set; it does not say that full
function values or Taylor data at every interior point agree. The flat
functions \eqref{eq:s2-20} are outside the globally subanalytic class.

For any fixed \(\alpha\ne\alpha'\), the two risks fail all-count
comparison equivalence already on \(m=0\). This is a statistical
distinction, not an arbitrary encoding claim. The construction makes no
assertion of analytic or quasianalytic flatness. Complete labelled
Gaussian distance arrays differ between the models; the distinction
concerns reduced summaries of the experiment, not equality of its full
geometry.

\subsection{The coupled polynomial--flat acquisition window}
\label{app:s2-3}

Write \(\ell=\log n\) and

\[
r_\ell=(\log\ell)^{-1/2},\qquad v_\ell=1-r_\ell,
\]

for sufficiently large \(n\) that \(r_\ell<a\). For every fixed
\(C_0<\infty\), uniformly over integers \(0\le m\le C_0\ell\),

\begin{equation}
S_{\psi_*,\chi_\alpha}(n,m)\asymp n^{-\alpha}\Phi(-\sqrt m\,v_\ell). \tag{S2.22}
\label{eq:s2-22}
\end{equation}

If additionally \(m\to\infty\), the stronger profile statement is

\begin{equation}
\frac{S_{\psi_*,\chi_\alpha}(n,m)}{n^{-\alpha}\Phi(-\sqrt m\,v_\ell)}\longrightarrow (2\alpha/e)^\alpha. \tag{S2.23}
\label{eq:s2-23}
\end{equation}

This is an exact constant for the optimized binary lower profile, not
the unknown composite minimax constant in \ref{app:s2-1}.

\textbf{Proof.} With \(q=n\psi_*(r)^2\), one has

\[
r=r(\ell-\log q),\qquad r(u)=(\log u)^{-1/2}.
\]

For \(q\in[\ell^{-2},\ell^2]\), the mean-value theorem gives

\[
\sup|r(\ell-\log q)-r_\ell|=O\{1/(\ell\sqrt{\log\ell})\}.
\]

Hence \(m\{(1-r(\ell-\log q))^2-v_\ell^2\}=o(1)\) uniformly in this
interval and the declared count range. A bounded additive tail shift and

\[
\Phi(-\sqrt{h+q})/\Phi(-\sqrt h)\le e^{-q/2}\qquad(h,q\ge0)
\]

give an integrable-in-the-supremum upper envelope
\(Cq^\alpha e^{-q/2}\). For \(q<\ell^{-2}\), \(r\le r_\ell\), so the
normalized profile is at most \(q^\alpha\le\ell^{-2\alpha}\). For
\(q>\ell^2\), bound \(H\ge q\) and use the Mills lower bound on the
denominator: the normalized profile is at most a constant times

\[
\sqrt{1+C_0\ell}\,e^{C_0\ell/2}\sup_{q\ge\ell^2}q^\alpha e^{-q/2}=o(1).
\]

On each fixed compact interval \(q\in[\epsilon,Q]\), the ratio in
\eqref{eq:s2-23} tends uniformly to \(q^\alpha e^{-q/2}\) when \(m\to\infty\).
The preceding envelopes let \(\epsilon\downarrow0\),
\(Q\uparrow\infty\). The maximum is attained at \(q=2\alpha\) and equals
\((2\alpha/e)^\alpha\). The same estimates and one fixed positive \(q\)
give the uniform comparison \eqref{eq:s2-22}, including bounded \(m\).

\ref{app:s2-1} and \eqref{eq:s2-22} imply, in the logarithmic range,

\begin{equation}
R_\alpha(n,m)\asymp n^{-1}+n^{-\alpha}(1+m)^{-1/2}\exp\{-m v_\ell^2/2\}. \tag{S2.24}
\end{equation}

Therefore the bounded acquisition window for regular \(n^{-1}\) risk is

\begin{equation}
m=\frac{2(1-\alpha)\ell-\log\ell}{(1-r_\ell)^2}+O(1). \tag{S2.25}
\end{equation}

More precisely, \(R_\alpha(n,m)=O(n^{-1})\) in this range if and only if

\[
\tfrac12 m v_\ell^2+\tfrac12\log(1+m)\ge(1-\alpha)\ell-O(1).
\]

The equivalence of this inequality to the window follows because its
critical \(m\) is proportional to \(\ell\), with a coefficient tending
to \(2(1-\alpha)>0\); replacing \(\log(1+m)\) by \(\log\ell\) changes
only a bounded term. For \(m/\log n\to\kappa\ge0\), the squared-risk
exponent is \(\min\{1,\alpha+\kappa/2\}\).

The factor \((1-r_\ell)^{-2}\) multiplies the discrimination requirement
after target coalescence has reduced it from \(\ell\) to
\((1-\alpha)\ell\). It cannot be applied to a fixed-target requirement
and then corrected by subtracting \(2\alpha\ell\) raw
observations:
that prescription oversamples by

\begin{equation}
2\alpha\ell\{(1-r_\ell)^{-2}-1\}\sim4\alpha\ell r_\ell\longrightarrow\infty. \tag{S2.26}
\end{equation}

This acquisition law couples target coalescence and source
discrimination through the same parameter, retaining both effects in the
observation requirement.

\subsection{A slowly shrinking target gap and an unbounded smaller cross term}
\label{app:s2-4}

Keep \(\psi_*\) but set

\[
\chi_{\mathrm{slow}}(r)=e^{-1/r^2}\quad(r\ne0),\qquad\chi_{\mathrm{slow}}(0)=0.
\]

This is again a fixed even smooth flat function, so all source jets at
\(r=0\) remain the same as in \ref{app:s2-2}. Uniformly for \(0\le m\le C_0\ell\),

\begin{equation}
S_{\psi_*,\chi_{\mathrm{slow}}}(n,m)\sim\ell^{-2}\Phi(-\sqrt m\,v_\ell). \tag{S2.27}
\label{eq:s2-27}
\end{equation}

\textbf{Proof.} Let \(r_- =r(\ell+2\log\ell)\) and
\(r_+=r(\ell-2\log\ell)\). At \(r_-\) the core information is
\(\ell^{-2}\), the target squared gap is
\((\ell+2\log\ell)^{-2}\sim\ell^{-2}\), and the resolver information
differs from \(m v_\ell^2\) by \(o(1)\). The normal-tail ratio for an
\(o(1)\) additive squared-argument change tends to one uniformly,
yielding the lower bound in \eqref{eq:s2-27}. On \(r\le r_+\), the target squared
gap is at most \((\ell-2\log\ell)^{-2}\sim\ell^{-2}\), and the resolver
gap is at least \(m(1-r_+)^2=m v_\ell^2+o(1)\). Dropping the nonnegative
core information gives the matching upper bound. On \(r>r_+\) the core
information exceeds \(\ell^2\); the Gaussian tail there is negligible
compared with \(\ell^{-2}\Phi(-\sqrt m\,v_\ell)\) for the declared count
range. This proves \eqref{eq:s2-27}.

Consequently the core-only risk has order \((\log n)^{-2}\) despite
immersion and target identification. Its regular acquisition window is

\begin{equation}
m=\frac{2\ell-5\log\ell}{(1-r_\ell)^2}+O(1). \tag{S2.28}
\end{equation}

Subtracting the coalescence saving \(4\log\ell\) from the fixed-target
flat window \((2\ell-\log\ell)/(1-r_\ell)^2\) again fails at bounded
precision: it oversamples by

\begin{equation}
4\log\ell\{(1-r_\ell)^{-2}-1\}\sim8\sqrt{\log\ell}\longrightarrow\infty. \tag{S2.29}
\end{equation}

This second example checks that the coupling is not restricted to a
leading polynomial exponent. Here both source-gap and target-gap germs
are flat; their relative scale creates a logarithmic base risk and an
unbounded subleading resource correction.

\subsection{Polynomial-exponent blindness within every fixed nonanalytic Gevrey class}
\label{app:s2-5}

Fix any Gevrey order \(s>1\) and choose a fixed real \(p\ge1/(s-1)\). In
the same full box and the same five-coordinate core and two-coordinate
resolver of \ref{app:s2-1}, put

\[
\psi_p(r)=\exp(-|r|^{-p}),\qquad
\chi_{p,\alpha}(r)=\exp(-\alpha|r|^{-p}),\qquad0<\alpha<1,
\]

with both values defined to be zero at \(r=0\). Then both mean maps are
in the fixed Gevrey class \(G^s\) on a neighborhood of \(K\). Throughout
this family, all labelled exact fibre relations agree, all source jets
along \(r=0\) agree, the core has full derivative rank everywhere, and
its Fisher matrix has a positive minimum eigenvalue on \(K\) for each
fixed \(\alpha\).
Nevertheless,

\begin{equation}
R_{p,\alpha}(n,0)\asymp n^{-\alpha}. \tag{S2.30}
\label{eq:s2-30}
\end{equation}

Consequently polynomial-exponent blindness of the critical formal jets
persists inside every prescribed Gevrey class of order strictly greater
than one. Neither source number, block dimension, parameter dimension,
domain nor target is changed. Constants are for each fixed \(\alpha\);
no uniformity as \(\alpha\downarrow0\) is asserted.

\textbf{Complete derivative bound.} For \(c>0\), consider
\(f_c(x)=e^{-c x^{-p}}\) on \(x>0\). Choose a fixed
\(0<\varepsilon<1/2\), depending only on \(p\), so small that

\[
\operatorname{Re}(1+w)^{-p}\ge1/2\quad\text{for }|w|\le\varepsilon.
\]

Such a choice exists by continuity, using the analytic branch
\((1+w)^{-p}=\exp\{-p\operatorname{Log}(1+w)\}\) in \(|w|<1\). On the
complex disk \(|z-x|\le\varepsilon x\), the function \(e^{-c z^{-p}}\)
is therefore analytic and bounded in modulus by \(e^{-(c/2)x^{-p}}\).
Cauchy's derivative estimate gives, for every integer \(k\ge1\),

\[
|f_c^{(k)}(x)|\le k!(\varepsilon x)^{-k}e^{-(c/2)x^{-p}}.
\]

Putting \(u=x^{-p}\), the maximum of \(u^{k/p}e^{-cu/2}\) over \(u>0\)
is \((2k/(cp e))^{k/p}\). Since \((k/e)^k\le k!\),

\begin{equation}
|f_c^{(k)}(x)|\le
\left\{\varepsilon^{-1}(2/(cp))^{1/p}\right\}^{k}(k!)^{1+1/p}. \tag{S2.31}
\label{eq:s2-31}
\end{equation}

For each fixed \(k\), the unoptimized bound tends to zero as
\(x\downarrow0\). Hence the even extension \(f_c(|r|)\), given value
zero at the origin, is \(C^\infty\) with every derivative zero there; on
the negative half-line the derivative magnitudes are the same. The
order-zero bound is at most one. Equation \eqref{eq:s2-31} proves membership in
\(G^{1+1/p}\) and therefore in \(G^s\), since \(1+1/p\le s\). Apply it
with \(c=1\) and \(c=\alpha\). Multiplying these functions by the
polynomial coordinates \(b\) and adding the other fixed polynomial
coordinates preserves the same Gevrey order. More explicitly, mixed
derivatives involving the nonpolynomial terms contain at most one
derivative in \(b\) and a derivative in \(r\), so \eqref{eq:s2-31} immediately
bounds them by \(CA^{|\beta|}(|\beta|!)^s\) on a compact neighborhood of
\(K\).

\textbf{Statistical proof.} These functions satisfy every assumption of
\ref{app:s2-1}, and \(\chi_{p,\alpha}^2=(\psi_p^2)^\alpha\). The same exact change
of variable \(q=n\psi_p(r)^2\) gives

\[
S_{\psi_p,\chi_{p,\alpha}}(n,0)
=n^{-\alpha}\sup_{0\le q\le n\psi_p(a)^2}q^\alpha\Phi(-\sqrt q)
\asymp n^{-\alpha}.
\]

Combining with \ref{app:s2-1} proves \eqref{eq:s2-30}, because \(\alpha<1\). The fibre,
rank and jet statements follow directly from the already proved \ref{app:s2-1}
geometry and the flat derivative limits above. The conclusion follows
from the same proved risk comparison.

\textbf{Precise quasianalytic boundary of this construction.} In a
quasianalytic function class, the Taylor map on germs is injective by
definition: two germs in the class with identical derivatives at a point
coincide on a neighborhood of that point. Thus the distinct source germs
used here, which differ arbitrarily close to \(r=0\) while having
identical Taylor series, cannot all belong to one quasianalytic class.
This excludes this particular equal-jet construction there. It does not
establish a general risk classification for quasianalytic models, a
universal finite-order contact theorem, or a sufficient criterion for
global regular minimax risk.

\newpage

\section{Finite certified estimation and precision transfer}
\label{app:s3}

\subsection{Experiment, target and precision}
\label{app:s3-1}

Fix \(0<a<1/2\), and let

\[
K=[-a,a]\times[-2,2]\times[-2,2],\qquad \theta=(r,b,t),\qquad T(\theta)=t.
\]

The functions \(\psi,\chi\) are fixed even \(C^2\) functions on an open
neighborhood of \([-a,a]\). Assume \(\psi\ge0\) on that
neighborhood,
\(\psi(0)=\psi'(0)=0\), \(\psi(r)>0\) for \(r\ne0\), \(0\le\chi\le1\) on
the parameter interval, and \(\chi(0)=0\). Their values and derivatives,
including the size of a nonnegative extension neighborhood for \(\psi\),
are allowed to enter comparison constants. No common constant over an
unrestricted collection of such functions is asserted.

Observe \(n\) independent \(N_5(F(\theta),I_5)\) blocks and \(m\)
independent \(N_2(G(\theta),I_2)\) blocks, where

\[
F(r,b,t)=(r,t-b\chi(r),b^2,b^3-b,b\psi(r)),
\qquad G(r,b,t)=(r,b(1-r)).
\]

All coordinates and blocks have independent unit-variance real Gaussian
noises; all sources share the same parameter. Counts \(n,m\) are
nonnegative integers, and the scalar-observation count is \(5n+2m\).
Define

\[
\mathcal R(n,m)=\inf_{\widehat t}\sup_{\theta\in K}
\mathbb E_\theta(\widehat t-t)^2,
\qquad H_{n,m}(r)=n\psi(r)^2+m(1-r)^2,
\]
 
\[
S(n,m)=\sup_{0\le r\le a}\chi(r)^2\Phi(-\sqrt{H_{n,m}(r)}),
\]

where \(\Phi\) is the standard real normal distribution function. The
infimum is over measurable estimators based on the observations. The
result is a uniform comparison, with constants independent of both
counts:

\begin{equation}
\quad \mathcal R(n,m)\asymp (1+n)^{-1}+S(n,m).\quad \tag{S3.1}
\label{eq:s3-1}
\end{equation}

The construction below gives the upper bound using finite rational
arithmetic and a declared absolute-accuracy function oracle. For
explicit flat families, rational function enclosures implement that
oracle.

\subsection{Independent lower bounds and exact geometry}
\label{app:s3-2}

For each \(r\in[0,a]\), the actual full-box points 
\[
\theta_+=(r,1,\chi(r)),\qquad\theta_-=(r,-1,-\chi(r))
\]
 have target gap \(2\chi(r)\). Their first four core coordinates
coincide; their remaining differences are \(2\psi(r)\) and \(2(1-r)\).
Thus the squared whitened Gaussian mean distance is \(4H_{n,m}(r)\), the
real Gaussian Kullback--Leibler divergence is \(2H_{n,m}(r)\), and their
equal-prior testing error is exactly \(\Phi(-\sqrt{H_{n,m}(r)})\).
Turning any target estimator into the nearest-target classifier shows 
\[
\mathcal R(n,m)\ge S(n,m).
\]
 This argument optimizes over actual common-parameter pairs. There is
no sourcewise duplication of the nuisance.

For the independent floor, use \((r,b,t)=(0,0,\pm c/\sqrt{1+n})\), with
fixed small \(c>0\). Only the second core coordinate changes, and the
resolver is exactly invariant. The divergence is bounded for every
allocation, while squared target separation has order \((1+n)^{-1}\).
Therefore 
\[
\mathcal R(n,m)\ge c_0(1+n)^{-1}.
\]
 The maximum of these two bounds is at least one half of their sum,
after adjusting a fixed constant.

The core derivative has rank three everywhere: its first coordinate
isolates the \(r\) derivative, its second isolates the \(t\) derivative,
and the \(b\) derivative in coordinates three and four is
\((2b,3b^2-1)\), which is never zero. Equality of core values implies
equality of \(r\), then \(b'=b\) or \(b'=-b\). The latter possibility
with distinct points requires \(b=\pm1\), then \(\psi(r)=0\), hence
\(r=0\). Since \(\chi(0)=0\), the targets agree. Precisely the
nontrivial fibres are 
\[
\{(0,1,t),(0,-1,t)\},\qquad -2\le t\le2.
\]
 Thus the core identifies the target globally, even at its nontrivial
nuisance fibres. This geometric statement does not itself give the
global risk rate.

Write \(C(b)=(b^2,b^3-b)\), \(d=b-b'\), and \(s=b+b'\). Factoring each
difference gives the exact identity 
\begin{equation}
\|C(b)-C(b')\|^2
=d^2\left[s^2+\frac{(3s^2+d^2-4)^2}{16}\right]. \tag{S3.2}
\label{eq:s3-2}
\end{equation}
 For \(|d|\le1\), the bracket is at least \(1/16\). Indeed, either
\(s^2\ge1/16\), or \(3s^2+d^2-4\le-45/16\), in which case its squared
contribution alone exceeds \(1/16\). Therefore 
\begin{equation}
|b-b'|\le4\|C(b)-C(b')\|\qquad (|b-b'|\le1). \tag{S3.3}
\label{eq:s3-3}
\end{equation}
 If instead \(|d|>1\) and the norm in \eqref{eq:s3-2} is at most
\(\epsilon\le1\), then \(|s|\le\epsilon\) and 
\[
|d^2-4|\le4\epsilon+3\epsilon^2\le7\epsilon,
\qquad \bigl||d|-2\bigr|\le7\epsilon/3.
\]
 Consequently,
with \(e=\operatorname{sign}(d)\), 
\begin{equation}
|b-e|\le2\epsilon,\qquad |b'+e|\le2\epsilon. \tag{S3.4}
\label{eq:s3-4}
\end{equation}
 These estimates hold on all of \([-2,2]\), including both portions
outside \([-1,1]\). They supply the
local-fit/opposite-branch
dichotomy needed below.

\subsection{The finite polynomial computation}
\label{app:s3-3}

For background on certified real-root isolation and its computational
complexity, see Sagraloff and Mehlhorn \cite{Sagraloff2016}. The fixed-degree
Sturm construction used here is specified below.

For observed or rounded values \(U,V\), put 
\[
P_{U,V}(b)=(b^2-U)^2+(b^3-b-V)^2.
\]
 Direct
differentiation,
with every term retained, gives 
\begin{equation}
\begin{aligned}
\tfrac12 P'_{U,V}(b)
&=2b(b^2-U)+(3b^2-1)(b^3-b-V)\\
&=3b^5-2b^3-3Vb^2+(1-2U)b+V.
\end{aligned} \tag{S3.5}
\label{eq:s3-5}
\end{equation}
 The leading coefficient is always three. Thus this stationary
polynomial is never identically zero. Every global minimum of
\(P_{U,V}\) on the compact interval \([-2,2]\) is either an endpoint or
a real root of \eqref{eq:s3-5} in the interval. There are at most seven distinct
candidate points, including endpoints. Repeated stationary roots are
retained as points rather than discarded because they do not change a
derivative sign. Evaluating all candidates and choosing the smallest
\(b\) among exact minimum ties specifies an exact-real global minimizer.
A local optimizer or a sign-change-only search does not supply this
certificate.

For rational \(U,V\), a fully rational approximation avoids
algebraic-value comparison. Let \(\tau>0\) be the required objective
tolerance. On \([-2,2]\), \(\|C(b)\|<8\) and \(\|C'(b)\|<12\), so 
\begin{equation}
|P'_{U,V}(b)|\le M_{U,V},\qquad
M_{U,V}=24(8+|U|+|V|). \tag{S3.6}
\label{eq:s3-6}
\end{equation}
 Remove repeated factors by dividing \eqref{eq:s3-5} by its greatest common
divisor with its derivative. Use its square-free Sturm sequence to
isolate every distinct real root into a rational interval of width at
most \(\tau/M_{U,V}\). Exact rational roots encountered at endpoints or
bisection points are recorded and deflated; counts are only applied on
intervals whose endpoints are not roots. Include the two original
endpoints, evaluate \(P_{U,V}\) exactly at every root-interval midpoint
and endpoint, and choose the least rational candidate among minimum
objective ties. The returned rational number \(\widetilde b\) satisfies

\begin{equation}
P_{U,V}(\widetilde b)
\le\min_{[-2,2]}P_{U,V}+\tau. \tag{S3.7}
\label{eq:s3-7}
\end{equation}
 To see this, if a global minimum is stationary, one midpoint is
within the certified interval width of it, and \eqref{eq:s3-6} bounds the
objective increase by \(\tau\). Endpoint minima are included exactly.
Taking the least objective over all candidates can only improve this
bound.

Repeated roots, objective ties and root separation do not jeopardize
termination. The square-free polynomial has finitely many distinct
separated real roots. Each bisection eventually isolates one root or
proves a root count of zero, and at most five exact deflation restarts
are possible. A certificate records the stationary polynomial, repeated
factor, root intervals, interval widths, objective values, selected
point and the Lipschitz bound in \eqref{eq:s3-6}. Verification recomputes the
root counts and every rational objective comparison.

\subsection{Explicit tolerances and estimator}
\label{app:s3-4}

Return zero when \(n\le1\). The bounded loss handles these counts. For
\(n\ge2\), write \(N=n+m\). Let \(\bar R\) be the pooled
first-coordinate sample mean from all \(N\) blocks; it is
\(N(r,N^{-1})\). Denote the other core means by
\(\bar X,\bar U,\bar V,\bar Y\), in their source order, and the second
resolver mean by \(\bar B\) when \(m>0\). Their noises are mutually
independent and independent of \(\bar R\).

The following absolute tolerances are sufficient:

\begin{center}
\small
\begin{tabular}{@{}lr@{}}
\toprule
Quantity & Certified absolute error or objective tolerance \\
\midrule
Rounded \(\bar X,\bar U,\bar V\) & \((64n)^{-1}\) each \\
Rounded \(\bar R,\bar Y,\bar B\) & \((64N)^{-1}\) each \\
Nuisance objective tolerance \(\tau\) & \((64n^2)^{-1}\) \\
\(\psi\) evaluation at the rounded, clipped radius & \((64N)^{-1}\) \\
\(\chi\) evaluation at that radius & \((64n)^{-1}\) \\
Magnitude square-root evaluation & \((64n)^{-1}\) \\
Optional extra normalized score arithmetic error & \(N^{-1/2}\) \\
\bottomrule
\end{tabular}
\end{center}

The last row is unnecessary in the rational-arithmetic
construction,
whose unnormalized score is evaluated exactly. It states a tolerance for
other arithmetic procedures. These are absolute-error requirements. They
do not demand a relative approximation of an arbitrarily small flat
function. The construction reserves half of each sample-mean budget for
representing the real input by a rational enclosure midpoint and the
other half for its internal dyadic rounding; the table gives the
combined error. An enclosure midpoint can be produced without deciding
whether an arbitrary real observation is exactly on a rounding boundary.

The representation of the pooled radius must be a deterministic
measurable function of the first-coordinate observations alone.
Clipping, inward endpoint
approximation,
and the \(\psi\)-weight oracle must then use only that represented
radius, the known counts, and the fixed known functions. They may not
inspect the score coordinates \(\bar Y,\bar B\) or their noises.
Coordinatewise representation and rounding meet this contract. It
ensures that conditioning on the represented radius leaves the
score-coordinate noises independent standard Gaussians, as required in
Section \ref{app:s3-6}. A bound on representation error alone would not justify
that conditional independence if an arbitrary data-dependent
representation could inspect other coordinates.

Use deterministic dyadic rounding within the displayed tolerances,
denoting rounded means by \(X^\#,U^\#,V^\#,Y^\#,B^\#,R^\#\). Set 
\[
r^\#=\operatorname{clip}_{[-a,a]}R^\#,
\qquad \widetilde b=\text{the finite fit (S3.7) for }U^\#,V^\#.
\]
 Because \(\sqrt{x+\tau}\le\sqrt{x}+\sqrt\tau\), the objective
certificate and triangle inequality imply 
\begin{equation}
\|C(\widetilde b)-C(b)\|
\le2\|(\bar U,\bar V)-C(b)\|
 +2\sqrt2/(64n)+1/(8n)
\le2\|(\bar U,\bar V)-C(b)\|+1/n. \tag{S3.8}
\label{eq:s3-8}
\end{equation}

For a rational screening threshold, let \(j_n=\lceil\log_2 n\rceil\) and
compute a dyadic upper square-root approximation \(\delta_n\) to
\(\sqrt{16j_n/n}\), with additive error at most \((64n)^{-1}\). Then 
\[
\delta_n\ge\sqrt{16\log n/n},\qquad
\delta_n=O(\sqrt{\log n/n}).
\]
 If \(\operatorname{dist}(\widetilde b,\{-1,1\})>8\delta_n\), take
\(\widehat b=\widetilde b\). Otherwise compute 
\[
q^\#\approx\sqrt{\operatorname{clip}_{[1/4,4]}U^\#},\quad
p^\#\approx\psi(r^\#),\quad c^\#=1-r^\#,
\]
 at the declared precisions, keeping \(p^\#\ge0\) and
\(q^\#\in[1/2,2]\). Define 
\[
D^\#=n(p^\#)^2+m(c^\#)^2.
\]
 If \((D^\#)^2\le N^{-1}\), take the branch sign to be \(+1\). This
deliberately permits a bounded-risk decision when the score weights are
too small to require discrimination. Otherwise use the exact rational
score 
\[
A^\#=np^\#Y^\#+mc^\#B^\#,
\qquad \operatorname{sign}(0)=1,
\]
 omitting the resolver term when \(m=0\), and set
\(\widehat b=q^\#\operatorname{sign}(A^\#)\). Finally obtain
\(h^\#\approx\chi(r^\#)\), clipped into \([0,1]\), and return 
\begin{equation}
\widehat t=\operatorname{clip}_{[-2,2]}\{X^\#+\widehat b h^\#\}. \tag{S3.9}
\end{equation}
 Every root operation is finite, every tie has a convention, and every
clipping operation is explicit. With a deterministic measurable value
oracle, all operations define a measurable estimator. In the explicit
families below the oracle is constructed using finite rational
enclosures, so no such measurability assumption is hidden in a
nonconstructive function evaluation.

\subsection{Regular error and branch localization}
\label{app:s3-5}

On 
\[
\mathcal E_C=\{\|(\bar U,\bar V)-C(b)\|\le\delta_n\},
\]
 the norm in \eqref{eq:s3-8} is at most \(3\delta_n\), for sufficiently large
\(n\). Moreover 
\[
\Pr(\mathcal E_C^c)
=\exp(-n\delta_n^2/2)\le n^{-8},
\]
 because the squared norm of two independent standard Gaussian
coordinates has the \(\chi^2_2\) distribution. All sufficiently large
thresholds in this proof may depend on the fixed functions; bounded
smaller \(n\) are absorbed by the bounded target loss.

On a trusted local fit, \eqref{eq:s3-4} excludes \(|\widetilde b-b|>1\): it would
put \(\widetilde b\) within \(6\delta_n\) of one of \(\pm1\),
contradicting the screening decision. Equation \eqref{eq:s3-3}, followed by
\eqref{eq:s3-8}, therefore bounds its squared error by a constant times the
actual squared Gaussian noise plus \(n^{-2}\). Its expectation is
\(O(n^{-1})\). The larger screening radius is not used as the final
estimation error.

On the remote branch, \(\widetilde b\) is within \(8\delta_n\) of
\(\pm1\); \eqref{eq:s3-8} and its first coordinate imply 
\[
\bigl||b|-1\bigr|\le C\delta_n.
\]
 In particular, for large enough \(n\), \(b^2\in[1/4,4]\). Clipping is
contractive and the square root is Lipschitz on that fixed interval,
hence 
\begin{equation}
|q^\#-|b||\le C|\bar U-b^2|+C/n. \tag{S3.10}
\end{equation}
 Its squared expectation is again \(O(n^{-1})\).

Clipping the final target estimate cannot increase loss. The target
error can be decomposed into the error of \(X^\#\), the local or
magnitude error multiplied by \(\chi(r)\), the branch-sign error
multiplied by at most \(4\chi(r)\), and the radius and
\(\chi\)-evaluation errors. Since \(\chi\) is bounded and Lipschitz,
\(|\widehat b|\le2\), and 
\[
\mathbb E|r^\#-r|^2\le C/N,
\]
 this yields 
\begin{equation}
\mathbb E_\theta(\widehat t-t)^2
\le C/n+C\chi(r)^2
\Pr_\theta\{\text{wrong remote sign on the good events}\}. \tag{S3.11}
\label{eq:s3-11}
\end{equation}
 The exceptional-event probabilities are included in \(C/n\). This is
the step that excludes an unnecessary \(\log n/n\) term.

\subsection{Finite-precision weights preserve the statistically consequential tail}
\label{app:s3-6}

A fixed nonnegative \(C^2\) extension gives the Glaeser inequality \cite{Glaeser1963}
\(|\psi'(r)|^2\le C\psi(r)\) on the compact parameter interval. Here is
an explicit verification. Choose \(h_0>0\) so that every displacement of
magnitude at most \(h_0\) from the interval stays in a compact
nonnegative extension neighborhood. Let \(M\ge1\) bound \(|\psi''|\)
there and let \(D_{\max}\) bound \(|\psi'|\). Put \(D=\psi'(r)\). If
\(|D|\le Mh_0\), Taylor's inequality at displacement \(-D/M\), together
with
nonnegativity,
gives \(0\le\psi(r)-D^2/(2M)\). If \(|D|>Mh_0\), use displacement
\(-h_0\operatorname{sign}D\); it gives
\(\psi(r)\ge |D|h_0-Mh_0^2/2\ge |D|h_0/2\), hence
\(D^2\le(2D_{\max}/h_0)\psi(r)\). These two cases prove the inequality
with a fixed constant. Taylor expansion and bounded second derivative
consequently give 
\begin{equation}
|\psi(s)-\psi(r)|\le C\sqrt{\psi(r)}|s-r|+C|s-r|^2. \tag{S3.12}
\label{eq:s3-12}
\end{equation}

Consider the independent radius event 
\[
\mathcal E_R=
\{|\bar R-r|\le\sqrt{16\log n/N}\},
\qquad \Pr(\mathcal E_R^c)\le2n^{-8}.
\]
 Let 
\[
v=(\sqrt n\psi(r),\sqrt m(1-r)),
\qquad w=(\sqrt n p^\#,\sqrt m c^\#),
\qquad H=\|v\|^2,
\]
 with zero-count components omitted. The rounding errors and \eqref{eq:s3-12}
give, uniformly over all allocations on \(\mathcal E_R\), 
\begin{equation}
E:=\|w-v\|^2
\le C\left\{
\frac{\log n}{\sqrt N}\sqrt H+
\frac{H\log n}{N}+
\frac{\log^2 n+1}{N}\right\}. \tag{S3.13}
\label{eq:s3-13}
\end{equation}
 For clarity, the terms before simplification are bounded by 
\[
C\{n\psi(r)d^2+nd^4+md^2+nN^{-2}\},
\qquad d=|r^\#-r|\le C\sqrt{\log n/N}.
\]
 Use \(n\psi(r)\le\sqrt{nH}\), \(m\le H/(1-a)^2\) and \(n\le N\) to
obtain \eqref{eq:s3-13}. The oracle contributes only \(nN^{-2}\). No relative
error of \(\psi\) appears.

If the small-score fallback is used, \(\|w\|\le N^{-1/4}\). On
\(\mathcal E_R\), this forces \(H<1\) for sufficiently large \(n\).
Otherwise \eqref{eq:s3-13} gives \(E/H=o(1)\) uniformly for \(H\ge1\), so
\(\|w\|\ge\sqrt H-\sqrt E\ge1/2\), a contradiction. In the fallback
region, therefore, 
\begin{equation}
\chi(r)^2\Pr\{\text{wrong sign}\}
\le\chi(r)^2
\le\Phi(-1)^{-1}\chi(r)^2\Phi(-\sqrt H). \tag{S3.14}
\label{eq:s3-14}
\end{equation}

For a nonzero weight vector, first use the unrounded score means.
Conditional on the rounded radius, the normalized score is a
unit-variance real normal variable, with signed mean 
\[
\rho=|b|\frac{w\cdot v}{\|w\|}\ge0.
\]
 All components are nonnegative, since the oracle output has been
clipped nonnegative and \(1-r^\#>0\). Orthogonal projection onto the
line spanned by \(w\) gives 
\begin{equation}
\rho^2
=b^2\frac{(w\cdot v)^2}{\|w\|^2}
\ge b^2(H-E). \tag{S3.15}
\end{equation}
 This inequality remains true if its right side is negative. Rounding
the score data changes its normalized value by at most 
\begin{equation}
\frac{n p^\#|Y^\#-\bar Y|
+m c^\#|B^\#-\bar B|}{\|w\|}
\le\frac1{64\sqrt N}. \tag{S3.16}
\label{eq:s3-16}
\end{equation}
 The bound follows from Cauchy--Schwarz and does not deteriorate when
\(\|w\|\) is small. Additional arithmetic error bounded by \(N^{-1/2}\)
merely changes the constant. Thus the computed wrong-sign probability is
bounded above by \(\Phi(-\rho+C/\sqrt N)\).

Fix a large constant \(K_0\). On the remote branch and the good events,
\(b^2=1+O(\delta_n)\). For \(H\le K_0\log n\), \eqref{eq:s3-13} implies
\(E=o(1)\), and \(\delta_nH=o(1)\). Therefore 
\begin{equation}
\rho^2\ge H-o(1),\qquad
(\rho-C/\sqrt N)_+^2\ge H-o(1). \tag{S3.17}
\end{equation}
 The notation is uniform in the parameter and all resolver allocations
in this region. The elementary Mills bounds 
\[
\Phi(-\sqrt x)\asymp (1+x)^{-1/2}e^{-x/2},\qquad x\ge0,
\]
 show that a bounded additive change of \(x\) changes this tail by at
most a fixed factor. Hence the computed wrong-sign probability is at
most \(C\Phi(-\sqrt H)\). More precisely, if \(H\to\infty\) within this
region, then \(\rho^2\le b^2H=H+o(1)\) as well; sandwiching the rounding
error by both signs in \eqref{eq:s3-16} gives conditional error divided by
\(\Phi(-\sqrt H)\) tending to one. This is a conditional branch-testing
statement, not a ratio-one assertion for the full minimax risk.

For \(H>K_0\log n\), \eqref{eq:s3-13} gives \(E/H=o(1)\) uniformly and
\(b^2=1+o(1)\). Thus, for all sufficiently large \(n\),
\(\rho\ge\sqrt{H/2}\), and the additional \(O(N^{-1/2})\) threshold
shift is negligible. Increasing \(K_0\), if necessary, makes the
wrong-sign probability at most \(n^{-3}\). This contribution is absorbed
by the regular floor. The argument covers arbitrarily large \(m\); no
subsampling or uncounted radius observations are needed.

Combining \eqref{eq:s3-11}, \eqref{eq:s3-14}, and the two information regions gives 
\[
\sup_{\theta\in K}\mathbb E_\theta(\widehat t-t)^2
\le C/n+C\sup_{r\in[-a,a]}\chi(r)^2\Phi(-\sqrt{H_{n,m}(r)}).
\]
 Evenness of the fixed functions and \((1-r)^2\le(1+r)^2\) for
\(r\ge0\) reduce the supremum exactly to \([0,a]\). This proves the
upper bound in \eqref{eq:s3-1}, including all counts after the bounded-count
convention.

The safe normalized score tolerance is consequently far more permissive
than a relative-accuracy requirement on a flat weight: for a
bounded-factor tail comparison in the consequential region, any extra
deterministic normalized score error \(\zeta_n\) with
\(\zeta_n\sqrt{\log n}=O(1)\) suffices. The stronger condition
\(\zeta_n\sqrt{\log n}=o(1)\) retains the conditional ratio-one tail
there. The specified \(O(N^{-1/2})\) bound satisfies the stronger
condition. These observations concern additional score arithmetic after
the weight bounds already proved above.

This dependence cannot be replaced by an unrestricted assertion that a
small absolute score error is harmless. For the simple normal decision
with signed mean \(\sqrt H\), an adverse threshold displacement
\(\zeta>0\) changes the error from \(\Phi(-\sqrt H)\) to
\(\Phi(-\sqrt H+\zeta)\). When \(H\to\infty\) and \(\zeta=o(\sqrt H)\),
Mills' ratio gives 
\begin{equation}
\frac{\Phi(-\sqrt H+\zeta)}{\Phi(-\sqrt H)}
=(1+o(1))\frac{\sqrt H}{\sqrt H-\zeta}
\exp\{\zeta\sqrt H-\zeta^2/2\}. \tag{S3.18}
\end{equation}
 Thus \(\zeta\sqrt H\to\infty\) produces an unbounded error inflation
even if \(\zeta\to0\). This is an exact binary precision obstruction. It
establishes the necessity of controlling the score-error product at the
tail resolution being used, without claiming that every full-estimator
construction requires the conservative table entries above.

\subsection{Computational model and explicit flat-family oracle}
\label{app:s3-7}

For arbitrary fixed \(C^2\) functions, finite computability is an
additional assumption: smoothness alone supplies no algorithm for
evaluating a function. The oracle formulation asks for deterministic
measurable absolute-error enclosures of \(\psi(x)\) and \(\chi(x)\) at a
represented radius, and an effective representation of the fixed
endpoint \(a\). An inward rational approximation of \(a\) with error
\(O(N^{-1})\) can be used for clipping; it adds only the same order to
the radius error already allowed above. The explicit construction takes
rational \(a\). The model also assumes that sufficient sample means can
be obtained to the stated absolute errors. Real Gaussian observations
are mathematically real numbers; exact access to their entire expansions
is neither assumed nor required. Rounding each raw observation to the
desired mean tolerance, followed by exact summation, is one finite input
model. The count of raw observations must still be paid, even though the
estimation stage uses only sufficient means.

For rational \(a\), positive integer \(p\), and positive rational
coefficients,
the following construction supplies enclosures for 
\[
f_c(r)=\exp(-c|r|^{-p}),
\qquad g_c(r)=\exp\{-c\exp(1/r^2)\},
\]
 with value zero at \(r=0\). The Gevrey family uses \(\psi=f_1\),
\(\chi=f_\alpha\); the double-exponential family uses \(\psi=g_{1/2}\),
\(\chi=g_{\alpha/2}\). The latter evaluates \(\chi\) directly, avoiding
the numerical operation of raising an underflowed \(\psi\) to a power.
Any fixed Gevrey order greater than one is covered by choosing an
integer \(p\) large enough. Irrational family constants require their
own effective representation; they are not silently treated as exact
rational input.

Here is the enclosure argument. For rational \(x\ge0\), the positive
exponential Taylor sum through degree \(k\) has next term
\(a_{k+1}=x^{k+1}/(k+1)!\). If \(x/(k+2)<1\), its remaining tail is
bounded by 
\[
\frac{a_{k+1}}{1-x/(k+2)}.
\]
 All quantities are rational. Increasing \(k\) terminates at every
prescribed positive accuracy. Inversion of a positive exponential
enclosure gives an enclosure of \(e^{-x}\), because the positive
exponential is at least one. If accuracy \(2^{-P}\) is requested and
\(x\ge P\), the interval \([0,2^{-P}]\) already encloses \(e^{-x}\),
since \(e>2\).

For \(g_c\), set \(x=1/r^2\). If \(x\ge s\), where the nonnegative
integer \(s\) satisfies \(2^s\ge P/c\), then \(ce^x\ge P\), and the same
small-value enclosure applies before forming the inner exponential.
Otherwise \(x=O(\log P)\), for fixed \(c\). Enclose \(e^x\) rationally,
multiply by \(c\), and apply the outer negative-exponential bounds to
both endpoints. The derivative of \(e^{-z}\) on \(z\ge0\) has magnitude
at most one, so the inner absolute-error budget transfers directly. This
proves the finite oracle guarantee even arbitrarily near the flat point.
Correctly deciding a relative sign of an astronomically small value is
unnecessary.

The post-observation polynomial fit has a fixed degree. If the rounded
rational inputs have numerator and denominator bit lengths at most
\(B\), clearing denominators gives integer coefficients of bit length
\(O(B)\). Exact square-free reduction and Sturm preprocessing require
finitely many constant-degree rational operations with bit lengths
polynomial in \(B\). A nonzero integer discriminant of the square-free
polynomial, together with the elementary root bound \(|z|\le1+H\) for
coefficient height \(H\), bounds distinct-root separation below by
\(2^{-C B}\), for a fixed degree-dependent \(C\). Indeed the
discriminant is an integer of magnitude at least one, while all but one
root-difference factors and the leading-coefficient factor have
magnitude at most powers of \(1+H\). The same evaluation argument
separates non-root rational endpoints. Consequently isolation to \eqref{eq:s3-6}
requires depth \(O(B+\log n)\). At most five nonempty root intervals
survive each depth, so the number of bisection and sign-evaluation
stages is \(O(B+\log n)\). Exact rational arithmetic makes the bit cost
polynomial in \(B+\log n\); no specific optimal bit exponent is claimed.

The flat oracle has finite rational Taylor cost with early clipping; for
fixed represented family parameters its working accuracy is
\(P=O(\log N)\) bits. The construction does not claim optimal
exponentials,
optimal streaming arithmetic, or a run time uniform in unbounded
observed data magnitudes. Its principal certificate is the deterministic
statistical tolerance \eqref{eq:s3-7}, \eqref{eq:s3-8}, \eqref{eq:s3-13}, and \eqref{eq:s3-16}, rather than
a practical-speed benchmark. Arbitrary \(C^2\) oracles retain their
declared, possibly unbounded, evaluation costs.

The construction above is finite under the stated effective-function
assumptions. Its proof establishes the risk comparison and the
conditional rare-error accuracy directly; no numerical experiment is
needed for either conclusion.

\newpage

\section{Exact orientation calibration and uniform asymptotics}
\label{app:s4}

\subsection{Fixed experiment and exact invariant minimax reduction}
\label{app:s4-1}

Fix an integer \(d\ge1\). Let \(u\in S^{d-1}\) be a common unknown
direction and let \(b\in\{0,1\}\) be the target. For known real
\(R,k\ge0\), observe independently conditional on \((b,u)\)

\[
A\sim N_d(\sqrt{k}\,u,I_d),\qquad B\sim N_d(bRu,I_d).
\]

Here \(k\) is reference information and \(R^2\) is detection
information; integer replication of fixed unit-covariance sources is a
special case. Define

\[
e_d(R,k)=\inf_\varphi\max\left\{\sup_uP_{0,u}(\varphi=1),\sup_uP_{1,u}(\varphi=0)\right\}.
\]

Tests may initially be randomized; the optimal test below is
deterministic for \(R>0\). Let \(\mu_d\) denote normalized Haar measure
on the sphere and put

\[
Z_d(t)=\int_{S^{d-1}}e^{tu_1}\,d\mu_d(u),\qquad h=d-1.
\]

For \(d=1\), \(Z_1(t)=\cosh t\). For \(d\ge2\), the standard Bessel representation \cite[Eq.~10.32.2]{DLMF} is

\[
Z_d(t)=\Gamma(d/2)(2/t)^{d/2-1}I_{d/2-1}(t),\qquad Z_d(0)=1.
\]

The two Haar-mixture densities, relative to independent standard
Gaussian density \(\phi_d(a)\phi_d(b')\), are

\[
p_0=\phi_d(a)\phi_d(b')e^{-k/2}Z_d(\sqrt{k}\,|a|),
\]

\[
p_1=\phi_d(a)\phi_d(b')e^{-(k+R^2)/2}Z_d(|\sqrt{k}\,a+Rb'|).
\]

Consequently their exact likelihood ratio is

\begin{equation}
L_{R,k}(a,b')=e^{-R^2/2}\frac{Z_d(|\sqrt{k}\,a+Rb'|)}{Z_d(\sqrt{k}\,|a|)}. \tag{S4.1}
\label{eq:s4-1}
\end{equation}

\begin{theorem}
For \(R>0\), a likelihood-ratio threshold
\(\varphi=1\{L_{R,k}>\tau_{R,k}\}\) with equalized mixture errors is
minimax for the full nuisance experiment. Its two errors are the same
for every \(u\) in their respective classes. Such an equalizing
threshold exists. For \(R=0\), \(e_d(0,k)=1/2\).
\end{theorem}

\textbf{Proof.} Average any test over the compact orthogonal group
acting simultaneously on both vectors. Its worst risk does not increase.
The resulting test is invariant and therefore has constant error on each
of the two transitive nuisance orbits. Its risks equal the risks under
the two Haar mixtures. For the simple mixture pair, an equal-error
likelihood-ratio test is Bayes for a suitable binary prior and hence is
minimax. The ratio in \eqref{eq:s4-1} is continuous and nonconstant, and its
distribution has no atoms on its nonconstant range: it is a nonconstant
real analytic function of Gaussian
observations,
whose level sets have Lebesgue measure zero. The two errors vary
continuously from \((1,0)\) to \((0,1)\) as the threshold traverses the
ratio range. This proves existence. At \(R=0\), the distributions
coincide; the deterministic decision \(1\{B_1>0\}\) has error \(1/2\)
under every direction.

The least squared distance between the two full mean sets equals \(R^2\)
for \textbf{every} \(k\):

\[
\inf_{u,v}\{k|u-v|^2+R^2\}=R^2.
\]

The calibration effect below therefore cannot be recovered from that
scalar minimum-distance profile. The complete labelled distance arrays
change with \(k\); only their scalar minimum is held fixed in this
comparison.

\subsection{The calibration affinity}
\label{app:s4-2}

Let \(P_k=N_d(\sqrt{k}\,e_1,I_d)\) and let
\(\overline P_k=\int N_d(\sqrt{k}\,u,I_d)d\mu_d(u)\). Define their
Hellinger affinity, without squaring it, by

\[
M_d(k)=\int\sqrt{dP_k\,d\overline P_k}.
\]

Two exact formulas are

\begin{equation}
M_d(k)=e^{-k/2}\,E_{Z\sim N_d(0,I_d)}\left[e^{\sqrt{k}Z_1/2}\sqrt{Z_d(\sqrt{k}|Z|)}\right], \tag{S4.2}
\label{eq:s4-2}
\end{equation}

\begin{equation}
M_d(k)=e^{-k/2}\int_0^\infty f_{\chi_d}(q)\sqrt{Z_d(\sqrt{k}q)}Z_d(\sqrt{k}q/2)\,dq. \tag{S4.3}
\label{eq:s4-3}
\end{equation}

Write

\[
c_d=\frac{(2\pi)^{h/2}}{|S^{d-1}|}=\frac{2^{(d-3)/2}\Gamma(d/2)}{\sqrt\pi},\qquad
K_d=\frac{2^{2-d}}{\pi^{1/4}\sqrt{\Gamma(d/2)}}.
\]

Then \(M_d\) is positive, continuous and
nonincreasing,
\(M_d(0)=1\), and

\begin{equation}
M_d(k)\sim 2^{h/2}\sqrt{c_d}\,k^{-h/4}\quad(k\to\infty). \tag{S4.4}
\label{eq:s4-4}
\end{equation}

In particular \(M_d(k)\asymp_d(1+k)^{-h/4}\) for all \(k\ge0\).

\textbf{Proof.} The density formulas directly give \eqref{eq:s4-2}; polar
integration gives \eqref{eq:s4-3}. Positivity and continuity follow by Gaussian
domination on compact \(k\) ranges. If \(k_1\le k_2\), the Gaussian
channel \(X\mapsto\sqrt{k_1/k_2}X+\sqrt{1-k_1/k_2}Z\) maps both members
of the pair at \(k_2\) to the corresponding pair at \(k_1\).
Cauchy--Schwarz applied to the channel kernel proves that affinity
cannot decrease under this channel, giving monotonicity.

The classical spherical Laplace expansion \cite[Eq.~10.40.1]{DLMF} is

\begin{equation}
Z_d(t)=c_de^tt^{-h/2}(1+O_d(t^{-1}))\quad(t\to\infty). \tag{S4.5}
\label{eq:s4-5}
\end{equation}

It follows either from the displayed Bessel representation or directly
by writing a hemisphere locally as \((\sqrt{1-|w|^2},w)\) and scaling
\(w=t^{-1/2}v\). For \(d=1\) it follows from
\(\cosh t=(e^t/2)(1+e^{-2t})\). The same hemisphere decomposition also
proves the global bounds

\begin{equation}
c(1+t)^{-h/2}e^t\le Z_d(t)\le C(1+t)^{-h/2}e^t,\qquad t\ge0. \tag{S4.6}
\label{eq:s4-6}
\end{equation}

Substitute \eqref{eq:s4-5} twice into \eqref{eq:s4-3}. Around its unique radial saddle
\(q=\sqrt{k}\), the integrand becomes

\[
c_{\chi,d}c_d^{3/2}2^{h/2}k^{-3h/8}q^{h/4}e^{-(q-\sqrt{k})^2/2}(1+o(1)),
\]

where \(c_{\chi,d}=2^{1-d/2}/\Gamma(d/2)\) and
\(f_{\chi_d}(q)=c_{\chi,d}q^he^{-q^2/2}\). The identity
\(c_{\chi,d}c_d\sqrt{2\pi}=1\) yields \eqref{eq:s4-4}. To justify the saddle
step, first restrict \(|q-\sqrt{k}|\le k^{1/8}\), where the displayed
expansion is uniform; \eqref{eq:s4-6} bounds the complementary integral by a
polynomial in \(k,q\) times \(e^{-(q-\sqrt{k})^2/2}\). The part
\(q\le\sqrt{k}/2\) has an exponentially small bound and the remaining
Gaussian tails are negligible relative to \(k^{-h/4}\). Compact-range
continuity completes the global comparison.

For \(d=1\), the limit in \eqref{eq:s4-4} is \(1/\sqrt2\); this is a discrete
two-direction nuisance and does not create an unbounded calibration
factor.

\subsection{Uniform sharp calibration theorem}
\label{app:s4-3}

\setcounter{theorem}{2}
\begin{theorem}
For fixed \(d\ge1\), as \(R\to\infty\),

\begin{equation}
e_d(R,k)=\{1+o(1)\}K_dM_d(k)R^{(d-3)/2}e^{-R^2/8}
\left(1+\frac{2k}{R^2}\right)^{(d-1)/4},
\tag{S4.7}
\label{eq:s4-7}
\end{equation}

where the relative \(o(1)\) is uniform over \textbf{all} \(k\ge0\). The
likelihood-ratio threshold satisfies \(\tau_{R,k}\to1\) uniformly. Thus
the equal-prior Haar-mixture likelihood-ratio test is asymptotically
minimax with relative error tending to one uniformly in the reference
allocation.

Equivalently, in
the joint regime \(R\to\infty\), \(k\to\infty\), with no constraint on
their relative growth,

\begin{equation}
e_d(R,k)\sim\Phi(-R/2)
\left(1+\frac{R^2}{2k}\right)^{(d-1)/4}.
\tag{S4.8}
\label{eq:s4-8}
\end{equation}

At bounded \(k\), \eqref{eq:s4-7} retains the exact affinity multiplier
\(M_d(k)\); replacing it by a power of \(1+k\) would lose the exact
constant.
\end{theorem}

\subsubsection*{Proof preparation: uniform tail bounds}
\addcontentsline{toc}{subsubsection}{Proof preparation: uniform tail bounds}

For a threshold \(\tau\) in a fixed compact subset of \((0,\infty)\),
write \(\alpha_{R,k}(\tau)=P_0(L>\tau)\) and
\(\beta_{R,k}(\tau)=P_1(L\le\tau)\). Both probabilities here are under
the Haar mixtures.

The elementary inequality \((\log Z_d)'\le1\) and the triangle
inequality give

\begin{equation}
L(a,b')\le\exp\{R|b'|-R^2/2\}. \tag{S4.9}
\label{eq:s4-9}
\end{equation}

Under \(P_0\), \(B\) is independent of \(A\) and has standard Gaussian
law.
Consequently,
uniformly in the realized \(A=a\) and in \(k\),

\begin{equation}
P_0(L>\tau\mid A=a)\le C R^{d-2}e^{-R^2/8}, \tag{S4.10}
\label{eq:s4-10}
\end{equation}

for all sufficiently large \(R\). This follows because \eqref{eq:s4-9} requires
\(|B|\ge R/2+(\log\tau)/R\).

The corresponding conditional type-II contribution is

\begin{equation}
E_0[L1\{L\le\tau\}\mid A=a]\le C R^{d-2}e^{-R^2/8}. \tag{S4.11}
\label{eq:s4-11}
\end{equation}

For \(r=|B|\le R/2\), bound the integrand by \(e^{Rr-R^2/2}\); its
radial integral is bounded by \(C R^{d-2}e^{-R^2/8}\) using the endpoint
\(r=R/2\) of \(r^he^{-(R-r)^2/2}\). For \(r>R/2\), bound it by \(\tau\)
and use the central-chi tail. Thus \eqref{eq:s4-11} does not assume that type-II
error follows from a radial event under the alternative.

The \(A\)-radius has law \(\chi_d(\sqrt{k})\) under both hypotheses.
Since it can be represented as \(|\sqrt{k}e_1+Z|\), the event

\[
\big||A|-\sqrt{k}\big|>T
\]

has probability at most \(P(|Z|>T)\le C(1+T)^de^{-T^2/2}\). Choose
\(T=\sqrt{C_0\log R}\), with \(C_0\) as large as needed. Equations
\eqref{eq:s4-10}--\eqref{eq:s4-11}
then let us discard this event at relative cost \(o(1)\) compared even
with the known-direction lower scale \(R^{-1}e^{-R^2/8}\). In the
small-\(k\) regime it suffices to impose only \(|A|\le\sqrt{k}+T\).

\subsubsection*{\texorpdfstring{Proof, regime I: \(0\le k\le R^{1/4}\)}{Proof, regime I: 0 <= k <= R to the power 1/4}}
\addcontentsline{toc}{subsubsection}{\texorpdfstring{Proof, regime I: \(0\le k\le R^{1/4}\)}{Proof, regime I: 0 <= k <= R to the power 1/4}}

Put \(B=rv\), where \(v\) is uniform on the sphere independently of
\(A\) under \(P_0\), and write \(t=\sqrt{k}|A|\). On
\(|A|\le\sqrt{k}+T\), uniformly for \(R/4\le r\le3R/4\),

\[
|\sqrt{k}A+Rrv|=Rr+\sqrt{k}A\cdot v+O\left(\frac{k|A|^2}{R^2}\right).
\]

The remainder is \(o(1)\) uniformly in this regime. From \eqref{eq:s4-5},

\begin{equation}
\log L=R(r-R/2)-h\log R+\log C_k(A,v)+o(1), \tag{S4.12}
\end{equation}

at every threshold neighborhood used below, where

\[
C_k(A,v)=c_d2^{h/2}\frac{e^{\sqrt{k}A\cdot v}}{Z_d(t)}.
\]

The relevant radial threshold is

\begin{equation}
r_\tau=\frac R2+\frac{h\log R+\log\tau-\log C_k(A,v)+o(1)}R. \tag{S4.13}
\label{eq:s4-13}
\end{equation}

Indeed \(|\log C_k|\le C+2t+C\log(1+t)\) and hence the numerator is
\(O(R^{1/4}+\sqrt{\log R}\,R^{1/8}+\log R)\). Its square divided by
\(R^2\) tends uniformly to zero. The ratio is increasing in \(r\ge R/4\)
because \(Rr>\sqrt{k}|A|\) and \(Z_d\) is increasing; below \(R/4\),
\eqref{eq:s4-9} excludes a threshold crossing for sufficiently large \(R\).

The central-chi tail at \eqref{eq:s4-13} gives, uniformly on this set,

\begin{equation}
P_0(L>\tau\mid A,v)=\{1+o(1)\}
2c_{\chi,d}2^{-h}R^{h/2-1}e^{-R^2/8}\tau^{-1/2}\sqrt{C_k(A,v)}. \tag{S4.14}
\label{eq:s4-14}
\end{equation}

For the conditional type-II integral, put \(r=r_\tau-s/R\) below the
threshold. The logarithm of its integrand relative to its endpoint tends
to \(-s/2\) uniformly on bounded \(s\). It is bounded by \(Ce^{-s/4}\)
until \(r=R/4\): the derivative of \(\log L+\log f_{\chi_d}\) is at
least \(R/4\) there for sufficiently large \(R\). The remaining interval
\(r<R/4\) has the exponentially smaller bound from \eqref{eq:s4-9}. Endpoint
integration therefore gives the same right side of \eqref{eq:s4-14}, with
\(\tau^{-1/2}\) replaced by \(\tau^{1/2}\).

The discarded \(A\)-tail is harmless also in the affinity integral: by
\eqref{eq:s4-6}, \(\sqrt{C_k(A,v)}\le C(1+\sqrt{k}|A|)^{h/4}\), and its tail
expectation is bounded by a polynomial in \(k,T\) times \(e^{-T^2/4}\).
Equation \eqref{eq:s4-4}, continuity, and a sufficiently large \(C_0\) make this
\(o(M_d(k))\) uniformly for \(k\le R^{1/4}\).

Finally, integrating over \(A,v\) gives

\[
E_0\sqrt{C_k(A,v)}=\sqrt{c_d}\,2^{h/4}M_d(k),
\]

and \(2c_{\chi,d}2^{-h}\sqrt{c_d}2^{h/4}=K_d\). We obtain, uniformly for
\(0\le k\le R^{1/4}\) and compact positive \(\tau\),

\[
\alpha_{R,k}(\tau)\sim\tau^{-1/2}K_dM_d(k)R^{h/2-1}e^{-R^2/8},
\]

\begin{equation}
\beta_{R,k}(\tau)\sim\tau^{1/2}K_dM_d(k)R^{h/2-1}e^{-R^2/8}. \tag{S4.15}
\label{eq:s4-15}
\end{equation}

\subsubsection*{\texorpdfstring{Proof, regime II: \(k\ge R^{1/4}\)}{Proof, regime II: k >= R to the power 1/4}}
\addcontentsline{toc}{subsubsection}{\texorpdfstring{Proof, regime II: \(k\ge R^{1/4}\)}{Proof, regime II: k >= R to the power 1/4}}

Use polar coordinates for both vectors. Let \(q=|A|\), \(r=|B|\), and
let \(\theta\in[0,\pi]\) be their angle. Under \(P_0\), these are
independent with radial laws \(\chi_d(\sqrt{k})\), \(\chi_d\), and
uniform spherical angular law. Write

\[
w_0=k+R^2/2,\qquad D=\frac{kR^2}{4w_0},\qquad a_0=\frac h2\log(w_0/k).
\]

Here \(D\asymp\min(k,R^2)\ge cR^{1/4}\). With \(q=\sqrt{k}+z\),
\(r=R/2+\delta/R\), the main region will satisfy

\begin{equation}
|z|\le T,\qquad \sqrt D\,\theta\le T,\qquad |\delta|\le C_1\log R. \tag{S4.16}
\label{eq:s4-16}
\end{equation}

On this region, Taylor's formula for the exact norm

\[
s=\sqrt{t^2+R^2r^2+2tRr\cos\theta},\qquad t=\sqrt{k}q,
\]

and \eqref{eq:s4-5} give uniformly

\begin{equation}
\log L=\delta-a_0-D\theta^2+o(1). \tag{S4.17}
\label{eq:s4-17}
\end{equation}

For clarity, the errors are bounded by constants times

\[
\frac{(\log R)^{3/2}}{\sqrt{k}}+
\frac{(\log R)^2}{\min(k,R^2)}+
\frac{(\log R)^2}{R^2}+\frac1k,
\]

which tend uniformly to zero. In particular this is a rare-event
expansion, not an ordinary central-limit approximation. The angle
remainder is \(O(D\theta^4)\), not \(O(R^2\theta^4)\); using the latter
would incorrectly exclude slowly increasing \(k\).

Here are the localization details needed to justify \eqref{eq:s4-16}. The
discarded radial \(q\) tails were handled by
\eqref{eq:s4-10}--\eqref{eq:s4-11}.
On the remaining \(q\) range and \(r\in[R/4,3R/4]\), the exact identity

\[
t+Rr-s=\frac{2tRr(1-\cos\theta)}{t+Rr+s}
\]

and \eqref{eq:s4-6} imply

\begin{equation}
L\le C R^h\exp\{R(r-R/2)-cD\theta^2\}. \tag{S4.18}
\label{eq:s4-18}
\end{equation}

The polynomial factor is uniform in arbitrarily large \(k\): if
\(t>2Rr\), then \((1+t)/(1+s)\) is bounded; if \(t\le2Rr\), then
\(1+t\le CR^2\).

For fixed \(q,\theta\), use \(\delta=R(r-R/2)\) and

\[
f_{\chi_d}(r)dr\le C R^{h-1}e^{-R^2/8}e^{-\delta/2}d\delta.
\]

Integrating either a threshold indicator or \(\min\{\tau,L\}\) using
\eqref{eq:s4-18} bounds their radial contributions by

\[
C R^{3h/2-1}e^{-R^2/8}e^{-c'D\theta^2}.
\]

The excluded intervals \(r<R/4\) and \(r>3R/4\) contribute at most a
polynomial in \(R\) times \(e^{-9R^2/32}\), by \eqref{eq:s4-9} and the Gaussian
radial tail. Taking \(T^2=C_0\log R\) with sufficiently large \(C_0\)
therefore discards \(\sqrt D\theta>T\) at relative cost \(o(1)\). On the
resulting angular range, \eqref{eq:s4-18} bounds the negative-\(\delta\) error
integral by

\[
C R^{2h-1}e^{-R^2/8}\int_{-\infty}^{-C_1\log R}e^{\delta/2}d\delta
\le C R^{2h-1-C_1/2}e^{-R^2/8}.
\]

The positive-\(\delta\) error integral is bounded by
\(C R^{h-1-C_1/2}e^{-R^2/8}\) using the bounded threshold and the radial
density. The type-I indicator vanishes on the sufficiently negative part
by \eqref{eq:s4-18}. Choose \(C_1>4h+4\) and also large enough to contain
\(a_0+D\theta^2+\log\tau\), whose absolute value is at most
\((h+C_0)\log R+O(1)\). This explicitly localizes \(\delta\) to
\(|\delta|\le C_1\log R\) at relative cost \(o(1)\) even against the
known-direction scale. These estimates bound both errors separately,
because the type-II integrand is \(L1\{L\le\tau\}\le\min\{\tau,L\}\).

On \eqref{eq:s4-16},

\[
f_{\chi_d}(r)dr=\{1+o(1)\}c_{\chi,d}2^{-h}R^{h-1}e^{-R^2/8}e^{-\delta/2}d\delta.
\]

For \(d\ge2\), integrating the spherical angular measure near zero gives

\[
\int e^{-D\theta^2/2}d\mu_d(\theta)\sim c_dD^{-h/2}.
\]

This follows by the ordinary Gaussian integral in the \(h\) tangent
coordinates. For \(d=1\), the aligned direction has mass \(1/2=c_1\),
the opposite direction is discarded by \eqref{eq:s4-18}, and the same formula
with \(h=0\) holds.

Using \eqref{eq:s4-17}, integration in \(\delta\) above or below
\(a_0+D\theta^2+\log\tau\) yields respectively
\(2\tau^{-1/2}e^{-a_0/2}e^{-D\theta^2/2}\) and
\(2\tau^{1/2}e^{-a_0/2}e^{-D\theta^2/2}\). The \(q\) integral tends to
one. The identity \(c_{\chi,d}c_d=(2\pi)^{-1/2}\) gives, uniformly for
\(k\ge R^{1/4}\),

\[
\alpha_{R,k}(\tau)\sim\tau^{-1/2}\frac2{\sqrt{2\pi}R}e^{-R^2/8}(w_0/k)^{h/4},
\]

\begin{equation}
\beta_{R,k}(\tau)\sim\tau^{1/2}\frac2{\sqrt{2\pi}R}e^{-R^2/8}(w_0/k)^{h/4}. \tag{S4.19}
\label{eq:s4-19}
\end{equation}

\subsubsection*{Proof completion}
\addcontentsline{toc}{subsubsection}{Proof completion}

The error equalizer lies in any prescribed compact interval around one,
for all sufficiently large \(R\), because \eqref{eq:s4-15} and \eqref{eq:s4-19} hold
uniformly at its two endpoints. They further force \(\tau_{R,k}\to1\)
uniformly. This proves the minimax asymptotic in each regime. In regime
I, \((1+2k/R^2)^{h/4}=1+o(1)\). In regime II, \eqref{eq:s4-4} is uniform for
\(k\ge R^{1/4}\) and converts \eqref{eq:s4-19} exactly into \eqref{eq:s4-7}. The
normal-tail asymptotic then gives \eqref{eq:s4-8}. This completes the proof over
every allocation.

\subsection{All-information comparison, budget laws, and calibrated interpretation}
\label{app:s4-4}

There are constants depending only on \(d\) such that for all
\(R,k\ge0\),

\begin{equation}
e_d(R,k)\asymp_d
\frac{e^{-R^2/8}}{1+R}
\left(1+\frac{R^2}{1+k}\right)^{(d-1)/4}.
\tag{S4.20}
\end{equation}

For large \(R\) this follows from \eqref{eq:s4-7} and
\(M_d(k)\asymp(1+k)^{-h/4}\). For bounded \(R\), the known-direction
two-point lower bound \(\Phi(-R/2)\) and the trivial upper bound \(1/2\)
give the comparison uniformly over \(k\).

Thus the displayed expression gives a bounded-factor comparison scale.
Its use of \(1+k\) appropriately regularizes bounded reference
information; it does not specify the exact constant supplied by \eqref{eq:s4-7}.

For every \(d\ge2\), as \(R\to\infty\):

\begin{itemize}
\setlength{\itemsep}{0pt}
\item
  Bounded \(k\) leaves the uncalibrated power \(R^{(d-3)/2}\), with
  exact multiplier \(M_d(k)\).
\item
  \(1\ll k\ll R^2\) reduces the error by an unbounded factor of order
  \(k^{(d-1)/4}\) but retains an unbounded penalty relative to a known
  direction.
\item
  \(k=\Omega(R^2)\) is necessary and sufficient for the error to be
  within a bounded factor of the known-direction error.
\item
  \(k/R^2\to\infty\) is necessary and sufficient for
  \(e_d(R,k)/\Phi(-R/2)\to1\).
\end{itemize}

The necessity in the last two statements follows directly from
\eqref{eq:s4-7}--\eqref{eq:s4-8},
including bounded \(k\). These statements concern information counts;
they do not equate a calibration observation to a scalar detector
observation unless the specified source normalizations make their
information equal.

If \(H=R^2\to\infty\) and a desired classification error is
\(O(n^{-1})\), write \(\ell=\log n\). The all-count comparison gives the
exact bounded-log criterion

\begin{equation}
\frac H8+\frac12\log(1+H)
-\frac{d-1}{4}\log\left(1+\frac H{1+k}\right)
\ge\ell-O(1).
\tag{S4.21}
\label{eq:s4-21}
\end{equation}

For \(k\asymp H^\rho\) with fixed \(0\le\rho\le1\), its critical
information window is

\begin{equation}
H=8\ell+\{2(d-1)(1-\rho)-4\}\log\ell+O(1), \tag{S4.22}
\end{equation}

provided the comparability constants implicit in \(k\asymp H^\rho\)
remain fixed. For \(k\gtrsim H\), the known-direction window
\(H=8\ell-4\log\ell+O(1)\) is recovered. Formula \eqref{eq:s4-21}, rather than a
power-law shorthand, covers irregular allocations and retains their
bounded-window effects.

For fixed dimension \(d=1\), the nuisance has only two possible
directions; calibration changes bounded constants throughout, and no
unbounded power penalty occurs. For growing \(d\), none of the
uniformity claims is asserted.

The full unknown-coordinate transfer, with its separate
nuisance-invariant lower bound, is proved in Appendix \ref{app:s5}. The
equalizer and uniform asymptotics in this section concern the
fixed-radius experiment alone.

\newpage

\section{Calibration through an unknown contact coordinate}
\label{app:s5}

\subsection{Experiment, target, and precision of the conclusion}
\label{app:s5-1}

Fix an integer dimension \(d\ge1\), a constant \(0<a<1/2\), and a fixed
function \(\psi\in C^2([0,a])\) with

\[
\psi(0)=\psi'(0)=0,\qquad \psi(x)>0\quad(0<x\le a).
\]

The parameter is \(\theta=(x,b,u)\in[0,a]\times\{0,1\}\times S^{d-1}\).
The same coordinate \(x\), binary label \(b\), and unit vector \(u\) are
shared by every observation. Independently conditional on this
parameter, observe

\begin{equation}
\begin{aligned}
(X_i,V_i)&\sim N_{d+1}\big((x,b\psi(x)u),I_{d+1}\big),&&1\le i\le n,\\
U_j&\sim N_d\big(b(1-x)u,I_d\big),&&1\le j\le m,\\
C_l&\sim N_d(u,I_d),&&1\le l\le k.
\end{aligned}\tag{S5.1}
\end{equation}

Here \(n,m,k\) are arbitrary nonnegative integers; absent sources
produce no observations. The number of scalar observations is

\[
(d+1)n+dm+dk.
\]

The target is \((x,b)\), with loss
\((\widehat x-x)^2+(\widehat b-b)^2\). The direction \(u\) is nuisance.
Let \(\mathcal R_d(n,m,k)\) denote minimax target risk. Let
\(\mathcal E_d(n,m,k)\) denote minimax binary classification error,
namely the infimum over tests of their maximum error over all \(x,b,u\).
Randomized tests may be admitted in this definition; the asymptotic
upper construction below is a deterministic function of already counted
observations.

Define the known design function

\begin{equation}
H=H_\psi(n,m)=\min_{x\in[0,a]}\{n\psi(x)^2+m(1-x)^2\}.\tag{S5.2}
\label{eq:s5-2}
\end{equation}

Its minimum exists by continuity and compactness. The calibration count
does not enter \(H\): calibration changes composite testing difficulty
while leaving the least cross-target Euclidean distance unchanged.

For known \(R,k\ge0\), write \(e_d(R,k)\) for the minimax error of

\[
A\sim N_d(\sqrt{k}u,I_d),\qquad B\sim N_d(bRu,I_d),
\]

with common \(u\in S^{d-1}\). The established \ref{app:s4-1} result supplies an
invariant likelihood-ratio equalizer \(\varphi_{R,k}\) whose two errors
equal \(e_d(R,k)\) for every direction. At \(k=0\), the test does not
require an observed calibration vector. At \(R=0\), the error is
\(1/2\).

\begin{theorem}[all-count target risk and sharp logarithmic transfer]
Constants \(0<c<C<\infty\), depending only on the fixed
model \((d,a,\psi)\), satisfy

\begin{equation}
c\left\{\frac1{1+n}+e_d(\sqrt{H_\psi(n,m)},k)\right\}
\le \mathcal R_d(n,m,k)
\le C\left\{\frac1{1+n}+e_d(\sqrt{H_\psi(n,m)},k)\right\}
\tag{S5.3}
\label{eq:s5-3}
\end{equation}

for every nonnegative integer triple \((n,m,k)\). Moreover, for each
fixed \(C_0<\infty\),

\begin{equation}
\sup_{\substack{0\le m\le C_0\log n\\k\ge0}}
\left|\frac{\mathcal E_d(n,m,k)}{e_d(\sqrt{H_\psi(n,m)},k)}-1\right|
\longrightarrow0\qquad(n\to\infty).
\tag{S5.4}
\label{eq:s5-4}
\end{equation}

The supremum in \eqref{eq:s5-4} is over integer counts. It includes zero,
bounded, slowly increasing,
proportional,
and arbitrarily excessive calibration. Constants concern a fixed
function \(\psi\); a common \(C^2\) norm alone is not asserted to make
them uniform over changing functions.

Equation \eqref{eq:s5-3} is a comparison up to fixed positive multiplicative
constants for squared target risk. Equation \eqref{eq:s5-4} is a relative-error
asymptotic for binary classification. It does not give an exact leading
constant for the two-coordinate squared minimax risk.
\end{theorem}

\subsection{Lower bounds and the surviving nuisance-invariant floor}
\label{app:s5-2}

Choose a minimizer \(x_*\) of \eqref{eq:s5-2}. Restrict the parameter to this
fixed coordinate while retaining both binary labels and every direction.
Set

\[
v_*=(\sqrt n\psi(x_*),\sqrt m(1-x_*)).
\]

The vector sample means, after multiplying by count square roots, have
mean \(b(v_{*,1}u,v_{*,2}u)\) and identity covariance on their present
source coordinates. When \(H>0\), orthogonal projection onto
\(v_*/\|v_*\|\) gives a \(d\)-vector with law \(N_d(b\sqrt H u,I_d)\).
Every orthogonal component is independent standard Gaussian and
ancillary. The scalar core sample mean is also ancillary in this
restricted submodel. The calibration mean gives
\(A=\sqrt k\bar C\sim N_d(\sqrt k u,I_d)\) when \(k>0\). Thus the
restricted experiment is exactly the \ref{app:s4-1} experiment plus independent
ancillary observations. At \(H=0\), its two label classes have identical
laws.
Consequently,

\begin{equation}
\mathcal E_d(n,m,k)\ge e_d(\sqrt H,k).\tag{S5.5}
\label{eq:s5-5}
\end{equation}

This restriction and projection are used to prove a lower bound, where
\(x_*\) is fixed and known. They are not proposed as a statistic in the
original unknown-coordinate experiment.

For any real estimator \(\widehat b\), thresholding at \(1/2\) makes an
error only when \((\widehat b-b)^2\ge1/4\). Hence

\begin{equation}
\mathcal R_d(n,m,k)\ge \tfrac14e_d(\sqrt H,k).\tag{S5.6}
\label{eq:s5-6}
\end{equation}

For the other lower bound, fix \(b=0\) and fix one direction \(u_0\).
Vary only \(x\) in a closed interval contained in \((0,a)\). Both vector
detector means vanish, while the calibration mean remains \(u_0\).
Therefore every resolver and calibration observation is exactly
invariant along this curve. Only the \(n\) scalar core observations
depend on \(x\).

More explicitly, choose \(x_\pm=a/2\pm c_0/\sqrt{1+n}\) with fixed
\(0<c_0<a/4\). Their real-Gaussian Kullback--Leibler divergence is

\[
\operatorname{KL}(P_+\|P_-)=\frac n2(x_+-x_-)^2=\frac{2nc_0^2}{1+n}\le2c_0^2.
\]

Their squared target distance is \(4c_0^2/(1+n)\). The elementary
two-point testing lower bound yields

\begin{equation}
\mathcal R_d(n,m,k)\ge \frac{c_1}{1+n},\tag{S5.7}
\label{eq:s5-7}
\end{equation}

uniformly in both other counts. Taking the maximum of \eqref{eq:s5-6} and \eqref{eq:s5-7}
bounds their sum within a factor two. Calibration therefore \textbf{does
not remove the core-budget floor in the full model}. The relevant
invariant curve fixes \(b=0\) and \(u\); no recentering involving an
unknown direction is used.

\subsection{A Gaussian perturbation inequality with relative error control}
\label{app:s5-3}

The following elementary lemma is the statistical stability input. If
\(P=N_q(\mu,I_q)\) and \(Q=N_q(\mu+h,I_q)\), then every measurable event
\(E\) satisfies

\begin{equation}
Q(E)\le\Phi\big(\Phi^{-1}(P(E))+\|h\|\big).\tag{S5.8}
\label{eq:s5-8}
\end{equation}

For \(h\ne0\), the likelihood ratio \(dQ/dP\) is increasing in the
scalar \(h^T(Z-\mu)/\|h\|\). Among all events of a given \(P\)
probability, the upper half-space in that scalar coordinate maximizes
the \(Q\) probability, by the Neyman--Pearson rearrangement argument \cite{NeymanPearson1933}.
Evaluating this half-space gives \eqref{eq:s5-8}. The cases \(h=0\) and
probabilities zero or one follow directly or by limits. This proof
applies to the joint
calibration/detector
Gaussian vector; a perturbation of only the detector mean has the same
norm even when calibration information is arbitrarily large.

The known-direction submodel gives

\begin{equation}
\Phi(-R/2)\le e_d(R,k)\le\tfrac12\quad\text{for every }R,k\ge0.\tag{S5.9}
\label{eq:s5-9}
\end{equation}

Let \(p=e_d(R,k)\) and \(q=-\Phi^{-1}(p)\in[0,R/2]\). The normal hazard
bound

\[
\frac{\phi(t)}{\Phi(-t)}\le C(1+t_+),\qquad t\in\mathbb R,
\]

follows for \(t\ge0\) by integrating the normal density over
\([t,t+(1+t)^{-1}]\), and for \(t<0\) by \(\Phi(-t)\ge1/2\). Integrating
the logarithmic derivative along \(\delta\ge0\) gives

\begin{equation}
\frac{\Phi(-q+\delta)}{\Phi(-q)}\le\exp\{C\delta(1+R)\}.\tag{S5.10}
\end{equation}

Thus a detector-mean perturbation at most \(\varepsilon R\) increases
the error of the nominal equalizer by at most the factor

\begin{equation}
\exp\{C\varepsilon R(1+R)\},\tag{S5.11}
\label{eq:s5-11}
\end{equation}

uniformly in \(k\). This controls both bounded and diverging radii. It
does not invoke monotonicity of the power of a likelihood-ratio test as
its unspecified signal amplitude changes.

\subsection{Localization and the counted Gaussian attenuation construction}
\label{app:s5-4}

Assume temporarily that \(n\ge d+1\) and \(m\ge1\). Let \(\ell=\log n\)
and use the scalar core coordinate to define

\[
\widehat x=\operatorname{clip}(\bar X,[0,a]),\qquad
v(x)=(\sqrt n\psi(x),\sqrt m(1-x)),\qquad
\widehat v=v(\widehat x),\qquad\widehat R=\|\widehat v\|.
\]

The sample mean \(\bar X\) is independent of the vector detector sample
means and of calibration. Define the conditionally linear projection

\begin{equation}
\widehat B=\frac{\widehat v_1\sqrt n\bar V+\widehat v_2\sqrt m\bar U}{\widehat R}.\tag{S5.12}
\label{eq:s5-12}
\end{equation}

Conditional on \(\widehat x\),

\begin{equation}
\widehat B\sim N_d(b\rho u,I_d),\qquad
\rho=\frac{\langle\widehat v,v(x)\rangle}{\widehat R}.\tag{S5.13}
\end{equation}

All denominators are positive because \(m\ge1\) and
\(1-\widehat x\ge1-a\).

We need an independent standard Gaussian vector for attenuation, but no
additional observations are required. From the already counted scalar
core observations take the \(d\) Helmert contrasts

\begin{equation}
W_r=\frac{\sum_{i=1}^rX_i-rX_{r+1}}{\sqrt{r(r+1)}},\qquad 1\le r\le d.\tag{S5.14}
\end{equation}

Each coefficient vector sums to zero and has squared norm one; distinct
vectors are orthogonal. Therefore \(W=(W_1,\ldots,W_d)\sim N_d(0,I_d)\)
and is independent of \(\bar X\), all vector sample means, and
calibration. These are exact facts for the stated Gaussian
observations,
including when \(n=d+1\).

Put \(R_0=\sqrt H\). By definition of the minimum, \(\widehat R\ge R_0\)
\textbf{for every realized} \(\widehat x\in[0,a]\). Consequently the
following operation is defined without a good-event restriction:

\begin{equation}
s=\frac{R_0}{\widehat R}\in[0,1],\qquad
\widetilde B=s\widehat B+\sqrt{1-s^2}\,W.\tag{S5.15}
\end{equation}

Conditional on \(\widehat x\), this is a unit-covariance Gaussian
detector with mean

\begin{equation}
b\widetilde R u,\qquad
\widetilde R=R_0\frac{\langle\widehat v,v(x)\rangle}{\|\widehat v\|^2}.\tag{S5.16}
\label{eq:s5-16}
\end{equation}

It is independent of the calibration sample mean under that
conditioning. The nominal radius is now the fixed least-favourable
radius \(R_0\), rather than the much larger and unknown true radius.
This exact attenuation is the step that makes a uniform transfer
possible even at parameter points far from the contact set.

Every operation in
\eqref{eq:s5-12}--\eqref{eq:s5-16}
is a Borel function of counted observations and known model functions.
In particular, the construction neither assumes an independent nuisance
copy nor creates an uncounted independent observation.

\subsection{Uniform relative-radius accuracy, including flat contact}
\label{app:s5-5}

The fixed \(\psi\) has a nonnegative \(C^{1,1}\) extension to the real
line. To the left, extend by zero, using \(\psi(0)=\psi'(0)=0\); near
the positive endpoint use positivity and a smooth
continuation,
then turn to a positive constant. The extension has Lipschitz derivative
with some finite constant \(K\). As in Glaeser's inequality \cite{Glaeser1963}, nonnegativity and the upper Taylor bound at \(x-\psi'(x)/K\) imply

\[
|\psi'(x)|^2\le2K\psi(x).
\]

Hence \(\sqrt\psi\) is Lipschitz on the interval, and there is a
model-dependent constant such that

\begin{equation}
|\psi(y)-\psi(x)|\le C|y-x|\sqrt{\psi(x)}+C|y-x|^2.\tag{S5.17}
\label{eq:s5-17}
\end{equation}

This estimate remains valid at the zero endpoint and for flat functions;
no positive lower derivative, finite-order contact assumption, or
inverse-function estimate for \(\psi\) is used.

For any fixed \(A>0\), let

\[
\mathcal A_n=\{|\widehat x-x|\le A\sqrt{\ell/n}\}.
\]

Clipping does not enlarge the estimation error, so

\begin{equation}
P_\theta(\mathcal A_n^c)\le2n^{-A^2/2}.\tag{S5.18}
\end{equation}

Writing \(t=\sqrt n\psi(x)\), \eqref{eq:s5-17} gives on this event

\[
|\widehat v_1-v_1|\le Cn^{-1/4}\sqrt\ell\sqrt t+Cn^{-1/2}\ell.
\]

Since \(\|v\|\ge\sqrt{t^2+(1-a)^2}\) for \(m\ge1\), and since

\[
\frac{|\widehat v_2-v_2|}{\|v\|}\le\frac{|\widehat x-x|}{1-a},
\]

we obtain uniformly in \(x\) and every \(m\ge1\)

\begin{equation}
\frac{\|\widehat v-v\|}{\|v\|}\le C\eta_n,\qquad
\eta_n=n^{-1/4}\sqrt\ell+n^{-1/2}\ell,\qquad \eta_n\ell\longrightarrow0.\tag{S5.19}
\end{equation}

For sufficiently large \(n\), \(\|\widehat v\|\ge\|v\|/2\). Equation
\eqref{eq:s5-16} and Cauchy--Schwarz then yield

\begin{equation}
|\widetilde R-R_0|
=R_0\frac{|\langle\widehat v,v-\widehat v\rangle|}{\|\widehat v\|^2}
\le C\eta_nR_0.\tag{S5.20}
\label{eq:s5-20}
\end{equation}

This bound is independent of the true radius \(\|v(x)\|\) and
independent of \(k\). In particular, a true detector radius of order
\(\sqrt n\) causes no large-mean perturbation after attenuation.

\subsection{Proof of sharp composite classification transfer}
\label{app:s5-6}

For \(m\ge1\), apply the established equalizer \(\varphi_{R_0,k}\) to
\((\sqrt k\bar C,\widetilde B)\), omitting the calibration argument when
\(k=0\). Under \(b=0\), its conditional error is \textbf{exactly}
\(e_d(R_0,k)\) for every \(\widehat x\) and \(u\): the distribution of
\(\widetilde B\) is standard normal and the calibration law is
unchanged.

Under \(b=1\), on \(\mathcal A_n\), compare its conditional distribution
with the nominal Gaussian pair at radius \(R_0\). The full joint mean
perturbation has norm at most \(C\eta_nR_0\) by \eqref{eq:s5-20}; the calibration
block contributes zero to that perturbation. Applying
\eqref{eq:s5-8}--\eqref{eq:s5-11}
to the equalizer's type-II event gives

\begin{equation}
P_\theta(\widehat b\ne b)
\le e_d(R_0,k)\exp\{C\eta_nR_0(1+R_0)\}+2n^{-A^2/2}.\tag{S5.21}
\label{eq:s5-21}
\end{equation}

For \(m\le C_0\ell\), we have \(H\le m\le C_0\ell\) by evaluation at
\(x=0\). Therefore the exponential factor in \eqref{eq:s5-21} tends to one
uniformly in this range and in \(k\). Moreover, from \eqref{eq:s5-9} and a
normal-tail lower bound,

\begin{equation}
e_d(R_0,k)\ge\Phi(-\sqrt{C_0\ell}/2)
\ge\frac{c}{1+\sqrt\ell}\,n^{-C_0/8}.\tag{S5.22}
\label{eq:s5-22}
\end{equation}

Choose the fixed event constant \(A\) so that \(A^2/2>C_0/8+2\). The
last term in \eqref{eq:s5-21}, divided by the error in \eqref{eq:s5-22}, then tends
uniformly to zero. The exact lower bound \eqref{eq:s5-5} proves \eqref{eq:s5-4} for
\(m\ge1\).

For \(m=0\), \(H=0\) because \(\psi(0)=0\). At \(x=0\), both detector
sources reveal no label information and calibration depends only on
\(u\), so classification error is at least \(1/2\). The sign of \(W_1\)
implements a fair label decision with error \(1/2\) at every parameter.
Thus \eqref{eq:s5-4} includes \(m=0\) exactly.

The argument handles \(H\to0\), bounded \(H\), and \(H\to\infty\)
without a separate continuity assertion for the unbounded calibration
plane. The calibration allocation does not enter the relative-error
bound.

\subsection{Completion of the all-count target-risk proof}
\label{app:s5-7}

For \(n\) large and \(m\) in any fixed logarithmic range, retain
\(\widehat x\) and the hard decision just constructed. Because clipping
reduces squared error,

\[
E_\theta(\widehat x-x)^2\le n^{-1}.
\]

For a hard binary decision, squared label loss equals classification
error. Equation \eqref{eq:s5-21}, with a sufficiently large fixed \(A\),
therefore proves the upper bound in \eqref{eq:s5-3} in that range. At \(m=0\),
one may simply use \(\widehat b=1/2\) for squared risk.

To cover larger resolver counts, choose a fixed \(C_1\) such that
\((1-a)^2C_1/8>2\). For \(m>C_1\log n\), use only the resolver and
classify by

\[
\widehat b=1\{\|\sqrt m\bar U\|>(1-a)\sqrt m/2\}.
\]

Under the null, an error requires the norm of a standard \(d\)-variate
Gaussian to exceed \((1-a)\sqrt m/2\). Under the alternative, the mean
has norm at least \((1-a)\sqrt m\), so a missed detection requires the
same lower bound on the noise norm by the triangle inequality. Thus both
errors are bounded by

\[
C_d(1+m)^{d/2}\exp\{-(1-a)^2m/8\}=O(n^{-1})
\]

uniformly over this whole range, after increasing the fixed starting
value of \(n\) if necessary. Calibration may be ignored. Together with
\(\widehat x\), the resulting target risk is \(O(n^{-1})\), which is
sufficient for the upper bound in \eqref{eq:s5-3}.

For the finitely many remaining values of \(n\), including \(n=0\), use
the constant estimate \((a/2,1/2)\). Its risk is bounded, while the
right-hand comparison scale is at least \(1/(1+n)\). Increasing the
fixed constant handles all these counts and completes \eqref{eq:s5-3}.

\subsection{Sharp calibration law and resource consequences in the full experiment}
\label{app:s5-8}

The statistical transfer is now proved, so \ref{app:s4-3} may legitimately be
applied to \(R=\sqrt H\). Along any sequence with \(H\to\infty\) and
\(m\le C_0\log n\),

\begin{equation}
\mathcal E_d(n,m,k)\sim
K_dM_d(k)H^{(d-3)/4}e^{-H/8}
\left(1+\frac{2k}{H}\right)^{(d-1)/4},
\tag{S5.23}
\end{equation}

uniformly over all calibration counts. Here

\[
K_d=\frac{2^{2-d}}{\pi^{1/4}\sqrt{\Gamma(d/2)}}
\]

and \(M_d(k)\) denotes the Hellinger affinity of the Gaussian distribution
\(N_d(\sqrt k e_1,I_d)\) and the spherical mixture
\(\int N_d(\sqrt k u,I_d)\,d\mu_d(u)\), with affinity defined as
\(\int\sqrt{pq}\). These are orientation quantities, not new constants
inferred from a pairwise-distance calculation.

The all-information comparison \ref{app:s4-4} and \eqref{eq:s5-3} give an exact criterion
at bounded-logarithm resolution:

\begin{equation}
\begin{gathered}
\mathcal R_d(n,m_n,k_n)=O(n^{-1})\quad\Longleftrightarrow\quad\\
\frac H8+\frac12\log(1+H)
-\frac{d-1}{4}\log\left(1+\frac H{1+k_n}\right)
\ge\log n-O(1).
\end{gathered}\tag{S5.24}
\label{eq:s5-24}
\end{equation}

Here and below \(n\to\infty\) and \(H=H_\psi(n,m_n)\). This covers
arbitrary irregular count sequences. It is not a sharp constant for
squared target loss.

An explicit representative of the information window, uniform over
\textbf{all} calibration counts, is

\begin{equation}
h_d(\ell,k)=8\ell-4\log\ell+2(d-1)\log\left(1+\frac{\ell}{1+k}\right),\qquad \ell=\log n.\tag{S5.24a}
\label{eq:s5-24a}
\end{equation}

The boundary in \eqref{eq:s5-24} is \(H=h_d(\ell,k)+O(1)\), with a remainder
bounded uniformly in \(k\). To verify this without solving an implicit
equation, observe that \(h_d=8\ell+O_d(\log\ell)\) uniformly. Thus
\(\log(1+h_d)=\log\ell+O(1)\) and

\[
\log\left(1+\frac{h_d}{1+k}\right)
=\log\left(1+\frac{\ell}{1+k}\right)+O(1)
\]

uniformly over \(k\ge0\). Substitution into the left side of \eqref{eq:s5-24}
gives \(\ell+O_d(1)\). Its derivative with respect to \(H\) is

\[
\frac18+\frac1{2(1+H)}-\frac{d-1}{4(1+k+H)},
\]

which lies between fixed positive constants for all sufficiently large
\(H\), uniformly in \(k\). Therefore the bounded discrepancy in the
criterion changes the information boundary by only a bounded amount. For
example, when \(k\to\infty\) but \(k=o(\ell)\), \eqref{eq:s5-24a} becomes
\(8\ell+2(d-3)\log\ell-2(d-1)\log k+O(1)\); the unbounded \(\log k\)
correction must not be omitted for slowly growing calibration.

For \(k\asymp(\log n)^\rho\) with fixed \(0\le\rho\le1\), the critical
information window is

\begin{equation}
H=8\log n+\{2(d-1)(1-\rho)-4\}\log\log n+O(1).\tag{S5.25}
\end{equation}

Bounded \(k\), including \(k=0\), corresponds to \(\rho=0\). When
\(k\gtrsim\log n\), the known-direction information window
\(8\log n-4\log\log n+O(1)\) is recovered. The comparability constants
in a power shorthand must remain fixed; \eqref{eq:s5-24} is the authoritative
expression for other sequences.

For \(d\ge2\), relative to the known-direction classification experiment
in the same logarithmic regime, bounded-factor performance is obtained
if and only if \(k=\Omega(H)\). Relative error tending to one is
obtained if and only if \(k/H\to\infty\). These assertions require
\(H\to\infty\) and follow from \ref{app:s4-3} together with \eqref{eq:s5-4}; the
known-direction transfer follows from the same attenuation proof with
its scalar midpoint test. For \(d=1\), the direction nuisance has only
two values, so calibration changes bounded constants and creates no
unbounded power penalty.

The exact deterministic conversion from an information requirement
\(h>0\) to a continuous resolver count is unchanged:

\begin{equation}
M_h(n)=\inf\{m\ge0:H_\psi(n,m)\ge h\}
=\sup_{x\in[0,a]}\frac{h-n\psi(x)^2}{(1-x)^2}.
\tag{S5.26}
\label{eq:s5-26}
\end{equation}

The supremum is positive because \(x=0\) contributes \(h\). Increasing
\(m\) by \(\Delta m\) increases \(H\) by between \((1-a)^2\Delta m\) and
\(\Delta m\), so a bounded information error changes the required
resolver count by only a bounded amount. Integer rounding changes the
count by at most one. Calibration enters the statistical level \(h\)
through \eqref{eq:s5-24}; the common-coordinate geometry enters its conversion
\eqref{eq:s5-26}.

For the flat function \(\psi_\star(x)=\exp[-\tfrac12\exp(1/x^2)]\) on
\(x>0\), extended by zero at the origin, put \(\ell=\log n\) and
\(r(\ell)=(\log\ell)^{-1/2}\). The proved relation
\(H=m(1-r(\ell))^2+o(1)\) on logarithmic resolver ranges now yields

\begin{equation}
m=\frac{8\ell+\{2(d-1)(1-\rho)-4\}\log\ell}{[1-r(\ell)]^2}+O(1),\qquad k\asymp\ell^\rho.\tag{S5.27}
\end{equation}

Thus the calibration-adjusted orientation surcharge for \(0\le\rho<1\)
is still multiplied by the same contact-geometry factor. Adding only
\(2(d-1)(1-\rho)\log\ell\) to the known-direction raw count misses the
unbounded correction \(4(d-1)(1-\rho)\sqrt{\log\ell}+O(1)\) when
\(d\ge2\).

Under the actual scalar budget \(N=(d+1)n+d(m+k)\), regular squared risk
\(O(N^{-1})\) is attainable with \(n\asymp N\) and \(m=O(\log N)\).
Calibration is optional for that rate order. To obtain bounded-factor
known-direction classification one may also take \(k\asymp\log N\); to
make its relative classification penalty vanish one may take
\(\log N\ll k\ll N\), retaining an asymptotically vanishing calibration
fraction. No exact optimal allocation fraction or risk constant is
asserted.

\newpage

\section{Acquisition under a total resource cost}
\label{app:s6}

\subsection{Experiment, cost and attainable order}
\label{app:s6-1}

Fix the functions, parameter box, target and real Gaussian sources of
\ref{app:s2-1}. Write

\[
S(n,m)=\sup_{0\le r\le a}\chi(r)^2
\Phi\{-\sqrt{n\psi(r)^2+m(1-r)^2}\},
\qquad R(n,m)\asymp (1+n)^{-1}+S(n,m).
\]

Here \(n,m\) are nonnegative \textbf{integer block counts}, and the
comparison constants are independent of the counts. Fix positive block
costs \(c_F,c_G\), independent of the total budget \(B\), and define

\begin{equation}
\mathcal A_B=\{(n,m)\in\mathbb N_0^2:c_Fn+c_Gm\le B\},
\qquad R^*(B)=\min_{(n,m)\in\mathcal A_B}R(n,m).
\tag{S6.1}
\end{equation}

The minimum exists because this is a finite allocation set. A choice
\((c_F,c_G)=(5,2)\) counts scalar Gaussian measurements. Unit costs
count blocks. Arbitrary other fixed positive values represent an
explicitly supplied economic objective; they are not implied by the
source dimensions.

\begin{theorem}
For every fixed admissible \ref{app:s2-1} model,

\begin{equation}
R^*(B)\asymp (1+B)^{-1}.
\tag{S6.2}
\label{eq:s6-2}
\end{equation}

Every feasible allocation sequence achieving \(R(n_B,m_B)=O(B^{-1})\)
necessarily has \(n_B=\Omega(B)\). More precisely, a feasible sequence
is order optimal if and only if

\begin{equation}
n_B=\Omega(B),\qquad S(n_B,m_B)=O(B^{-1}).
\tag{S6.3}
\label{eq:s6-3}
\end{equation}

\end{theorem}

\textbf{Proof.} The independent core-coordinate lower bound in \ref{app:s2-1}
gives

\[
R(n,m)\ge c(1+n)^{-1}\ge c(1+B/c_F)^{-1}.
\]

It also proves the necessity of \(n_B=\Omega(B)\). Each summand in the
\ref{app:s2-1} comparison is nonnegative, so the same comparison proves both
directions of \eqref{eq:s6-3}. For
attainability,
let \(\gamma>2/(1-a)^2\), set \(m_B=\lceil\gamma\log(1+B)\rceil\), and,
for sufficiently large \(B\), set

\[
n_B=\left\lfloor\frac{B-c_Gm_B}{c_F}\right\rfloor.
\]

Then \(n_B\sim B/c_F\), and \(0\le\chi\le1\) gives

\[
S(n_B,m_B)\le\Phi\{-(1-a)\sqrt{m_B}\}
\le\tfrac12e^{-(1-a)^2m_B/2}=O(B^{-1}).
\]

All counts are integers and the cost inequality holds exactly. The
finitely many smaller budgets can use zero
observations,
whose bounded target loss is absorbed by a fixed comparison constant.
This proves \eqref{eq:s6-2}. \(\square\)

The logarithmic sufficient allocation above is conservative. It makes no
necessity claim for arbitrary \(\chi\): for example, \(\chi\equiv0\) has
regular core-only risk and requires no resolver.

\subsection{Complete order-optimal characterization for the polynomial--flat family}
\label{app:s6-2}

Now fix \(0<\alpha<1\), and use the explicit functions

\[
\psi(r)=\exp\{-\tfrac12\exp(1/r^2)\},\qquad
\chi(r)=\psi(r)^\alpha\quad(r\ne0),
\]

both defined as zero at the origin. For sufficiently large \(\ell\), put

\begin{equation}
v(\ell)=1-(\log\ell)^{-1/2},\qquad
W_\alpha(\ell)=
\frac{2(1-\alpha)\ell-\log\ell}{v(\ell)^2}.
\tag{S6.4}
\end{equation}

All logarithms are natural. In particular, the radius
\((\log\ell)^{-1/2}\) lies in the fixed interval \([0,a]\) once \(\ell\)
is sufficiently large. Define \(\ell_B=\log(B/c_F)\) and
\(W_B=W_\alpha(\ell_B)\).

\begin{theorem}
For feasible integer allocations with
\(B\to\infty\),

\begin{equation}
R_\alpha(n_B,m_B)=O(B^{-1})
\quad\Longleftrightarrow\quad
n_B=\Omega(B),\qquad m_B\ge W_B-O(1).
\tag{S6.5}
\end{equation}

The expression \(m_B\ge W_B-O(1)\) means that \(W_B-m_B\) is bounded
above along the sequence. It allows arbitrary oversampling. It does not
mean that all optimal allocations have \(m_B=W_B+O(1)\).
\end{theorem}

\textbf{Proof.} First suppose \(n_B=\Omega(B)\). Feasibility gives
\(n_B\le B/c_F\), so \(\log n_B=\ell_B+O(1)\), with bounds depending on
the positive lower allocation fraction. Direct differentiation gives

\[
W_\alpha'(\ell)=
\frac{2(1-\alpha)-\ell^{-1}}{v(\ell)^2}
-\frac{2\{2(1-\alpha)\ell-\log\ell\}v'(\ell)}{v(\ell)^3},
\quad
v'(\ell)=\frac{1}{2\ell(\log\ell)^{3/2}}.
\]

Thus \(W_\alpha'(\ell)\to2(1-\alpha)\), and

\begin{equation}
W_\alpha(\log n_B)=W_B+O(1).
\tag{S6.6}
\label{eq:s6-6}
\end{equation}

For every fixed \(C_0\), \ref{app:s2-3} proves, uniformly over
\(0\le m\le C_0\log n\),

\begin{equation}
R_\alpha(n,m)\asymp n^{-1}
+n^{-\alpha}(1+m)^{-1/2}e^{-m v(\log n)^2/2}.
\tag{S6.7}
\end{equation}

In this range the ambiguity term is \(O(n^{-1})\) exactly when

\[
\tfrac12m v(\log n)^2+\tfrac12\log(1+m)
\ge(1-\alpha)\log n-O(1),
\]

equivalently \(m\ge W_\alpha(\log n)-O(1)\). For this last equivalence,
if \(m\) is in a fixed constant neighborhood of the threshold in units
of \(\log n\), then \(\log(1+m)=\log\log n+O(1)\); below any smaller
positive logarithmic fraction the inequality fails, and above any larger
one it holds. The function of \(m\) on the left is increasing with
derivative bounded below by a positive constant for sufficiently large
\(n\).

To cover all counts, choose a fixed \(C_0>2/(1-a)^2\). If
\(m>C_0\log n\), then the elementary bound used in \ref{app:s6-1} makes
\(S(n,m)=O(n^{-1})\), and \(m\ge W_\alpha(\log n)-O(1)\) holds
automatically,
since \(W_\alpha(\log n)/\log n\to2(1-\alpha)<2<C_0\). Combining the two
count ranges and \eqref{eq:s6-6} proves sufficiency and the resolver necessity
when \(n_B=\Omega(B)\). The latter core condition is necessary by \ref{app:s6-1},
completing the equivalence. \(\square\)

This is a second-stage minimum-resource statement only at bounded-count
resolution: among feasible allocations that attain the regular order, no
sequence can save an unbounded number of resolver blocks below \(W_B\),
whereas a sequence at \(W_B+O(1)\) suffices. There is no universal
smallest bounded additive constant. Its value depends on the specified
risk multiplier and on uncomputed minimax constants.

\subsection{Canonical integer allocation and negligible budget feedback}
\label{app:s6-3}

Fix any real constant \(s\). For all sufficiently large \(B\), choose

\begin{equation}
m_B=\lceil W_B+s\rceil,
\qquad
n_B=\left\lfloor\frac{B-c_Gm_B}{c_F}\right\rfloor.
\tag{S6.8}
\end{equation}

Then

\begin{equation}
m_B=O(\log B),\quad n_B\sim B/c_F,\quad
R_\alpha(n_B,m_B)\asymp B^{-1}.
\tag{S6.9}
\end{equation}

Indeed \(n_B=B/c_F+O(\log B)\), hence

\begin{equation}
\log n_B-\ell_B=O(\log B/B),\qquad
W_\alpha(\log n_B)-W_B=O(\log B/B)=o(1).
\tag{S6.10}
\end{equation}

Thus evaluating the acquisition window at \(B/c_F\) and subsequently
subtracting the resolver cost creates a vanishing change in the
real-valued window. Rounding \(m_B\) changes the cost and resolver count
by bounded amounts; rounding \(n_B\) changes it by less than one block.
The integer construction is therefore valid at the claimed bounded-count
resolution. This proof relies on fixed positive costs. It does not cover
a resolver price growing with \(B\), where a logarithmic resolver
allocation may consume a nonvanishing budget fraction.

At scalar-measurement cost \((c_F,c_G)=(5,2)\), the construction has
\(5n_B+2m_B\le B\). Its scientific interpretation remains an asymptotic
result in a fixed, extremely flat Gaussian model. The radius condition
can demand extraordinarily large budgets; no direct practical
acquisition recommendation follows merely from the asymptotic formula.

\subsection{A proxy optimizer has a different critical level}
\label{app:s6-4}

The minimum resolver count for the regular \textbf{order} must not be
relabelled as the optimizer of the total risk. A transparent
countercheck is supplied by an exactly specified proxy. Fix \(A>0\),
and, for large \(B\), put \(n_0=B/c_F\), \(\ell=\log n_0\),
\(a_B=v(\ell)^2/2\). Define a continuous cost-constrained objective

\begin{equation}
\Pi_B(m)=\frac{c_F}{B-c_Gm}
+A n_0^{-\alpha}(1+m)^{-1/2}e^{-a_Bm},
\qquad 0\le m<B/c_G.
\tag{S6.11}
\label{eq:s6-11}
\end{equation}

The first term is the regular core-risk proxy after its exact economic
cost is paid. The second is the proved tail comparison with the slowly
varying core scale frozen at the affordable core count. This is a
\textbf{defined optimization problem}, not an identity for \(R_\alpha\).

\setcounter{theorem}{3}
\begin{proposition}
The objective \eqref{eq:s6-11} has a unique minimizer
\(m_B^{\rm proxy}\), and

\begin{equation}
m_B^{\rm proxy}=
\frac{2(2-\alpha)\ell-\log\ell}{v(\ell)^2}+O(1).
\tag{S6.12}
\label{eq:s6-12}
\end{equation}

The exact integer minimizer of the same proxy lies at one of the two
adjacent integers and therefore obeys the same formula.
\end{proposition}

\textbf{Proof.} Both terms in \eqref{eq:s6-11} are strictly convex. For the
second, its logarithmic derivative is \(-a_B-1/[2(1+m)]\), and its
second derivative divided by itself is

\[
\left(a_B+\frac1{2(1+m)}\right)^2+
\frac1{2(1+m)^2}>0.
\]

The derivative at zero is negative for sufficiently large \(B\), because
its positive component is of order \(n_0^{-2}\), whereas its negative
component has order \(n_0^{-\alpha}\). The objective diverges at the
other boundary. Its unique minimizer solves

\begin{equation}
\frac{c_Fc_G}{(B-c_Gm)^2}
=A n_0^{-\alpha}(1+m)^{-1/2}e^{-a_Bm}
\left(a_B+\frac1{2(1+m)}\right).
\tag{S6.13}
\label{eq:s6-13}
\end{equation}

Testing \(m=c\ell\) with constants respectively below and above
\((2-\alpha)/a_B\), with a fixed positive margin, brackets the root and
proves \(m=\Theta(\ell)\). Taking logarithms of \eqref{eq:s6-13} now gives

\begin{equation}
a_Bm+\tfrac12\log(1+m)
=(2-\alpha)\ell+
\log\!\left\{\frac{Ac_F}{c_G}
\left(a_B+\frac1{2(1+m)}\right)\right\}
+2\log\left(1-\frac{c_Gm}{B}\right).
\tag{S6.14}
\end{equation}

The last term is \(o(1)\), the preceding logarithm is bounded, and
\(\log(1+m)=\log\ell+O(1)\). This proves \eqref{eq:s6-12}. Strict convexity makes
the integer minimizer one of the neighboring integers. \(\square\)

At this proxy optimum the ambiguity term has order \(B^{-2}\), because
\eqref{eq:s6-13} balances its marginal decrease against a core-cost derivative of
order \(B^{-2}\). At the least allocation that recovers the regular
order, the ambiguity term is merely of order \(B^{-1}\). Thus the proxy
consumes an additional

\[
\frac{2\log(B/c_F)}{v(\log(B/c_F))^2}+O(1)
\]

resolver observations. Both allocations attain risk order \(B^{-1}\).
The fact that they have different leading logarithmic counts decisively
prohibits using order-optimality alone as a proof of exact total-risk
optimization.

Changing the uncalculated proxy multiplier \(A\) shifts the bounded term
in \eqref{eq:s6-12}. More generally, the exact minimax risk of the
flat-coordinate model and its derivatives are not supplied by \ref{app:s2-1}.
Neither \eqref{eq:s6-12} nor the adjacent-integer statement is therefore asserted
for its true minimax optimizer.

\newpage

\section{Uniform repair of a polynomial tangential fibre}
\label{app:s7}

This section gives the complete proof and constructive estimator for
Theorem 5. The target is the ordinary real coordinate \(t\); the
nuisance parameter is shared by all observations. The auxiliary contact
order is denoted \(k_{\mathrm{tan}}\), separately from the calibration
count \(k\) used in the orientation experiment.

\subsection{Experiment, fibres, and the uniform statement}
\label{app:s7-1}

Fix integers \(p\ge2\), \(1\le q<p\), \(k_{\mathrm{tan}}\ge1\), and
constants \(T,E_*>0\). For \(b\in[-2,2]\), define 
\[
z(b)=b^2-1\in[-1,3],\qquad f_j(b)=bz(b)^j\quad(j\ge1).
\]
 On the full box \(K=[-2,2]\times[-T,T]\), independently observe 
\begin{equation}
X_i\sim N_3\big((t-f_q(b),z(b),f_p(b)),I_3\big),\quad 1\le i\le n,
\tag{S7.1}
\end{equation}
 
\begin{equation}
Y_l\sim N\big(\varepsilon f_{k_{\mathrm{tan}}}(b),1\big),\quad 1\le l\le m.
\tag{S7.2}
\end{equation}
 Here \(n,m\in\mathbb N_0\) and the gain \(\varepsilon\in[0,E_*]\) is
known. These are independent real unit-variance Gaussian coordinates.
Counts refer to repeated blocks with one common \((b,t)\); their scalar
cost is \(3n+m\). Write \(w=m\varepsilon^2\) and 
\begin{equation}
\Psi(n,w)=\frac1{1+n}+\min\{(1+n)^{-q/p},(1+w)^{-q/k_{\mathrm{tan}}}\}.
\tag{S7.3}
\label{eq:s7-3}
\end{equation}
 Theorem 5 asserts
\(c\Psi(n,w)\le\mathcal R_\varepsilon(n,m)\le C\Psi(n,w)\), where
\(\mathcal R_\varepsilon\) is the minimax expected squared error in
\(t\) over \(K\). The positive constants depend on the fixed box and
degrees, and are independent of the counts and gain. In particular, gain
zero and a zero auxiliary count are included.

The core mean is immersive on the entire box. Its derivative in \(t\)
isolates the first coordinate. Its derivative in \(b\) has second
coordinate \(2b\), and at \(b=0\) has third coordinate
\(f_p'(0)=(-1)^p\ne0\). Equality of two core means first implies \(c=b\)
or \(c=-b\). A distinct pair must satisfy \(b\ne0\) and
\(b(b^2-1)^p=0\), hence \(b=\pm1\). At these points \(f_q=0\), so the
equal means have equal target coordinates. The only nontrivial fibres
are therefore \(\{(1,t),(-1,t)\}\), \(-T\le t\le T\). Target
identification already holds at zero gain.

For the augmented nuisance curve
\((z,f_p,\varepsilon f_{k_{\mathrm{tan}}})\), the tangent vectors at
\(b=1,-1\) are \((2,0,2\varepsilon)\) and \((-2,0,2\varepsilon)\) when
\(k_{\mathrm{tan}}=1\). Their cross product has squared norm
\(64\varepsilon^2\). For \(k_{\mathrm{tan}}>1\) the last derivatives
vanish and these tangents remain parallel. The theorem treats the
specified polynomial geometry directly; it does not obtain gain-uniform
constants by substituting a varying gain into a fixed-geometry theorem.

\subsection{Lower bounds from actual common-parameter alternatives}
\label{app:s7-2}

For \(z>0\) sufficiently small, take 
\begin{equation}
b_\pm=\pm\sqrt{1+z},\qquad t_\pm=\pm\sqrt{1+z}\,z^q.
\tag{S7.4}
\label{eq:s7-4}
\end{equation}
 The first core means both vanish, and the second both equal \(z\).
The last core means differ by \(2\sqrt{1+z}\,z^p\), and the auxiliary
means differ by \(2\varepsilon\sqrt{1+z}\,z^{k_{\mathrm{tan}}}\).
Consequently, 
\begin{equation}
(t_+-t_-)^2=4(1+z)z^{2q},\qquad
\operatorname{KL}(P_+\Vert P_-)=2(1+z)\{nz^{2p}+wz^{2k_{\mathrm{tan}}}\}.
\tag{S7.5}
\end{equation}
 The Gaussian divergence contains the real-noise factor \(1/2\); no
proper-complex normalization is used here.

Choose a fixed \(0<\gamma<1\) small enough that
\(\sqrt{1+\gamma}\gamma^q<T\) and
\(4(\gamma^{2p}+\gamma^{2k_{\mathrm{tan}}})\le1/8\). Set 
\begin{equation}
z=\gamma\min\{(1+n)^{-1/(2p)},(1+w)^{-1/(2k_{\mathrm{tan}})}\}.
\tag{S7.6}
\end{equation}
 Both alternatives belong to the declared box for every count and
gain, and their divergence is at most \(1/8\). For any estimator,
nearest-target classification and the two-point testing identity give 
\begin{equation}
\max_{\pm}\mathbb E_\pm(\widehat t-t_\pm)^2
\ge\frac{(t_+-t_-)^2}{8}\{1-\operatorname{TV}(P_+,P_-)\}.
\tag{S7.7}
\end{equation}
 Pinsker's inequality \cite[Lemma~2.5]{Tsybakov2009} bounds the total variation away from one.
Substitution proves the lower bound for the minimum in \eqref{eq:s7-3}.

For a separate floor, fix \(b=1\) and use
\(t_\pm=\pm\gamma_0(1+n)^{-1/2}\) with fixed \(0<\gamma_0<T\). The
auxiliary observations have identical laws under these
alternatives,
and the core divergence is at most \(2\gamma_0^2\). Taking \(\gamma_0\)
sufficiently small gives
\(\mathcal R_\varepsilon(n,m)\ge c_0(1+n)^{-1}\) uniformly in \(m\) and
\(\varepsilon\). The maximum of these two lower bounds controls their
sum up to a factor two. Neither argument assigns a different nuisance
value to different sources under one hypothesis.

\subsection{Explicit estimator}
\label{app:s7-3}

At \(n=0\), return zero. Suppose \(n\ge1\), and write the three core
sufficient means as \((\overline X,\overline U,\overline V)\). If
\(m>0\), denote the auxiliary mean by \(\overline Y\). Define 
\[
\widehat z=\operatorname{clip}_{[-1,3]}(\overline U),\qquad Q_*=2\,3^q,
\]
 and set \(\operatorname{sign}(0)=1\). Choose the auxiliary detector
precisely when 
\begin{equation}
w>0\quad\hbox{and}\quad (1+w)^p>(1+n)^{k_{\mathrm{tan}}};
\tag{S7.8}
\label{eq:s7-8}
\end{equation}
 otherwise choose the core detector. This is equivalent to selecting
the smaller rate in \eqref{eq:s7-3}, and ties use the core. The choice depends
only on known quantities.

If \(|\widehat z|>1/2\), set 
\begin{equation}
\widehat Q=\operatorname{clip}_{[-Q_*,Q_*]}
\{\overline V\widehat z^{\,q-p}\}.
\tag{S7.9}
\end{equation}
 The negative power is evaluated only where its argument is bounded
away from zero. If \(|\widehat z|\le1/2\), set 
\[
\widehat s=
\begin{cases}
\operatorname{sign}(\overline V)\operatorname{sign}(\widehat z)^p,&\text{core detector},\\
\operatorname{sign}(\overline Y)\operatorname{sign}(\widehat z)^{k_{\mathrm{tan}}},&\text{auxiliary detector},
\end{cases}
\]
 and 
\begin{equation}
\widehat Q=\widehat s\sqrt{1+\widehat z}\,\widehat z^q.
\tag{S7.10}
\end{equation}
 Finally return 
\begin{equation}
\widehat t=\operatorname{clip}_{[-T,T]}(\overline X+\widehat Q).
\tag{S7.11}
\end{equation}
 All branches and ties are specified, so this is a Borel estimator. It
uses fixed-degree arithmetic, clipping, one sign decision, and one
square root, without nonlinear optimization. This arithmetic description
does not assert optimal bit complexity or automatic accuracy of ordinary
floating point at arbitrary counts.

\subsection{Upper bound}
\label{app:s7-4}

For \(n=0\), bounded loss gives risk at most \(T^2\), while
\(\Psi(0,w)\ge1\). Now take \(n\ge1\), abbreviate \(z=z(b)\), and put
\(\delta=|\widehat z-z|\). Contractivity of clipping gives 
\begin{equation}
\mathbb E\delta^2\le n^{-1},\qquad
\Pr\{\delta>1/4\}\le2e^{-n/32}.
\tag{S7.12}
\end{equation}
 The target estimate and correction are bounded. The event
\(\delta>1/4\) thus contributes at most \(Ce^{-n/32}\le C'/n\).

On the remaining event and in the far region, \(|z|\ge1/4\) and
\(z,\widehat z\) have the same sign. On each relevant compact interval
the function \(u\mapsto u^{q-p}\) is bounded and Lipschitz. Since
\(f_p\) is bounded on the parameter interval, 
\begin{equation}
|\overline V\widehat z^{q-p}-f_q(b)|
\le C\{|\overline V-f_p(b)|+\delta\}.
\tag{S7.13}
\end{equation}
 Clipping cannot increase the error because \(|f_q(b)|\le Q_*\). The
squared expectation in this region is therefore \(O(n^{-1})\).

In the near region with \(\delta\le1/4\), both \(z\) and \(\widehat z\)
lie in \([-3/4,3/4]\), and \(|b|=\sqrt{1+z}\ge1/2\). The function
\(M_q(u)=\sqrt{1+u}\,u^q\) is Lipschitz on that interval. Since
\(f_q(b)=\operatorname{sign}(b)M_q(z)\), 
\begin{equation}
|\widehat Q-f_q(b)|^2
\le C\delta^2+C|z|^{2q}\mathbf1\{\widehat s\ne\operatorname{sign}(b)\}.
\tag{S7.14}
\label{eq:s7-14}
\end{equation}
 If the signs of \(z\) and \(\widehat z\) differ, then
\(|z|\le\delta\); because \(q\ge1\) and the interval is bounded,
\(|z|^{2q}\le C\delta^2\). This includes the zero-contact convention and
does not presume that the sign of the directly estimated contact
coordinate is known.

When the contact-coordinate signs agree, a wrong nuisance sign requires
a wrong sign of the selected Gaussian amplitude mean. Its probability is

\begin{equation}
\Phi(-\sqrt n\,|b|\,|z|^p)
\quad\hbox{or}\quad
\Phi(-\sqrt w\,|b|\,|z|^{k_{\mathrm{tan}}}),
\tag{S7.15}
\end{equation}
 respectively. The direct shape mean is independent of both detector
means. At \(z=0\), the weight \(|z|^{2q}\) vanishes.

For every fixed integer \(j\ge1\) and \(v\ge0\), 
\begin{equation}
\sup_{0\le u\le3/4}u^{2q}\Phi(-\tfrac12\sqrt v\,u^j)
\le C_{j,q}(1+v)^{-q/j}.
\tag{S7.16}
\label{eq:s7-16}
\end{equation}
 For \(v\le1\), this follows from boundedness. For \(v>1\), set
\(a=\sqrt v\,u^j\) and bound the expression by
\(v^{-q/j}\sup_{a\ge0}a^{2q/j}\Phi(-a/2)\), whose supremum is finite.
Equations
\eqref{eq:s7-14}--\eqref{eq:s7-16}
give \(C/n\) plus the selected rate. The deterministic selection in
\eqref{eq:s7-8} makes this rate the minimum in \eqref{eq:s7-3}.

Finally, \(\overline X\) has mean \(t-f_q(b)\) and variance \(1/n\). A
squared triangle inequality and the final clipping give the desired
upper bound; \(1/n\le2/(1+n)\) for \(n\ge1\). Together with \ref{app:s7-2} this
proves Theorem 5.

\subsection{Sufficient numerical accuracy}
\label{app:s7-5}

\begin{proposition}[Order-preserving represented-mean contract]
Fix \(0\le c_0<\infty\). Replace each
primary sufficient mean by any measurable representation with absolute
error at most \(c_0n^{-1/2}\). If the auxiliary mean is used, represent
it with error at most \(c_0m^{-1/2}\). Preserve the known-quantity
detector decision \eqref{eq:s7-8}, and allow correction-arithmetic error at most
\(c_0n^{-1/2}\) before the final target clipping. The upper bound of
Theorem 5 remains valid, with constants also depending on \(c_0\).
\end{proposition}

The represented shape error has second moment at most \(C(c_0)/n\). For
sufficiently large \(n\), exceeding \(1/4\) requires a Gaussian
deviation of a fixed positive size, so its probability is at most
\(Ce^{-cn}\); bounded loss absorbs the remaining finite range of \(n\).
The far-region argument is unchanged. A wrong detector sign after
bounded perturbation is contained in a Gaussian event with standardized
threshold 
\[
c_0-\sqrt v\,|b|\,|z|^j,
\]
 where \((v,j)=(n,p)\) or \((w,k_{\mathrm{tan}})\). The weighted tail
bound \eqref{eq:s7-16} remains true with \(\Phi(c_0-a/2)\) in place of
\(\Phi(-a/2)\). This containment is pointwise, so the representation
errors may depend on the original data; their independence is
unnecessary. The additional correction error contributes \(O(n^{-1})\),
proving the proposition.

This contract preserves a polynomial worst-case risk order. It does not
preserve each rare binary tail in relative error, certify unrestricted
floating-point computation, or allow an unknown error in gain. For a
nonnegative rational represented gain \(\varepsilon=A/D\), decision
\eqref{eq:s7-8} can be made exactly by comparing the integers 
\begin{equation}
(D^2+mA^2)^p\quad\hbox{and}\quad(1+n)^{k_{\mathrm{tan}}}D^{2p}.
\tag{S7.17}
\end{equation}
 A finite binary floating-point gain has an exact rational
representation,
so the same comparison specifies its tie correctly. Numerical
approximation of the true gain itself remains outside the stated
contract.

\subsection{Resource consequences and binary normalization}
\label{app:s7-6}

Since \(q<p\), the core-only order is \((1+n)^{-q/p}\). Along
\(n\to\infty\), the matching theorem implies a strict order improvement
if and only if 
\begin{equation}
\frac{m\varepsilon^2}{n^{k_{\mathrm{tan}}/p}}\longrightarrow\infty.
\tag{S7.18}
\end{equation}
 It implies regular squared target risk \(O(n^{-1})\) if and only if

\begin{equation}
m\varepsilon^2=\Omega(n^{k_{\mathrm{tan}}/q}).
\tag{S7.19}
\end{equation}
 A constant multiple at the first threshold can alter constants
without giving a little-oh improvement. At gain zero, any auxiliary
count supplies only ancillary noise. The primary floor remains even with
unbounded effective auxiliary precision.

For \(m=n\) and \(\varepsilon_n=\varepsilon_0n^{-\beta}\), where
\(0<\varepsilon_0\le E_*\) is fixed and \(\beta\ge0\), 
\begin{equation}
\mathcal R_{\varepsilon_n}(n,n)\asymp n^{-\gamma(\beta)},\qquad
\gamma(\beta)=\min\{1,\max(q/p,q(1-2\beta)/k_{\mathrm{tan}})\}.
\tag{S7.20}
\label{eq:s7-20}
\end{equation}
 The fixed factor \(\varepsilon_0\) keeps this triangular array inside
the declared gain interval, including at \(\beta=0\). Uniformity of
Theorem 5, rather than separate fixed-gain statements, justifies the
substitution.

\begin{figure}
\centering
\includegraphics[width=\linewidth]{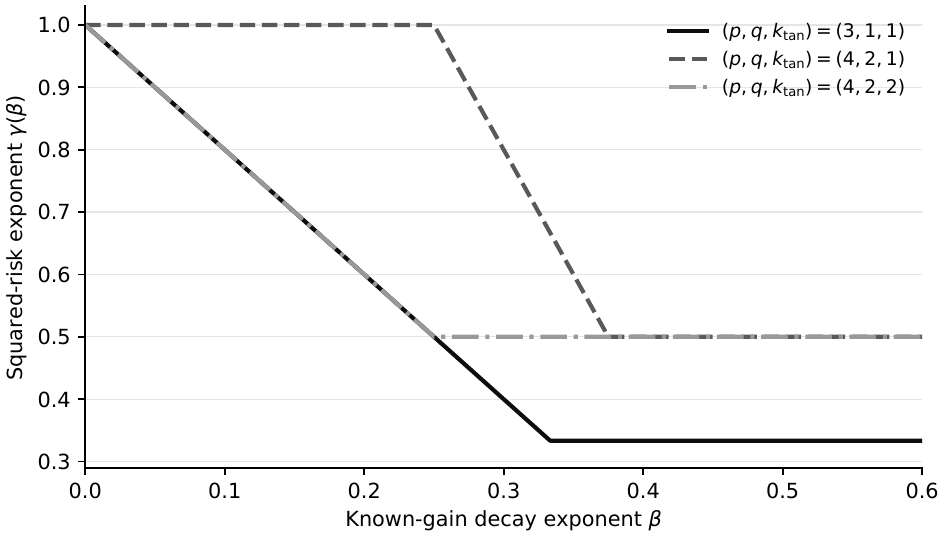}
\caption{Evaluation of the proved exponent in \eqref{eq:s7-20} for three fixed
contact-order triples. The gain is
\(\varepsilon_n=\varepsilon_0n^{-\beta}\) with fixed admissible
\(\varepsilon_0>0\), and \(m=n\). Curves are formula evaluations, not
fitted simulation slopes.}
\end{figure}

For fixed positive gain and scalar cost \(B=3n+m\), the best attainable
squared-risk order is 
\begin{equation}
B^{-\min\{1,\max(q/p,q/k_{\mathrm{tan}})\}}.
\tag{S7.21}
\end{equation}
 The count inequalities \(n\le B/3\), \(m\le B\) give the lower order;
taking both counts proportional to \(B\) attains it. If
\(k_{\mathrm{tan}}<q\), the regular order also follows with
\(n\asymp B\) and \(m\asymp B^{k_{\mathrm{tan}}/q}\), after integer
rounding. If \(k_{\mathrm{tan}}=q\), it requires an auxiliary count of
order \(B\). If \(k_{\mathrm{tan}}>q\), regular \(B^{-1}\) risk is
impossible in this architecture. These are order conclusions, without an
optimal fraction or exact minimax constant.

Finally, put \(z_n=h n^{-1/(2p)}\) in \eqref{eq:s7-4}, where \(h>0\) is fixed,
and assume \(w_n/n^{k_{\mathrm{tan}}/p}\to\lambda\in[0,\infty)\). The
exact Gaussian divergence gives 
\[
\operatorname{KL}(P_{+,n}\Vert P_{-,n})\longrightarrow
2\{h^{2p}+\lambda h^{2k_{\mathrm{tan}}}\},
\]
 
\begin{equation}
p_{\mathrm{binary},n}\longrightarrow
\Phi\big(-\sqrt{h^{2p}+\lambda h^{2k_{\mathrm{tan}}}}\big).
\tag{S7.22}
\end{equation}
 Indeed the squared whitened mean distance is
\(4(1+z_n)(nz_n^{2p}+w_nz_n^{2k_{\mathrm{tan}}})\), and the equal-prior
error for a specified real Gaussian pair is \(\Phi\) evaluated at minus
half the whitened distance, which is the square root of the displayed
squared distance. This is an exact limiting binary calculation, not an
exact composite minimax constant or an estimator limit distribution.

\subsection{Scope}
\label{app:s7-7}

The argument depends on a directly observed shape coordinate, a nuisance
magnitude bounded away from zero in the difficult region, and scalar
sign channels. It is not a theorem for arbitrary failures of nonparallel
branch separation, unknown or sign-indefinite gain, varying degrees,
growing dimension, unknown covariance, correlated fresh noise, nuisance
drift, adaptive acquisition, or an unverified physical forward
reduction. The positive-gain improvement and zero-gain boundary coexist
with target identification and positive local Fisher matrices throughout
the fixed core model. Their difference is controlled by how two
different nuisance branches approach one another in observation space
and by the target's order of vanishing across those branches.

\clearpage
\phantomsection
\addcontentsline{toc}{section}{References}
\begingroup
\small
\setlength{\bibsep}{4pt plus 1pt}
\bibliographystyle{unsrtnat}
\bibliography{references}

@article{Donoho1994,
  author = {Donoho, David L.},
  title = {{Statistical Estimation and Optimal Recovery}},
  journal = {The Annals of Statistics},
  year = {1994},
  volume = {22},
  number = {1},
  pages = {238--270},
  doi = {10.1214/aos/1176325367},
}

@article{PolyanskiyWu2026,
  author = {Polyanskiy, Yury and Wu, Yihong},
  title = {{Dualizing Le Cam's method for functional estimation I: General theory}},
  journal = {The Annals of Statistics},
  year = {2026},
  volume = {54},
  number = {1},
  pages = {1--24},
  doi = {10.1214/25-aos2498},
}

@article{GoldenshlugerJuditskyNemirovski2015,
  author = {Goldenshluger, Alexander and Juditsky, Anatoli and Nemirovski, Arkadi},
  title = {{Hypothesis testing by convex optimization}},
  journal = {Electronic Journal of Statistics},
  year = {2015},
  volume = {9},
  number = {2},
  pages = {1645--1712},
  doi = {10.1214/15-ejs1054},
}

@article{Guigues2020,
  author = {Guigues, Vincent and Juditsky, Anatoli and Nemirovski, Arkadi},
  title = {{Hypothesis testing via Euclidean separation}},
  journal = {Annales de l'Institut Henri Poincar{\'e}, Probabilit{\'e}s et Statistiques},
  year = {2020},
  volume = {56},
  number = {3},
  pages = {1929--1957},
  doi = {10.1214/19-aihp1022},
}

@article{Juditsky2020,
  author = {Juditsky, Anatoli and Nemirovski, Arkadi},
  title = {{Near-optimal recovery of linear and N-convex functions on unions of convex sets}},
  journal = {Information and Inference: A Journal of the IMA},
  year = {2020},
  volume = {9},
  number = {2},
  pages = {423--453},
  doi = {10.1093/imaiai/iaz011},
}

@article{RobinsonGhrist2012,
  author = {Robinson, Michael and Ghrist, Robert},
  title = {{Topological Localization Via Signals of Opportunity}},
  journal = {IEEE Transactions on Signal Processing},
  year = {2012},
  volume = {60},
  number = {5},
  pages = {2362--2373},
  doi = {10.1109/tsp.2012.2187518},
}

@article{Mallat2014,
  author = {Mallat, Achraf and Gezici, Sinan and Dardari, Davide and Vandendorpe, Luc},
  title = {{Statistics of the MLE and Approximate Upper and Lower Bounds--Part II: Threshold Computation and Optimal Pulse Design for TOA Estimation}},
  journal = {IEEE Transactions on Signal Processing},
  year = {2014},
  volume = {62},
  number = {21},
  pages = {5677--5689},
  doi = {10.1109/tsp.2014.2355776},
}

@article{RaySchmidtHieber2016,
  author = {Ray, Kolyan and Schmidt-Hieber, Johannes},
  title = {{Minimax theory for a class of nonlinear statistical inverse problems}},
  journal = {Inverse Problems},
  year = {2016},
  volume = {32},
  number = {6},
  pages = {065003},
  doi = {10.1088/0266-5611/32/6/065003},
}

@article{Sanz2014,
  author = {Sanz, Javier},
  title = {{Flat functions in Carleman ultraholomorphic classes via proximate orders}},
  journal = {Journal of Mathematical Analysis and Applications},
  year = {2014},
  volume = {415},
  number = {2},
  pages = {623--643},
  doi = {10.1016/j.jmaa.2014.01.083},
}

@article{Sagraloff2016,
  author = {Sagraloff, Michael and Mehlhorn, Kurt},
  title = {{Computing real roots of real polynomials}},
  journal = {Journal of Symbolic Computation},
  year = {2016},
  volume = {73},
  pages = {46--86},
  doi = {10.1016/j.jsc.2015.03.004},
}

@article{KlepPovhVolcic2018,
  author = {Klep, Igor and Povh, Janez and Vol{\v c}i{\v c}, Jurij},
  title = {{Minimizer Extraction in Polynomial Optimization Is Robust}},
  journal = {SIAM Journal on Optimization},
  year = {2018},
  volume = {28},
  number = {4},
  pages = {3177--3207},
  doi = {10.1137/17m1152061},
}

@article{Glaeser1963,
  author = {Glaeser, Georges},
  title = {{Racine carr{\'e}e d'une fonction diff{\'e}rentiable}},
  journal = {Annales de l'Institut Fourier},
  year = {1963},
  volume = {13},
  number = {2},
  pages = {203--210},
  doi = {10.5802/aif.146},
}

@article{Lindsey1966,
  author = {Lindsey, W. C.},
  title = {{Phase-Shift-Keyed Signal Detection with Noisy Reference Signals}},
  journal = {IEEE Transactions on Aerospace and Electronic Systems},
  year = {1966},
  volume = {AES-2},
  number = {4},
  pages = {393--401},
  doi = {10.1109/taes.1966.4501788},
}

@article{Cui2015,
  author = {Cui, Guolong and Liu, Jun and Li, Hongbin and Himed, Braham},
  title = {{Signal detection with noisy reference for passive sensing}},
  journal = {Signal Processing},
  year = {2015},
  volume = {108},
  pages = {389--399},
  doi = {10.1016/j.sigpro.2014.09.034},
}

@book{Tsybakov2009,
  author = {Tsybakov, Alexandre B.},
  title = {Introduction to Nonparametric Estimation},
  series = {Springer Series in Statistics},
  publisher = {Springer},
  address = {New York},
  year = {2009},
  doi = {10.1007/b13794},
}

@misc{DLMF,
  author = {{National Institute of Standards and Technology}},
  title = {{NIST Digital Library of Mathematical Functions}},
  year = {2026},
  note = {Version 1.2.7, released June 15, 2026. Equations 10.32.2 and 10.40.1},
  url = {https://dlmf.nist.gov/},
}

@article{NeymanPearson1933,
  author = {Neyman, Jerzy and Pearson, Egon S.},
  title = {On the problem of the most efficient tests of statistical hypotheses},
  journal = {Philosophical Transactions of the Royal Society of London. Series A},
  year = {1933},
  volume = {231},
  number = {694--706},
  pages = {289--337},
  doi = {10.1098/rsta.1933.0009},
}
\endgroup
\end{document}